\documentclass[review,11pt]{article}

\usepackage[utf8]{inputenc}
\usepackage[english]{babel}
\usepackage{placeins}
\usepackage[a4paper, left=1in, right=1in, top=1in, bottom=1in]{geometry}

\usepackage{amsmath,amssymb}
\usepackage{amsthm}
\usepackage{graphicx}
\usepackage{multirow}
\usepackage{enumitem}
\usepackage{float}
\usepackage{color}
\usepackage{xcolor}
\usepackage{subcaption}
\usepackage{adjustbox}
\usepackage{cancel}
\usepackage{ulem}
\usepackage{mathtools}
\usepackage{comment}

\usepackage{aliascnt}
\usepackage{hyperref}
\usepackage{cleveref}

\usepackage{multicol}
\usepackage{booktabs}
\usepackage{titlesec}
\usepackage{algorithm}
\usepackage{algpseudocode}
\usepackage{authblk}
\usepackage{tikz}
\usepackage{stmaryrd}
\usepackage{mathpazo}
\usepackage{bm}
\usepackage{microtype}
\usepackage{setspace}
\usepackage{tcolorbox}
\usepackage{wrapfig}
\tcbuselibrary{breakable}
\usepackage{longtable}

\usepackage{fancyhdr}
\usetikzlibrary{arrows.meta,positioning,shapes.geometric}

\titleclass{\subsubsubsection}{straight}[\subsection]
\newcounter{subsubsubsection}
\titleformat{\subsubsubsection}
  {\normalfont\normalsize\bfseries}
  {\thesubsubsubsection}
  {1em}
  {}
\titlespacing*{\subsubsubsection}
  {0pt}
  {3.25ex plus 1ex minus .2ex}
  {1.5ex plus .2ex}

\renewcommand{\thesubsubsubsection}
  {\thesubsection.\arabic{subsubsection}.\arabic{subsubsubsection}}

\newcounter{AssDD}
\stepcounter{AssDD}

\newcounter{AssM}
\stepcounter{AssM}

\numberwithin{equation}{section}

\title{An adaptive, space-time discretized linear iterative scheme for doubly-degenerate parabolic problems}

\author[1]{A. Javed}
\author[2]{K. Mitra}
\author[1]{I.S. Pop}

\affil[1]{Hasselt University, Belgium}
\affil[2]{Eindhoven University of Technology, The Netherlands}

\date{\today}

\def \a  {\alpha}

\def \g  {\gamma}

\def \d  {\delta}

\def \e  {\varepsilon}

\def \f  {\varphi}

\def \vs  {\omega}

\def \l  {\lambda}
\def \om {\omega}
\def \oma {\omega_a}

\def \Om {\Omega}

\def \t  {\tau}

\def \p  {\partial}
\def \N  {{\mathbb{N}}}
\def \R  {{\mathbb{R}}}
\def \dd {\mathrm{d}}

\def \x {{\bm{x}}}
\def \htt {{h\tau}}

\def \calE {\mathcal{E}}

\def \calV {\mathcal{V}}
\def \calN {\mathcal{N}}
\def \calP {\mathcal{P}}
\def \calQ {\bm{\mathcal{Q}}}
\def \calR {\mathcal{R}}
\def \resH {{\calR}_{\mathrm{H}}}
\def \calS {\mathcal{S}}
\def \calT {\mathcal{T}}
\def \calW {\mathcal{W}}

\def \flx {\bm{\sigma}_{\htt}^i}
\def \flxa {\bm{\sigma}^a_{\htt}}

\def \nhat {\bm{n}_{\mathrm{x}}}
\def \Shti {\calS_\htt^i}
\def \Fhti {\bm{F}_\htt^i}

\def \Maxi {{\mathrm{M}}}

\newcommand{\norm}[1]{{
  \left\vert\kern-0.25ex\left\vert\kern-0.25ex\left\vert #1
  \right\vert\kern-0.25ex\right\vert\kern-0.25ex\right\vert
}}

\newcommand{\snorm}[1]{{
  \left[\kern-0.25ex\left[ #1
  \right]\kern-0.25ex\right]
}}

\newtheorem{lemma}{Lemma}[section]

\newaliascnt{theorem}{lemma}
\newtheorem{theorem}[theorem]{Theorem}
\aliascntresetthe{theorem}

\newaliascnt{corollary}{lemma}

\aliascntresetthe{corollary}

\newaliascnt{proposition}{lemma}
\newtheorem{proposition}[proposition]{Proposition}
\aliascntresetthe{proposition}

\newaliascnt{assumption}{lemma}

\aliascntresetthe{assumption}

\newaliascnt{definition}{lemma}
\newtheorem{definition}[definition]{Definition}
\aliascntresetthe{definition}

\newaliascnt{remark}{lemma}
\newtheorem{remark}[remark]{Remark}
\aliascntresetthe{remark}

\newaliascnt{example}{lemma}

\aliascntresetthe{example}

\newaliascnt{conjecture}{lemma}

\aliascntresetthe{conjecture}

\newtheorem{problem}{Problem}

\crefname{lemma}{Lemma}{Lemmas}
\Crefname{lemma}{Lemma}{Lemmas}

\crefname{theorem}{Theorem}{Theorems}
\Crefname{theorem}{Theorem}{Theorems}

\crefname{corollary}{Corollary}{Corollaries}
\Crefname{corollary}{Corollary}{Corollaries}

\crefname{proposition}{Proposition}{Propositions}
\Crefname{proposition}{Proposition}{Propositions}

\crefname{assumption}{Assumption}{Assumptions}
\Crefname{assumption}{Assumption}{Assumptions}

\crefname{definition}{Definition}{Definitions}
\Crefname{definition}{Definition}{Definitions}

\crefname{remark}{Remark}{Remarks}
\Crefname{remark}{Remark}{Remarks}

\crefname{example}{Example}{Examples}
\Crefname{example}{Example}{Examples}

\crefname{conjecture}{Conjecture}{Conjectures}
\Crefname{conjecture}{Conjecture}{Conjectures}

\crefname{problem}{Problem}{Problems}
\Crefname{problem}{Problem}{Problems}

\crefname{subsubsubsection}{section}{sections}
\Crefname{subsubsubsection}{Section}{Sections}

\newtcolorbox{mybox}[1]{
    breakable,
    colback=yellow!10!white,
    colframe=red!75!black,
    fonttitle=\bfseries,
    title={#1}
}

\newcounter{Ass}
\stepcounter{Ass}

\begin{document}

\maketitle
\begin{abstract}
Degenerate diffusion problems, where the governing parabolic equation can change type to either an ordinary differential equation or an elliptic equation, model many real life applications. Due to the presence of free-boundaries, accurate numerical simulation of such problems require extremely small mesh and time step sizes locally. To remediate this issue, in this work, we consider a space-time formulation of the problem based on an efficient splitting of the nonlinearities.  First, an iterative linearization scheme is proposed to resolve the nonlinearities that effectively reduces to solving a sequence of heat equations. Unconditional convergence of the scheme is proven even for double degenerate cases with linear convergence achieved if the problem is non-degenerate. Next,  the dual norm of the nonlinear residual is decomposed into a linearization error component and a discretization error component corresponding to the heat equation. This leads to reliable and fully computable a posteriori estimates for the problem that are robust with respect to the nonlinearities/degeneracies. These estimates are used then in a fully adaptive (discretization + linearization) space-time solver. Numerical experiments for multiple test cases (one and two dimensions in space) demonstrate that this solver efficiently allocates the computational resources in the space-time domain, resulting in a rapid decay of error in terms of total degrees of freedom spent.

\end{abstract}
\section{Introduction} \label{introduction}
Degenerate parabolic equations arise in a wide range of applications, such as flow through porous media \cite{vazquez2007porous}, mathematical biology \cite{van2002mathematical}, phase transition \cite{Malgo}, traffic flow \cite{berthelin2008model}, and other nonlinear diffusion processes. 
The solution to such models features free boundaries (where the equation becomes an ordinary differential equation(ODE)) with sharp and moving interfaces, causing the solution to have low space regularity, along with degenerate regions (where the equation becomes elliptic) with low time regularity. These two "degeneracies" make it challenging to numerically approximate the solutions. They also render uniform discretizations (particularly, in time) inefficient since the sharp interfaces often require very fine meshes.

Let $\Om$ be a bounded open domain in $\R^{d}$ having a Lipschitz boundary $\p\Om$. For some final time $T>0$, let $Q:=(0,T]\times\Omega$ and $\Gamma:=(0,T]\times \p\Om$ be the space-time domain and interface respectively. Then the degenerate diffusion problem reads: find $u,\,w:Q\to \R$ satisfying
\begin{equation}\label{eq:parabolic_equation}
\left\{ \begin{array}{rcll}
\partial_t u &=& \Delta w + f, & \text{in } Q, \\
w &\in& \Phi(u), & \text{in } Q, \\
w &=& 0, & \text{on } \Gamma, \\
u(0) &=& u_0, & \text{in } \Omega,
\end{array}
\right.
\end{equation}
Here, $\Phi:[0,\omega)\to \R$ denotes a nonlinear, monotonically increasing mapping for either $\omega=1$ or $\infty$. In particular, $\Phi' $ might vanish at $0$, making the problem an ODE (see Figure~\ref{fig:RichVsBio_phi}). On the other hand, $\Phi'$ might blow up at $\om=1$ (for instance, for biofilm growth \cite{van2002mathematical} and traffic flow models \cite{berthelin2008model}) or at infinity, as in the porous medium equation (PME) case. It might even become multivalued at $\om=1$, like in the Richards equation case \cite{mitra2024posteriori}. These scenarios set off the elliptic degeneracy.

\begin{figure}[h]
    \centering
    \begin{subfigure}[b]{0.32\textwidth}
        \centering
    \includegraphics[width=\textwidth, height=4.1cm]{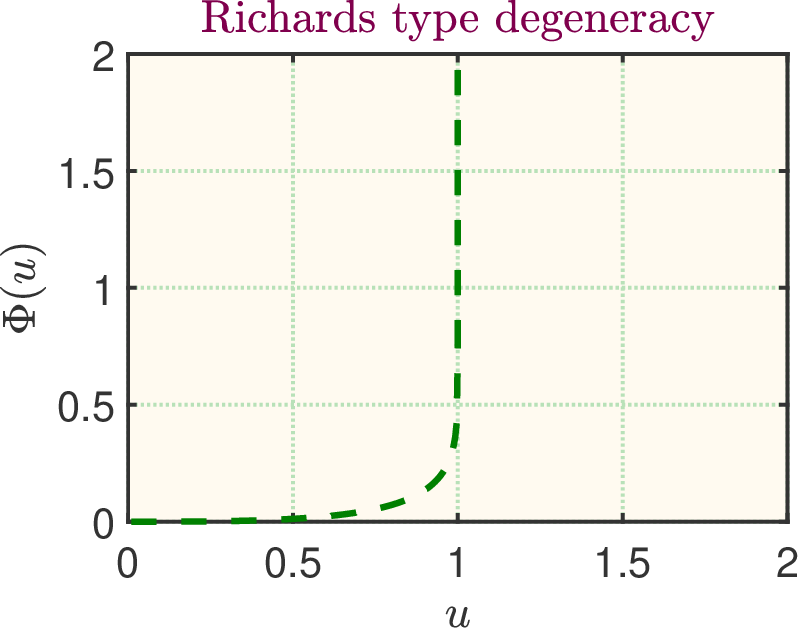}
        \caption*{$\Phi$ becomes multivalued at $u = 1$}
    \end{subfigure}
    \begin{subfigure}[b]{0.32\textwidth}
        \centering
        \includegraphics[width=\textwidth, height=4.1cm]{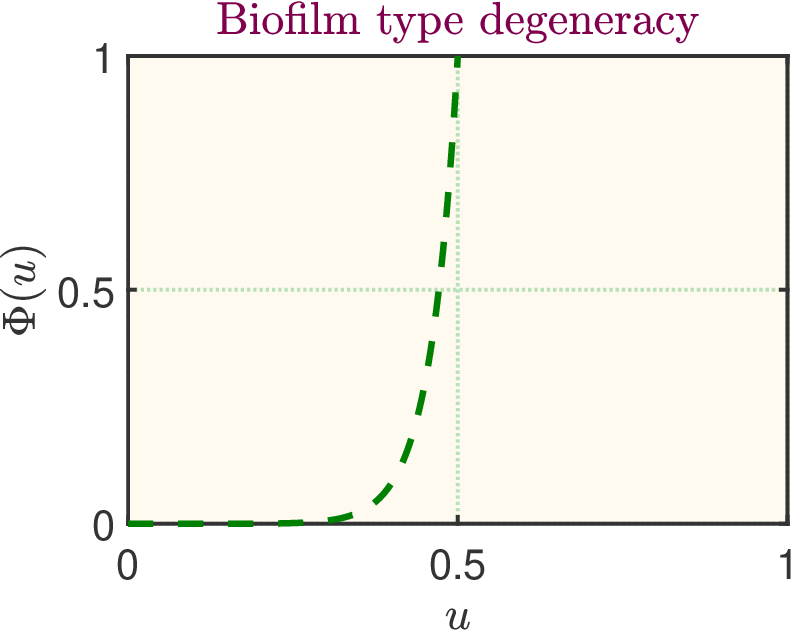}
        \caption*{$\Phi$ becomes infinite at $u = 1$}
    \end{subfigure}
      \begin{subfigure}[b]{0.32\textwidth}
        \centering
        \includegraphics[width=\textwidth, height=4.1cm]{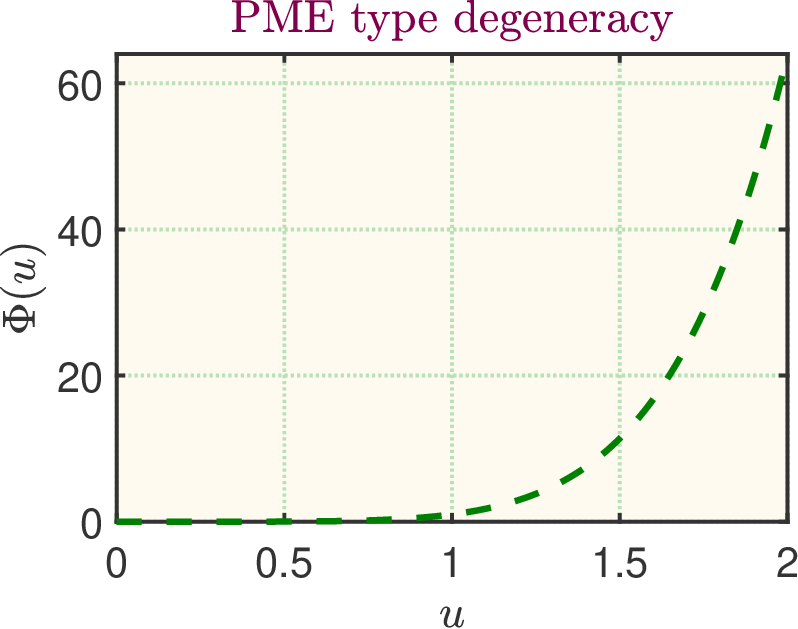}
        \caption*{degeneracy occurs at $u = 0$}
    \end{subfigure}
    \caption{Examples of degenerate nonlinearities $\Phi$: (\textbf{Left}) Richards-type \cite{mitra2024posteriori}, (\textbf{Center}) biofilm-type \cite{van2002mathematical,berthelin2008model}, and (\textbf{Right}) porous media equation-type \cite{vazquez2007porous} degeneracies.}
\label{fig:RichVsBio_phi}
\end{figure}
To deal with the double degeneracy discussed above, we use the framework discussed in~\cite{brenner2017improving} and also see in~\cite{javed2025robust}. The system is reformulated in terms of a new variable $s$, and two increasing functions $b, B \in C^1(\mathbb{R})$ that satisfy
\begin{align}\label{eq:condition}
    \Phi =B \circ b^{-1}, \text{ and }\; 0\leq b',B'\leq 1, \text{ and }  b'+B'\geq 1.
\end{align}
With this choice, if $u = b(s)$ and $w=B(s)$, then one immediately gets $w \in \Phi(u)$.  In this way, problem \eqref{eq:parabolic_equation} becomes of finding $s:Q\to \R$ satisfying
\begin{equation}\label{eq:reformulation}
\left\{ \begin{array}{rcll}
\partial_t b(s) &= &\Delta B(s) + f, \qquad & \text{ in } Q, \\
B(s) &=& 0, & \text{ on } \Gamma, \\
b(s(0)) &=& u_0, &\text{ in } \Omega. \\
\end{array}
\right.
\end{equation}
The advantage of this reformulation is that the functions $b$ and $B$ are Lipschitz continuous on $\mathbb{R}$, with Lipschitz constants equal to~1.  In this setting, the two types of degeneracy arise when either $b' \searrow 0$ (corresponding to the fast diffusion (elliptic) regime in the original formulation) or $B' \searrow 0$ (corresponding to the slow diffusion (ODE) regime).
Such a decomposition is possible under relatively weak assumptions, and an explicit construction of the $b,\, B$ functions is presented in \Cref{lemma:b_B_construction}, see also \cite{javed2025robust}. The graphs of these functions are illustrated in \Cref{bB-graphs} for the profiles of \Cref{fig:RichVsBio_phi}.

\begin{figure}[h!]
    \centering
    \begin{subfigure}{0.32\textwidth}
        \centering
        \includegraphics[width=\linewidth]{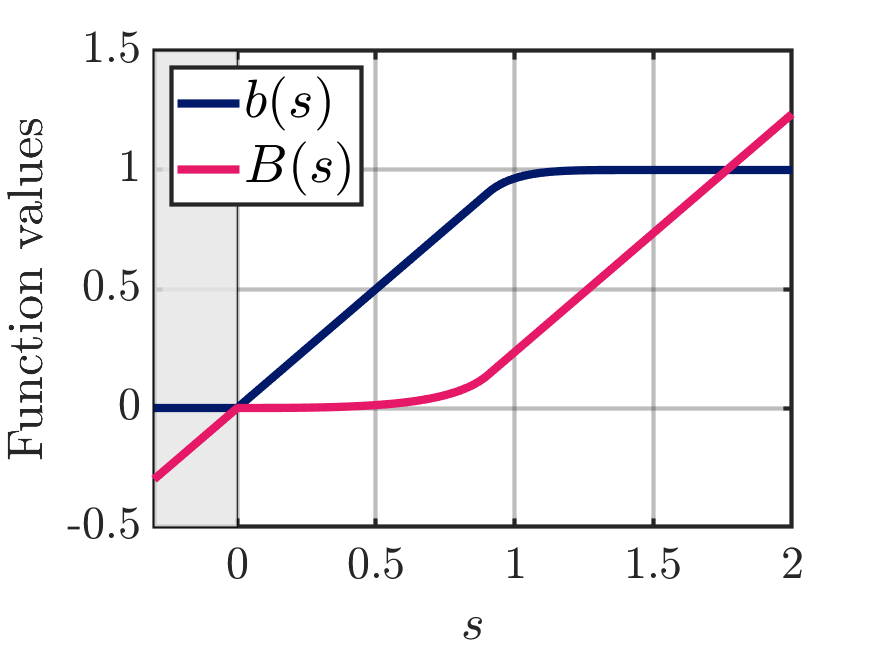}
    \end{subfigure}
    \hfill
    \begin{subfigure}{0.32\textwidth}
        \centering       \includegraphics[width=\linewidth]{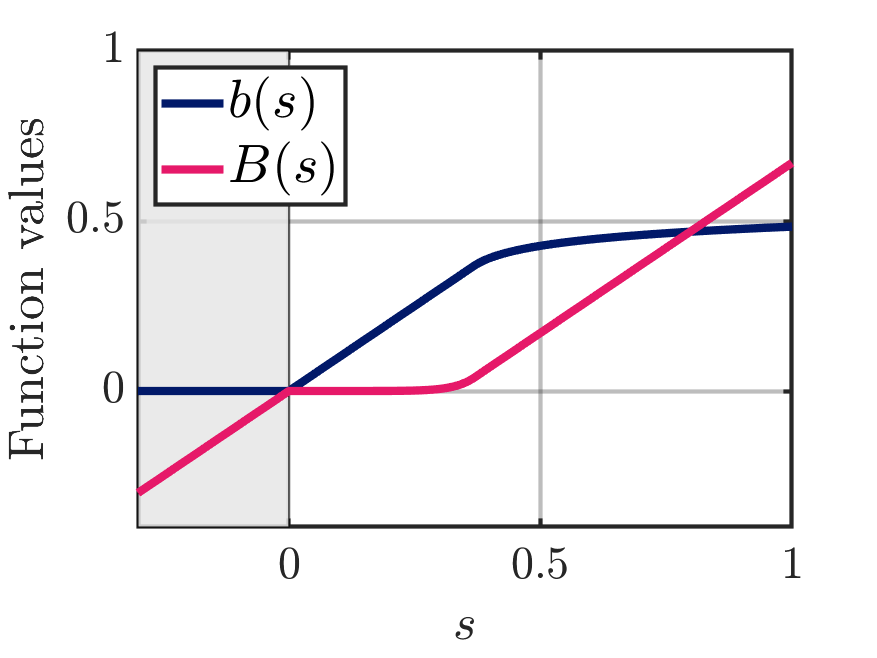}
    \end{subfigure}
     \hfill
    \begin{subfigure}{0.32\textwidth}
        \centering       \includegraphics[width=\linewidth]{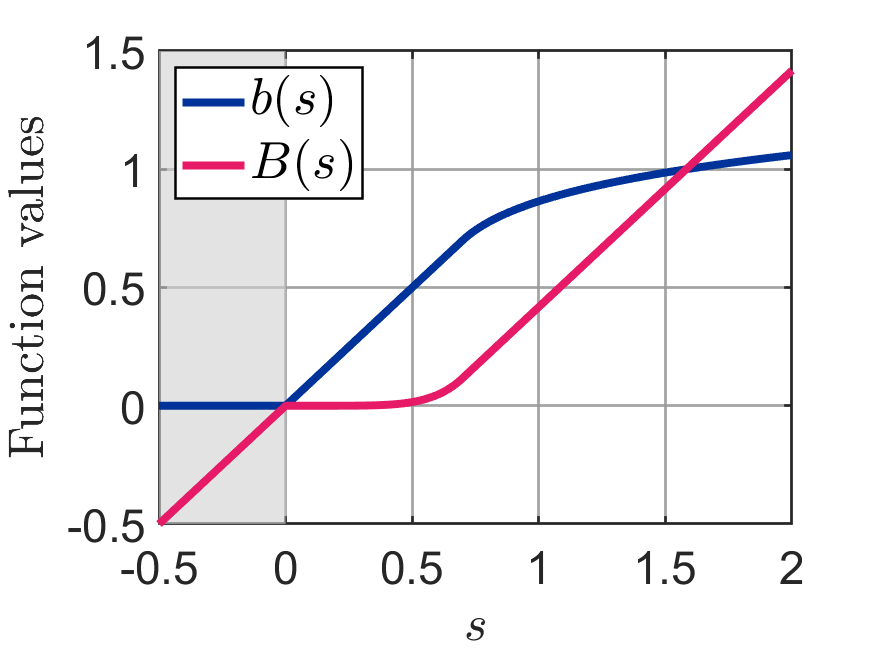}
    \end{subfigure}
\caption{The functions $b$ and $B$ for the mappings $\Phi$ appearing in different applications shown in \Cref{fig:RichVsBio_phi}, including their extensions to $s\le 0$. (Left) Richards case, (Center) biofilm case and (Right) porous medium equation case (PME).}
\label{bB-graphs}
\end{figure}

\noindent
Traditionally, the numerical treatment of such problems relies on semi-discretization techniques: either discretizing first in space \cite{MR2249024} (the method of lines), leading to a system of nonlinear ODEs in time, or discretizing first in time \cite{kavcur2006method} (Rothe’s method), solving a sequence of elliptic problems at successive time steps (time-stepping). However, features like free boundaries and saturated regions are typically confined to limited regions of the space-time domain and cannot be efficiently resolved by adaptive time-stepping or space refinements. Since their location is not known a priori, a posteriori indicators need to be considered. These considerations motivate variational space–time discretization \cite{neumuller2013space} and iterative linearization methods combined with adaptive refinements prompted by rigorous a posteriori error estimation techniques.


\textbf{Space-time discretization:} In recent years, there has been a rapidly growing interest in simultaneous space-time methods across a range of problems, including 
linear parabolic evolution equations. Various formulations and strategies have been proposed and analyzed, including conforming methods on general meshes \cite{steinbach2015space}, refinement strategies for flexible space-time meshes \cite{neumuller2011refinement}, simplex space-time meshes enabling local temporal refinement and unstructured meshing in both space and time 
\cite{Simplex_sp}, sparse Petrov-Galerkin discretizations \cite{Andreev2013}, space-time formulations in isogeometric analysis \cite{LANGER2016342, AntoniettiDede2025}, parallel multigrid algorithms with provable scalability \cite{GanderNeumuller2016}, space-time $hp$-approximations achieving exponential convergence under analytic regularity \cite{devaud2018space}, optimal adaptive wavelet solvers of linear complexity \cite{rekatsinas2019optimal}, coercive formulations yielding symmetric positive definite systems \cite{SteinbachZank2020}, and stable mixed variational formulations \cite{StevensonWesterdiep2021}. Space-time domain decomposition methods with non-matching spatial grids and asynchronous local time stepping have been developed for parabolic problems with solution features localized in space and time \cite{jayadharan:hal-03355088}. Closer to the nonlinear setting, \cite{beranek2025quasioptimal} reformulate parabolic equations with Lipschitz continuous, strongly monotone spatial operators by introducing an auxiliary variable, obtaining an equivalent system with a Lipschitz continuous operator and inverse, and establish quasi-optimality of the resulting space-time Galerkin discretizations. The discontinuous Galerkin method has been studied in \cite{Jan2012}, see also \cite{congreve2026residual,dolejvsi2024error,DOLEJSI2019276,DolejsiRoskovec2021}. Without these exceptions, methods are predominantly focused on \emph{linear} parabolic problems, and their extension to \emph{degenerate} parabolic problems in particular remains largely unexplored. Using the iterative space-time linearization introduced below, \emph{we reduce solving the reformulated nonlinear system \eqref{eq:reformulation} to solving a sequence of heat equations}. These are then discretized by space-time conforming finite element method \cite{steinbach2015space}, which is well-posed and consistent for all iterations.

\textbf{Space-time linearization:} A central challenge in the numerical solution of nonlinear degenerate parabolic equations is the design of robust and efficient linearization schemes. \emph{To the best of our knowledge, no linearization scheme with provable convergence has been developed for double degenerate problems within a fully global space-time variational framework}. For nonlinear elliptic problems, particularly those arising from time-discretizing parabolic equations, Newton's method is the classical choice \cite{bergamaschi1999mixed,lehmann1998comparison} due to its quadratic convergence. However, its convergence is guaranteed only when the initial guess is sufficiently close to the solution, which in the degenerate case translates into severe restrictions on the time step size or regularization \cite{radu2006newton,javed2025robust,fevotte2024adaptive}. To overcome this limitation, fixed-point iterative schemes have been proposed as robust alternatives. Among these, the L-scheme provides unconditional convergence regardless of the time step size, spatial discretization, or degree of degeneracy, albeit with a slow, linear convergence rate \cite{pop2004mixed, list2016study}. Several attempts have been made to reconcile with the trade-off between the speed of the Newton scheme and the stability of the L-scheme, including a modified Newton method \cite{mitra2019modified}, a modified Picard method \cite{celia1990general}, nested Newton iterations \cite{casulli2010nested},  L-scheme with fixed preconditioners \cite{kohle2025robust}, and an adaptive L/Newton scheme \cite{stokke2023adaptive}, to mention a few. For the Richards equation specifically, 
\cite{brenner2017improving} proposed the $b-B$ parametrization strategy that significantly improves the performance of Newton's method. Building on this, \cite{javed2025robust} extended these ideas to the L/M-schemes, yielding guaranteed convergence in double degenerate cases along with asymptotic quadratic convergence for the adaptive M-version. A related approach for singular degenerate evolution systems modelling biofilm growth was developed in \cite{SmeetsEtAl2024}.

However, all these results are restricted to the sequential time-stepping framework.  Here, we fill this gap by adopting the splitting L-scheme of \cite{javed2025robust} to space-time formulations. For an iterate $s^i$ and a constant $L>1$, the linearization finds the next iterate $s^{i+1}$ by solving a heat equation:
\begin{equation}\label{eq:Lsch-strong}
\partial_t(s^{i+1}-s^i) - \Delta(s^{i+1}- s^i) 
= -  L^{-1}[\partial_t b(s^i)-\Delta B(s^i) - f].
\end{equation}
We will show that \emph{this iterative method is unconditionally convergent even for double degenerate cases, the convergence being linear if the problem is non-degenerate}. The simple heat equation structure is used to construct a fully adaptive solver further on.

\textbf{A posteriori error estimation 
and adaptivity:} 
The global space-time approach, while powerful, comes with the high computational cost of solving and storing the entire 
space-time profile, all at once. However, this can be made computationally affordable by adaptive discretization methods, which enable local time-step refinements, unlike time-stepping. This goal is 
achieved for linear parabolic problems by space-time adaptive wavelet methods at optimal complexity \cite{schwab2009space}, but whose extension to nonlinear degenerate problems is non-trivial. Achieving adaptivity relies crucially on a posteriori error estimation. Guaranteed, locally space-time efficient, and polynomial degree robust estimates
were established for linear parabolic problems using equilibrated flux reconstructions in
\cite{ErnSmearsVohralik2017, ErnSmearsVohralik2019}, however, for the time-stepping case.  A posteriori estimators for linear and semilinear parabolic problems have also been developed in 
\cite{eriksson1991adaptive, Eriksson1995, 
schmich2008adaptivity, FUHRER202127, 
RattiVerani2019, GeorgoulisLakkisWihler2021}. 
Space-time adaptivity has been successfully applied to parabolic optimization and 
control problems in \cite{Meidner2012, gong2012space, LangerSchafelner+2022+247+266, GunzburgerKunoth2011}. 

 For nonlinear problems, a posteriori estimates driving adaptive refinement in both space and time were developed within sequential time-stepping frameworks in \cite{BernardiElAlaouiMghazli2014, 
CLEMENT2021103897}, with \cite{Dyja2018} investigating parallel-in-time framework and \cite{BARON2017104} additionally 
separating linearization errors. For the degenerate Richards equation,  in \cite{mitra2024posteriori}, guaranteed and locally space-time efficient estimators were derived that are robust with respect to the final time, and separate every error component. However, it too required sequential time-stepping and involved complex weighted norms.
Using discontinuous Galerkin (DG) methods, an adaptive strategy was proposed in \cite{DOLEJSI2019276}; reliable and efficient estimators were derived in \cite{DolejsiRoskovec2021,dolejvsi2024error}, however, in mesh and solution dependent norms; and in \cite{congreve2026residual} for mesh and nonlinearity dependent constants.

In this work, we draw heavy influence from \cite{mitra2023guaranteed}, where it was shown that for general elliptic problems, the total error, measured in an iteration-dependent energy dual-norm, can be orthogonally decomposed to a fully-computable linearization component and a discretization component arising solely from the linearization step. This produced fully-computable locally-efficient estimates, robust with respect to nonlinearities, paving the way for later works such as \cite{harnist2023robust,stokke2023adaptive,javed2025robust}. The estimates are even on a fixed norm for the L-scheme. Using the same principle here for the L-scheme, yields the following decomposition in \Cref{subsec:decomposition}
\[
\underbrace{\text{Total Error}}_{\substack{\text{measured by a fixed }\\\text{dual energy-norm of residual}}}\approx \underbrace{\left(\text{Linearization error}\right)}_{\substack{
\text{measured by the energy norm}\\
\text{of consecutive iterates}}} + \underbrace{\left(\text{Discretization error of the heat equation}\right)}_{\substack{
\text{measured by the dual energy norm }\\
\text{of the residual of the heat equation}}}.
\]
The discretization error due to the heat equation can potentially be estimated locally efficiently in space-time (see \cite{ErnSmearsVohralik2017} for the time-stepping case). Thus, \emph{so can the total error, and, in fact, robustly with respect to nonlinearities/degeneracies}. We propose such a reliable estimator that is fully-computable in parallel, and easily implementable in \Cref{sec:a-posteriori}. Using this estimator, a fully adaptive discretization/linearization scheme is proposed, which runs L-scheme iterations until linearization error is below a fraction of the discretization error, and then performs a refinement step based on the a posteriori estimator. Numerical experiments demonstrate the efficacy of this adaptive strategy.

The paper is organized as follows. 
\Cref{sec:setting} sets up the functional framework, states the main assumptions, and derives the weak and residual formulations. \Cref{Linearization} introduces the space-time $L$-scheme, establishes its rigorous convergence, and illustrates the convergence behaviour numerically. \Cref{subsec:decomposition} establishes the quasi-orthogonal decomposition of the total error into linearization and discretization components. \Cref{sec:apost} builds on this decomposition to derive fully-computable guaranteed a posteriori error bounds and proposes the adaptive refinement algorithm driven by local error indicators. \Cref{sec:Numerical_results} contains numerical experiments to validate the performance of the adaptive refinement strategy. Our findings are summarized in \Cref{sec:conclusion}.

\section{Mathematical formulation}\label{sec:setting}
In this section, we introduce the mathematical framework used throughout the paper. We start from the nonlinear degenerate parabolic model presented in \Cref{introduction}, followed by a reparametrization strategy. We then present the functional framework, weak formulation and the linear iterative scheme.

\subsection{Reparametrization}\label{subsec:reparametrization}
In view of the modeled processes mentioned before, one can see $u$ as a concentration or saturation. Then it is natural to think of an upper bound $\omega$ of $u$. Here, we either take $\omega = 1$ if $u$ is, for example, a concentration or saturation, or $\omega = \infty$ for unbounded solutions. We assume
\begin{enumerate}[label=(A.1\alph*)]
\item \label{ass:A1a}
The function $\Phi:(-\infty,\vs) \to [0,\infty)$ is continuous, differentiable, and strictly increasing in $(0,\vs)$. 
Moreover, $\Phi(s)=0$ for $s\leq 0$, and either $\lim_{s\searrow 0}\Phi'(0)>0$, or $\Phi$ is strictly convex in a right neighbourhood of 0. Further, the limit $\Phi_\Maxi:=\lim_{u\nearrow \vs}\Phi$ is either infinite, or, if $\Phi_\Maxi<\infty$, then we extend $\Phi$ to the set $[\Phi_\Maxi,\infty)$ at $u=\vs$.  
\end{enumerate}

Assumption \ref{ass:A1a} is very general, and covers all the nonlinearities described in \Cref{fig:RichVsBio_phi}.

\begin{enumerate}[label=(A.2\alph*)]
\item \label{ass:A1b}
There exist functions $b, B$  that are Lipschitz continuous, non-decreasing, and satisfy
\begin{equation}\label{eq:conditionbB1}
\left\{
\begin{aligned}
&B(s) = (\Phi \circ b)(s), \quad s \ge 0, \quad b(0)=0,\\
&|B(s) - B(r)| \le |s - r|, \quad |b(s) - b(r)| \le |s - r|, \forall\, s,r \in \mathbb{R},\\
&(b+B)(s) - (b+B)(r) \ge s - r, \quad \forall\, s \ge r .
\end{aligned}
\right.
\end{equation}
\end{enumerate}

\noindent
Note that by introducing the reparametrizations $b$ and $B$, the double degeneracy in \eqref{eq:parabolic_equation} due to $\Phi$ can be transformed into two degeneracies in \eqref{eq:reformulation}, involving two Lipschitz continuous functions satisfying \eqref{eq:conditionbB1}. For any convex $\Phi$, these functions can be found explicitly, as follows from  \cite{javed2025robust}, Lemma~2.5. This is summarized below.


\begin{lemma}[Construction of the $b$--$B$ decomposition]\label{lemma:b_B_construction}Let 
$\Phi$ have locally Lipschitz continuous derivatives in $(0,\vs)$. Assume that there exists $u^*\in (0,\vs)$ such that (after a possible rescaling of $\Phi$),\label{ass:Bphi}
\begin{equation*}
    \Phi'(u)\begin{cases}
        \leq 1 &\text{ for } u\in (0,u^*),\\
        =1&\text{ for } u=u^*,\\
        \geq 1 &\text{ for } u\in (u^*,\vs).
    \end{cases}
\end{equation*}    
Then, the functions $b, B : \mathbb{R} \to \mathbb{R}$,
\begin{subequations}\label{eq:bBexpression}
    \begin{align}
            &b(s):=\int_0^s \min\left\{1, \frac{1}{\Phi'(b(\rho))}\right\}d\rho=\begin{cases}
             s &\text{ if } s\leq u^*,\\
             \Phi^{-1}(\Phi(u^*) + s-u^*) &\text{ if } s\geq u^*.
         \end{cases}\\
         &B(s):=\int_0^s \min \left\{1, \Phi'(b(\rho))\right\}d\rho=\begin{cases}
             \Phi(s) &\text{ if } s\leq u^*,\\
             \Phi(u^*) + s-u^*&\text{ if } s\geq u^*,
         \end{cases}
    \end{align}   
\end{subequations}
satisfy Assumption \ref{ass:A1b}.

\end{lemma}

\noindent
By an abuse of notation, we take $\min\left\{1, \frac{1}{\Phi'(z)}\right\} = 1$ if $\Phi'(z)=0$. 
This means that $b(s)=s$ and $B(s)=0$ for any $s \le 0$. 
For subsequent analysis, for any Lipschitz continuous $\psi$, we use the notation:


\begin{equation}\label{eq:def_errors_b}
    \psi[s,r]
    :=
    \begin{cases}
        \dfrac{\psi(s)-\psi(r)}{s-r}, &\text{if } s\neq r, \\[1em]
        \dfrac{1}{2}\left[\lim\limits_{z\nearrow s} \psi'(z) + \lim\limits_{z\searrow s} \psi'(z)\right], &\text{if }  s=r.
    \end{cases}
\end{equation}
Using \ref{ass:A1b}, one easily gets for any $s, r \in \mathbb{R}$,
\begin{align}\label{eq:condition}
0 \leq b[s,r], \quad B[s,r]\leq 1 \quad {\text{and}} \quad 
 b[s,r]+B[s,r] \geq 1.
\end{align} 
The graphs of $b$ and $B$ given by \eqref{eq:bBexpression} for various $\Phi$ function are presented in \Cref{bB-graphs}. 

\noindent
Finally, for the source function $f$ and initial condition $u_0$ we assume the following
\begin{enumerate}[label=(B.\arabic*)]
\item \label{ass:A2} $u_0\in L^2(\Om)$ is in $[0,\omega)$ a.e. The source function $f\in L^2(Q)$ admits a non-negative solution to Problem~\ref{eq:reformulated_weak_form}.
\end{enumerate}
\subsection{Notations and functional spaces}\label{subsec:functional-spaces}
We briefly introduce the functional setting used throughout the paper. 
Let $L^2(\Omega)$ be the space of measurable, square-integrable functions defined on $\Omega$. The space $H^1(\Omega)$ includes functions in $L^2(\Omega)$ having weak derivatives in $L^2(\Omega)$, while, $H^1_0(\Omega)$ contains functions in $H^1(\Omega)$ having vanishing trace on $\partial \Omega$. 
The dual space of $H^1_0(\Omega)$ is denoted by $H^{-1}(\Omega)$. The inner product and norm in $L^2(\Omega)$ are denoted by $(\cdot, \cdot)$ and $\|\cdot\|$, respectively, while the duality pairing between $H^{-1}(\Omega)$ and $H_0^1(\Omega)$ is denoted by $\langle \cdot, \cdot \rangle$. We also use the notation $a\lesssim b$ to indicate that there exists a generic constant $C>0$, independent of the discretization parameters, such that $a\leq Cb$.

We will use the Poincar\'e--Friedrichs inequality which states that there exists a constant $C_{\,\!_{\Omega}}^{\mathrm{P}} > 0$ (depending only on the shape of $\Omega$) such that, for any $v \in H_0^1(\Omega)$ one has
\begin{align}\label{eq:Poincare_Friedrichs inequality}
\|v\|\le C_{\,\!_{\Omega}}^{\mathrm{P}} \, h_\Om\,\|\nabla v\|.
\end{align}
\noindent
Here, $h_\Om$ is the diameter of $\Om$. By \eqref{eq:Poincare_Friedrichs inequality}, $\|\nabla v\|$ defines an equivalent norm on $H_0^1(\Omega)$, and a norm on $H^{-1}(\Omega)$ can be defined as
\begin{align}
    \|g\|_{H^{-1}(\Omega)} 
:= \sup_{\varphi \in H_0^1(\Omega)\setminus \{0\}} 
\frac{\langle g, \varphi \rangle}{\|\nabla \varphi\|} \leq C_{\,\!_{\Omega}}^{\mathrm{P}} \, h_\Om \|g\|,\label{eq:defHinv_norm}
\end{align}
the last inequality holding only if additionally $g\in L^2(\Om)$.
By the Riesz representation theorem, for each $g \in H^{-1}(\Omega)$ there exists a unique $G_g \in H_0^1(\Omega)$ satisfying
\begin{align}\label{eq:weak_poisson}
(\nabla G_g, \nabla \varphi) = \langle g, \varphi \rangle \text{ for all } \varphi \in H_0^1(\Omega), \text{ with } \|g\|_{H^{-1}(\Omega)} = \|\nabla G_g\|.
\end{align}

\noindent
Let $Q'\subseteq Q$ be a space-time subdomain of $Q$. For the fluxes we use the space
\[
H(\mathrm{div};Q')
:= \big\{ \mathbf{v} \in [L^2(Q')]^d : \text{ the weak (spatial) divergence } \nabla\!\cdot \mathbf{v} \in L^2(Q') \big\},
\]
equipped with the norm $\|\mathbf{v}\|_{H(\mathrm{div};Q'.)} 
:= \|\mathbf{v}\|_{L^2(Q')} + \|\nabla\!\cdot \mathbf{v}\|_{L^2(Q')}$.

\noindent
For any Banach space $X$, the Bochner space  $L^2(0,T;X)$ consists of all strongly measurable functions $v:(0,T)\to X$ such that
\[
\|v\|_{L^2(0,T;X)}
:=\left[\int_0^T \|v(t)\|_X^2\,\dd t \right]^{1/2}
<\infty.
\]
We will also use the Bochner-Sobolev space $H^1(0,T;X)
:=\{v\in L^2(0,T;X): \partial_t v\in L^2(0,T;X)\}$. Using these, we define the main spaces that are used in the subsequent analysis:
\begin{subequations}\label{eq:spaces}
 \begin{align}
\mathcal{S} &:= \Big\{ s \in L^2(Q) 
\;\big|\; 
b(s) \in H^1(0,T;H^{-1}(\Omega)), \;
B(s) \in L^2(0,T;H_0^1(\Omega))\Big\}\\
\mathcal{V} &:= L^2(0,T;H_0^1(\Omega)) \cap H^1(0,T;H^{-1}(\Omega)) \\
\mathcal{W} &:= L^2(0,T;H^2(\Omega)) \cap H^1(0,T;L^2(\Omega)).
\end{align}
\end{subequations}
We will seek weak solutions to \eqref{eq:reformulation} in the space $\calS$. This solution will be obtained through iterations, where the iterates will lie in $\mathcal{V}$. Observe that $\mathcal{W} \subset \mathcal{V}$ and $\mathcal{W}\subset \calS$, but $\calV\not\subset \calS$.
The space $\mathcal{V}$ is also compactly embedded in $L^2(Q)$ due to Aubin-Lions Lemma, and continuously embedded in $C([0,T];L^2(\Omega))$, see \cite{simon1986compact}. This embedding ensures that functions in $\mathcal{V}$ are well-defined for every $t \in [0, T]$. Therefore, one can equip  $\mathcal{V}$ with the norm
\begin{align}\label{eq:norm_V}
    \|v\|_{\mathcal V}^2
:=\int_0^T \!\big(\|\nabla v(t)\|^2+\|\partial_t v(t)\|_{H^{-1}(\Omega)}^2\big)\,\dd t
+\|v(T)\|^2.
\end{align}
\noindent

\subsection{Weak formulation}
We can now introduce the concept of a weak solution to \eqref{eq:reformulation}:

\begin{problem}[{Weak solution of \eqref{eq:reformulation}}]\label{eq:reformulated_weak_form}
A weak solution of \eqref{eq:reformulation} is a function $s \in \cal S$ satisfying $b(s(0))=u_0$ in $L^2(\Omega)$, and for all $\varphi\in L^2(0,T;H_0^1(\Omega))$,
    \begin{align}\label{eq:weak form}
    \int_0^T\left\langle \partial_t b(s), \varphi\right\rangle+ \int_0^T\left(\nabla B(s), \nabla \varphi\right) = \int_0^T\left( f, \varphi\right).
\end{align}
\end{problem} 
Without being exhaustive, for results concerning the existence and uniqueness of solutions for doubly-degenerate equations we refer to \cite{alt1983quasilinear, ANDREIANOV20173633,CarilloEntropy,CARRILLO199993,DroniouEymardTalbot,LuJaeger, pop2011regularization,ZouExistence}. 

\subsection{Residual formulation}\label{subsec:Residual-formulation}
To prepare for the space-time convergence and error analysis, we introduce a residual notation for Problem \ref{eq:reformulated_weak_form}. More precisely, we associate a residual operator $\calR : \calS 
\;\longrightarrow\; L^2(0,T;H^{-1}(\Om))$ to \eqref{eq:weak form}. Specifically, given $s \in \mathcal{S}$, $\calR(s) \in L^2(0,T;H^{-1}(\Om))$ is defined as
\begin{align}\label{eq:residual}
   \calR(s)( \varphi) := \int_0^T\!\left\langle \partial_t b(s), \varphi \right\rangle 
  + \int_0^T\!(\nabla B(s), \nabla \varphi)
  - \int_0^T\!( f, \varphi )
\end{align}
for all $\varphi \in L^2(0,T;H_0^1(\Omega))$. Observe that $\calR(s)=0$ if and only if $s$ solves Problem \ref{eq:reformulated_weak_form}. For future reference, we also define the operator for the heat equation $\resH:\calV\to L^2(0,T;H^{-1}(\Om))$ as 
\begin{align}\label{eq:resH_def}
        \resH(s)(\varphi)   
  := \int_0^T\!\left\langle \partial_t s, \varphi \right\rangle 
  + \int_0^T\!(\nabla s, \nabla \varphi).
\end{align}
The following result [Theorem 2.1 of \cite{ErnSmearsVohralik2017}] will be used multiple times.
\begin{lemma}\label{lemma_2.2} For any $s\in \calV$,
\[\|s\|_{\calV}^2 =\int_0^T \|\resH(s)\|^2_{L^2(0,T;H^{-1}(\Omega))} + \|s(0)\|^2.\]\label{lemma:resH}
\end{lemma}
\vspace{-15mm}
\begin{proof}
 Let \( G_s \in L^2(0,T;H_0^1(\Omega)) \)  satisfy by Riesz representation theorem 
\begin{align}\label{eq:GF_test}
    \int_0^T \left(\nabla G_s,\nabla \psi\right) = \int_0^T \left(s, \psi\right).
\end{align}
for all $\psi \in L^2(0,T;H_0^1(\Omega))$.
Observe that, for a.e.\ $t$, $G_s(t)$ is the Green's function associated to the Laplace equation with homogeneous Dirichlet boundary condition, i.e., the solution to $-\Delta G_s(t) = s(t)$ in $\Omega$. Since $s\in H^1(0,T;H^{-1}(\Omega))$, one obtains $\partial_t G_s \in L^2(0,T;H_0^1(\Omega))$, and for all $\psi \in L^2(0,T;H_0^1(\Omega))$
\begin{align}\label{eq:nabla_GF}
\int_0^T (\nabla \partial_t G_s, \nabla \psi)
= \int_0^T \langle \partial_t s, \psi\rangle,\text{ implying }
\|\partial_t G_s\|_{L^2(0,T;H^1_0(\Omega))}
\overset{\eqref{eq:weak_poisson}}{=}
\|\partial_t s\|_{L^2(0,T;H^{-1}(\Omega))}.
\end{align}
Then, for the residual $\resH$ one has $\resH(s)(\psi)= \int_0^T \left( \nabla ( \partial_t G_s + s), \nabla \psi \right)$. 
Therefore,
\begin{align*}
&\| \resH(s) \|^2_{L^2(0, T;H^{-1}(\Om))}\overset{\eqref{eq:weak_poisson}}= \int_0^T \| \nabla ( \partial_t G_s + s) \|^2 = \int_0^T \left( \| \nabla \partial_t G_s \|^2 + \| \nabla s\|^2 + 2 ( \nabla \partial_t G_s, \nabla s ) \right)\\
&\overset{\eqref{eq:nabla_GF}}= \|\p_t s\|_{L^2(0,T;H^{-1}(\Om))}^2 + \|\nabla s\|_{L^2(Q)}^2 + 2 \int_0^T \langle \partial_t s, s \rangle \, \overset{\eqref{eq:norm_V}}= \|s\|_\calV^2 - \|s(0)\|^2. 
\end{align*}
In the last step, we have used the equality $2\int_0^T\langle \p_t s, s\rangle = \int_0^T \frac{d}{dt} \|s\|^2=\|s(T)\|^2 - \|s(0)\|^2$. 
\end{proof}

\section{Linearization scheme}\label{Linearization}
Note that Problem \ref{eq:reformulated_weak_form} is nonlinear in both $b(s)$ and $B(s)$; we consider a linear iterative scheme to approximate its solution.
With $i \in \mathbb{N}$ denoting the iteration index, we define the sequence of approximations \(\{s^i\}_{i \in \mathbb{N} }\) as follows. Given $s^i \in \mathcal{S}$, the next iterate solves the following problem. 
\begin{problem}[{Space-time L-scheme for Problem \ref{eq:reformulated_weak_form}}]\label{pb:2}
Let $i\in \N$ be fixed, and $L>1$ be a constant. Assume that $s_0 \in L^2(\Omega)$ can be taken such that $b(s_0) = u_0$, and $s^i \in  (\mathcal{S}\cap \calV)$ is known, satisfying $s^i(0)=s_0$. Find $s^{i+1} \in \mathcal{V}$ such that $s^{i+1}(0)=s_0$ and for all $\varphi \in  L^2(0,T;H_0^1(\Omega))$, one has
\begin{align}\label{eq:linearization_gen}
\int_0^T \!\left\langle \partial_t \left( L(s^{i+1} - s^i) + b(s^i) \right), \varphi \right\rangle \,  
+ \int_0^T \!\left( \nabla \left( L(s^{i+1} - s^i) + B(s^i) \right), \nabla \varphi \right) 
= \int_0^T \!( f, \varphi ).
\end{align}
\end{problem}
Observe that \eqref{eq:linearization_gen} is nothing but the linear heat equation in a weak form. This linearization approach is inspired by the $L$-scheme \cite{list2016study,pop2004mixed} developed for elliptic problems. This particular implementation with $b$-$B$ functions is inspired by the double splitting scheme in \cite{javed2025robust}. By \eqref{eq:condition}, $b$ and $B$ are Lipschitz continuous, both having $1$ as 
Lipschitz constant. To ensure convergence, one needs to take the stabilization parameter $L \geq 1$, and the proof of \Cref{theo:converegence} below gives the additional restriction $L < 2$. Therefore, we take $L \in [1, 2)$, and in practice use $L = 1$. We define the linearized residual
\[
\mathcal{R}_{\mathrm{lin}} :
\mathcal{V}\times(\mathcal{S}\cap\mathcal{V})
\to L^2(0,T;H^{-1}(\Omega)),
\]
\begin{align}\label{eq:defRlin}
\mathcal{R}_{\mathrm{lin}}(s,\bar{s})&:= L \mathcal{R}_{\mathrm{H}}(s-\bar{s}) + \mathcal{R}(\bar{s}),
\end{align}
with $\mathcal{R}_{\mathrm{H}}$ introduced in \eqref{eq:resH_def}.
\begin{remark}[Operator definition of the L-scheme] Given $s^i \in \mathcal{S}\cap\mathcal{V}$, Problem \ref{pb:2} can be interpreted as to find $s^{i+1} \in \mathcal{V}$ satisfying $s^{i+1}(0)=s_0$ such that
\begin{align}
&\mathcal{R}_{\mathrm{lin}}(s^{i+1},s^i)= 0.
\label{eq:l_b}
\end{align}
\end{remark}

The main results regarding well-posedness and the convergence of the scheme are summarized in the following theorems.
\begin{theorem}[Well-posedness of the L-scheme]
\label{Convergence_theorem}
Let $s^i \in \mathcal{S}$ be given. Then, for any stabilization parameter $L\geq1$, 
the linearized problem~\eqref{eq:linearization_gen} admits a unique solution $s^{i+1} \in \mathcal{V}$. 
If, in addition, $s^i \in \mathcal{W}$, then $s^{i+1} \in \mathcal{W}\subset (\calS\cap \calV)$, and the iteration preserves regularity.
\end{theorem}
\noindent
Given $s^i \in \mathcal{S}$, the existence and uniqueness of $s^{i+1} \in \mathcal{V}$ are well known (see, e.g., \cite[Chapter~7]{evans2022partial}). Moreover, if $s^i\in \calW$, then $\calR(s^i)$ can be extended to a linear functional on $L^2(Q)$. The embedding $\calW\subset C(0,T;H^1_0(\Om))$ gives $s^i(0)\in H^1_0(\Om)$. Using regularity results from \cite[Chapter 7]{evans2022partial}, one also gets $s^{i+1}\in \calW$.


\begin{theorem}[Convergence of the L-scheme]\label{theo:converegence}
Take $L \in (1,2)$ and let $s^0\in \calW$ be an initial guess. Further, let $\{s^i\}_{i\ge 0}\subset \calW$ be the sequence generated by the $L$-scheme, and $s\in \calS$ be the weak solution to Problem~\ref{eq:reformulated_weak_form}. For every $t\in(0,T]$, following convergence results hold as $i \to \infty$
\begin{subequations}\label{eq:L_conv_1}
    \begin{align}
\|s^{i+1}-s^{i}\| &\longrightarrow 0
&&\text{in } L^2(Q),\label{eq:L_conv_1a}\\
b(s^i) &\longrightarrow b(s)
&&\text{in } L^2(0,T;H^{-1}(\Omega)),\\
\int_0^t B(s^i) &\longrightarrow \int_0^t B(s)
&&\text{in } L^\infty(0,T;L^2(\Omega)),\\
\int_0^t [B(s^{i-1}) + L(s^i-s^{i-1})]& \longrightarrow \int_0^t B(s) &&\text{in } L^\infty(0,T;H^1_0(\Omega)).
\end{align}
\end{subequations}
\textbf{[Improved convergence]:} If, in addition, the inverse mapping $\Phi^{-1}$ is $\a$-H\"older continuous for some $\a\in (0,1)$, then one has
\begin{align}\label{eq:L_conv_2}
b(s^i) &\longrightarrow b(s)
\quad &&\text{in } L^{1 + 1/\a}(Q).
\end{align}

\textbf{[Contraction]:}
If \( \ell := \inf_{r\in\mathbb{R}} \min\{b'(r),\, B'(r)\} > 0 \), then the iteration error is contractive,
and the iterative scheme converges linearly in ${L^2(Q)}$. 
\end{theorem}

First observe that $\|s^{i+1}-s^{i}\| \longrightarrow 0$ in \eqref{eq:L_conv_1a} does not imply the convergence of $s^i$ to $s$. However, it is used in the proofs for the other convergence results; therefore, we give it here explicitly. Also, observe that the H\"older continuity of $\Phi^{-1}$ holds if $\Phi$ has a porous medium equation type growth, i.e., $\Phi(u)\sim u^m$ for some $m\geq 1$. In this case, $\a=1/m$ since 
\[
|u_1-u_2|^m \leq |u_1^m-u_2^m|\lesssim |\Phi(u_1)-\Phi(u_2)|\;\; \text{ or }\;\; |\Phi^{-1}(w_1)-\Phi^{-1}(w_2)|\lesssim |w_1-w_2|^{1/m}.
\]
This covers all the cases in \Cref{fig:RichVsBio_phi}.
We prove \Cref{theo:converegence} below.  

\subsection{Convergence Analysis}\label{sec:convergence-proof}

We give here the proof for \Cref{theo:converegence}. To begin, we introduce the errors at the $i$-th iterate
\begin{align}\label{eq:def_errors_a}
e_s^i := s^i - s, \quad e_b^i := b(s^i) - b(s), \quad e_B^i := B(s^i) - B(s).
\end{align}
Subtracting ~\eqref{eq:linearization_gen} from ~\eqref{eq:weak form}, and rearranging the resulting terms, yields
\begin{align}\label{eq:linearized_error_form_eq}
\int_0^T \left\langle \partial_t \left( L(e_s^{i+1} - e_s^i) + e_b^i \right), \varphi \right\rangle \,  
+ \int_0^T \left( \nabla \left( L\left(e_s^{i+1} - e_s^i\right) + e_B^i \right), \nabla \varphi \right) \, = 0.
\end{align}
Observe that $e^i_s(0) = e^i_b(0) = 0$ for any $i \in \mathbb{N}$. With $\chi_{(0,t)}$ denoting the characteristic function of $(0,t)$, for any $t \in (0,T]$ one can take in \eqref{eq:linearized_error_form_eq} $\varphi = \chi_{(0,t)}\psi$ with $\psi \in H^1_0(\Omega)$, using $(L(e^{i+1}_s -e^i_s) + e^i_b)(0)=0$ and $\int_0^t \p_t(L(e^{i+1}_s -e^i_s) + e^i_b)(\t)\,\dd \t= (L(e^{i+1}_s -e^i_s) + e^i_b)(t)$, to obtain
\begin{align}\label{eq:time-integrated-error}
\left\langle \left( L(e_s^{i+1} - e_s^i) + e_b^i \right)(t), \psi \right\rangle \,  
+ \left( \nabla \int_0^t \left( L\left(e_s^{i+1} - e_s^i\right) + e_B^i \right), \nabla \psi \right) \, = 0.
\end{align}
Therefore, the Riesz representative of $(L(e^{i+1}_s -e^i_s) + e^i_b)(t)$ is
\begin{align}\label{eq:green_def_new}
    \mathcal{G}^{i+1}(t) 
    := -\int_0^{t} 
    \bigl(L(e_s^{i+1} - e_s^i)(\tau) + e_B^i(\tau)\bigr)
    \,\dd\tau.
\end{align}
From \eqref{eq:weak_poisson}, one gets for a.e.\ $t\in (0,T]$
\begin{align}\label{eq:green_riesz}
\bigl(\nabla\mathcal{G}^{i+1}(t), \nabla\psi\bigr) = 
\bigl\langle L(e_s^{i+1}-e_s^i)(t)+e_b^i(t), \psi
\bigr\rangle, 
\end{align}
for all $ \psi\in H_0^1(\Omega)$, and $\|\nabla\mathcal{G}^{i+1}(t)\| =  \|L(e_s^{i+1}-e_s^i)(t)+e_b^i(t)\|_{H^{-1}(\Omega)}$. Furthermore, since $L(e_s^{\,i+1}-e_s^{\,i}) + e_b^{\,i}\in H^1\big(0,T;H^{-1}(\Omega)\big)$, one has $\mathcal{G}^{\,i+1} \in H^1\big(0,T;H_0^1(\Omega)\big)$.

The remaining part of the proof is divided into three steps.

\textbf{(Step 1) An inequality involving successive iterations:} Inserting $\varphi =\chi_{(0,\tilde{t})} \mathcal{G}^{i+1}$ into~\eqref{eq:linearized_error_form_eq}  gives
\begin{align}\label{eq:Main_new}
\int_0^{\tilde{t}} \left\langle \partial_t \left( L(e_s^{i+1} - e_s^i) + e_b^i \right), \mathcal{G}^{i+1} \right\rangle \,  + \int_0^{\tilde{t}} \left( \nabla \left( L\left(e_s^{i+1} - e_s^i\right) + e_B^i \right), \nabla \mathcal{G}^{i+1} \right) = 0.
\end{align}
Denoting the terms on the left by $T_1$ and $T_2$ respectively, one has
\begin{align}
 T_1 &:= \int_0^{\tilde{t}} \left\langle \partial_t \left( L(e_s^{i+1} - e_s^i) + e_b^i \right), \mathcal{G}^{i+1} \right\rangle \overset{\eqref{eq:green_riesz}}{=} 
\int_0^{\tilde{t}} (\nabla \partial_t \mathcal{G}^{i+1}, \nabla \mathcal{G}^{i+1})= \frac{1}{2} \int_0^{\tilde{t}} \frac{d}{dt} \big\|\nabla \mathcal{G}^{i+1}(t)\big\|^2\,dt \nonumber\\
&= \frac{1}{2}\big\|\nabla \mathcal{G}^{i+1}({\tilde{t}})\big\|^2\overset{\eqref{eq:green_riesz}}{=}\frac{1}{2} \big\|(L(e_s^{i+1}-e_s^i)+e_b^i)({\tilde{t}})\big\|_{H^{-1}(\Omega)}^2.\label{T_1_energy_new}
\end{align}
In the above, we used  $\mathcal{G}^{i+1}(0)=0$ that follows from~\eqref{eq:green_def_new}.
We now consider the second term in \eqref{eq:Main_new}, expanded as
\begin{align*}
T_2 & := \int_0^{\tilde{t}} \left( \nabla \left( L\left(e_s^{i+1} 
- e_s^i\right) + e_B^i \right), \nabla \mathcal{G}^{i+1} \right) \overset{\eqref{eq:green_riesz}}{=} 
\int_0^{\tilde{t}} \Bigl( L\left(e_s^{i+1} - e_s^i\right) + e_b^i,\, 
L\left(e_s^{i+1} - e_s^i\right) + e_B^i \Bigr)\\
&= L^2 \int_0^{\tilde{t}} \| e_s^{i+1} - e_s^i \|^2
+L \int_0^{\tilde{t}} \left( e_B^i +e_b^i, \,
e_s^{i+1} - e_s^i\right) 
+ \int_0^{\tilde{t}} \left( e_B^i, e_b^i \right) \\
&\overset{\eqref{eq:def_errors_a},\eqref{eq:def_errors_b}}{=} 
L^2\int_0^{\tilde{t}} \|e_s^{i+1} - e_s^i\|^2 +L \int_0^{\tilde{t}} \bigl((b[s^i,s]+B[s^i,s])e_s^i,\,
e_s^{i+1} - e_s^i\bigr)
+\int_0^{\tilde{t}} \bigl(e_s^i b[s^i,s],\,e_s^i B[s^i,s]\bigr).
\end{align*}
Applying the elementary identity $a(a-b)=\tfrac{1}{2}(a^2-b^2+(a-b)^2)$, with $a=e_s^{i+1}$ and $b=e_s^{i}$ to the second term (mixed product), one gets after combining terms
\begin{align}\label{eq:T2_rewritten_new}
T_2
&= \frac{L}{2}\int_0^{\tilde{t}} \int_\Omega 
    (b[s^i,s] + B[s^i,s])\, |e_s^{i+1}|^2\,dx\,dt\nonumber\\ 
  &\quad      
  +\frac{L}{2}\int_0^{\tilde{t}} \int_\Omega 
    \big( 2L - (b[s^i,s] + B[s^i,s]) \big)\, 
    |e_s^{i+1} - e_s^i|^2\,dx\,dt       \nonumber\\ 
  &\quad 
  + \frac{1}{2}\int_0^{\tilde{t}} \int_\Omega 
    \big( 2\, b[s^i,s]\, B[s^i,s] 
    - L(b[s^i,s] + B[s^i,s]) \big)\,
    |e_s^i|^2\,dx\,dt.
\end{align}
Inserting~\eqref{T_1_energy_new}--\eqref{eq:T2_rewritten_new} 
into~\eqref{eq:Main_new} and multiplying by $2$ gives
\begin{align}\label{eq:Main_equality_new}
    &\left\| \nabla \mathcal{G}^{i+1}({\tilde{t}})\right\|^2
    + \int_0^{\tilde{t}} \int_\Omega L \left( b[s^i,s] + B[s^i,s] \right) \left| e_s^{i+1} \right|^2 \nonumber \\
    &+ \int_0^{\tilde{t}} \int_\Omega \left( 2L^2 - L\left( b[s^i,s] + B[s^i,s] \right) \right) 
    \left| e_s^{i+1} - e_s^i \right|^2 
    + (2-L)\int_0^{\tilde{t}} \int_\Omega  b[s^i,s] B[s^i,s]  \left| e_s^i \right|^2 \nonumber \\
    &= \int_0^{\tilde{t}} \int_\Omega L\left( b[s^i,s] + B[s^i,s]-  b[s^i,s] B[s^i,s] \right) \left| e_s^i \right|^2.
\end{align}
We identify three coefficient functions 
in~\eqref{eq:Main_equality_new} as
\begin{subequations}\label{eq:P-coeff_new}
\begin{align}
P_1^i
&:=  L \left( b[s^i,s] + B[s^i,s] \right)
\overset{\eqref{eq:def_errors_b}}{=}
L \frac{(b + B)(s^i) - (b + B)(s)}{s^i - s}
\overset{\eqref{eq:condition}}{\geq} L, \\
P_2^i
&:= 2L^2 - L\left( b[s^i,s] + B[s^i,s] \right)
\overset{\eqref{eq:condition}}{\geq} 2L(L-1), \\
P_3^i
&:= L\left( b[s^i,s] + B[s^i,s]-  b[s^i,s] B[s^i,s] \right) = L - L(1 - b[s^i, s])(1 - B[s^i, s]) \overset{\eqref{eq:condition}}\in [0,L]. 
\end{align}
\end{subequations}
Using~\eqref{eq:P-coeff_new} in~\eqref{eq:Main_equality_new}, one has the following inequality, which is fundamental for convergence
\begin{align}\label{eq:final_estimate_new}
    &L \int_0^{\tilde{t}} \|e_s^{i+1}\|^{2}
    +\left\| \nabla \mathcal{G}^{i+1}({\tilde{t}})\right\|^2+
    (2-L)\int_0^{\tilde{t}} \int_\Omega e^i_B e^i_b\,dx\,dt + L(L-1)\int_0^{\tilde{t}} \|s^{i+1} -s^{\,i}\|^{2}\,dt
    \nonumber \\ 
    &\le\;
    L\int_0^{\tilde{t}} \|e_s^{\,i}\|^{2}\,dt.
\end{align}
\textbf{(Step 2) Proof of \eqref{eq:L_conv_1}:} 
Adding ~\eqref{eq:final_estimate_new} for $i=0$ to $N$, one has
\begin{align}\label{telescoped_estimate}
&L \int_0^{\tilde{t}} \|e_s^{N+1}\|^2\,dt + 
\sum_{i=0}^{N} 
\left\| \nabla \mathcal{G}^{i+1}({\tilde{t}})\right\|^2
+ (2-L)\sum_{i=0}^{N} 
\int_0^{\tilde{t}} \int_\Omega e_B^i\, e_b^i\,dx\,dt + L(L-1)\sum_{i=0}^{N} 
\int_0^{\tilde{t}} \|s^{i+1}-s^i\|^2\,dt
\nonumber\\
&\le L \int_0^{\tilde{t}} \|e_s^0\|^2\,dt.
\end{align}
Since ${\tilde{t}} \in (0,T]$ was chosen arbitrarily, and $\mathcal{G}^{i+1} \in  H^1(0,T;H^1_0(\Omega))$, we immediately get that for every $i\in \mathbb{N}$ and ${\tilde{t}} \in [0,T]$,
\begin{align}
    \left\| \nabla \mathcal{G}^{i}(\tilde{t})\right\|^2 \le L \int_0^{T} \|e_s^0\|^2\,dt,
\end{align}
so the sequence $\{\mathcal{G}^{i}\}_{i \in \mathbb{N}}$ has an equibounded norm in $L^{\infty}(0,T;H^1_0(\Omega))$. This immediately implies the same for the norm in $L^{\infty}(0,T;H^{-1}(\Omega))$ of $\{L(s^{i+1}-s^i) + e_b^i\}_{i \in \mathbb{N}}$. 

Further~\eqref{telescoped_estimate} implies the absolute convergence of the series, yielding the  convergences
\begin{align}\label{eq:inq_new}
&\|\left(L(s^{i+1} - s^i) + e_b^i\right)({\tilde{t}})
\|^2_{H^{-1}(\Omega)} = \left\|\nabla \int_0^{\tilde{t}} (L(s^{i+1} - s^i) + e_B^i)(\t)\, d\t\right\|= \left\| \nabla \mathcal{G}^{i+1}({\tilde{t}})\right\|\to 0,
\nonumber\\
&\int_0^{\tilde{t}} \|s^{i+1}-s^{i}\|^2\,dt \to 0, \qquad\text{and}\qquad
\int_0^{\tilde{t}}\int_\Omega e_b^i\,e_B^i\,dx\,dt \to 0,
\end{align}
as $i\to\infty$. Since $\tilde{t} \in (0,T]$ was chosen arbitrarily, one gets the convergence
\begin{align}
\|s^{i+1} - s^i\|_{L^2(Q)} \to 0 \quad \text{as } i \to \infty.
\end{align}
To prove that $e_b^i$ vanishes as $i \to \infty$, we recall the equiboundedness in $L^{\infty}(0,T;H^{-1}(\Omega))$ of 
$L(s^{i+1}-s^i) + e_b^i$. Using that for every $\tilde{t}\in (0,T]$ one has $\left\|\left(L(s^{i+1}-s^i) + e_b^i\right)(\tilde{t})\right\|_{H^{-1}(\Omega)} \to 0 \quad \text{as } i \to \infty$, the dominated convergence theorem, together with the equiboundedness, guarantees that 
\begin{align}
\left\|L(s^{i+1}-s^i) + e_b^i\right\|_{L^2(0,T;H^{-1}(\Omega))} 
\to 0 \quad \text{as } i \to \infty.
\end{align}
With this, we apply the triangle inequality and \eqref{eq:defHinv_norm} to obtain
\begin{align}
&\|e_b^i\|_{L^2(0,T;H^{-1}(\Omega))}
\le 
\|L(s^{i+1}-s^i)+e_b^i\|_{L^2(0,T;H^{-1}(\Omega))}
+L\,\|s^{i+1}-s^i\|_{L^2(0,T;H^{-1}(\Omega))} \nonumber\\
&\quad \leq \|L(s^{i+1} - s^i) + e_b^i\|_{L^2(0,T;H^{-1}(\Omega))} 
\overset{\eqref{eq:defHinv_norm}}{+}
LC_{\,\!_{\Omega}}^{\mathrm{P}} \, h_\Omega\,
\|s^{i+1}-s^i\|_{L^2(Q)}.\label{eq:b_convergence}
\end{align}
Since both terms on the right are vanishing  as $i \to \infty$, one gets that

\begin{align}
    \|e_b^i\|_{L^2(0,T;H^{-1}(\Omega))}  \;\to\; 0,
\qquad\text{i.e.,}\qquad
b(s^i)\to b(s) \quad \text{in } L^2(0,T;H^{-1}(\Omega)).\label{eq:bconv_new}
\end{align}
Next, by \eqref{eq:inq_new} one gets that for every $t\in(0,T]$,
\begin{align}
\left\|\nabla \int_0^t \left(L(s^{i+1}-s^i)+e_B^i\right)(\tau)\,d\tau\right\| \to 0 \quad \text{as } i \to \infty.
\end{align}
By Poincar\'e's inequality,
\begin{align}
\left\|\int_0^t \left(L(s^{i+1} - s^i) + e_B^i\right)(\tau)
\,d\tau\right\| \to 0 \quad \text{as } i \to \infty.
\end{align}
Further, since $\|s^{i+1}-s^i\|_{L^2(Q)} \to 0 \quad \text{as } i \to \infty$, one immediately gets that
\begin{align}
 \left\|\int_0^t (s^{i+1}-s^i)(\tau)\,d\tau\right\|_{L^2(\Omega)}\to 0 \quad \text{as } i \to \infty.  
\end{align}
 Using the triangle inequality again gives the convergence
\begin{align}
    \int_0^t B(s^i)\,d\tau \;\longrightarrow\; \int_0^t B(s)\,d\tau,
\end{align}
in $L^2(\Omega)$, for every $t\in (0,T]$. This completes the proof of \eqref{eq:L_conv_1}.


\textbf{(Step 3) Proving \eqref{eq:L_conv_2} and the contractive behaviour:} 
From  \eqref{telescoped_estimate} one also gets that
\[
\int_Q (B(s^i)-B(s))(b(s^i)-b(s)) 
= \int_Q e^i_B e^i_b \longrightarrow 0.
\]
However, due to the $\alpha$-H\"older continuity of $\Phi^{-1}$ one has
\begin{align*}
    &\int_Q |b(s^i)-b(s)|^{1+1/\alpha}
    \lesssim \int_Q |b(s^i)-b(s)|
    |\Phi^{-1}(b(s_1)) - \Phi^{-1}(b(s_2))|\\
    &= \int_Q (b(s^i)-b(s))(B(s^i)-B(s)) 
    \longrightarrow 0.
\end{align*}
This shows the strong convergence of $b(s^i)$ to $b(s)$ in $L^{1+1/\alpha}$. Finally, if $\ell:= \inf\{B',b'\}>0$, then we have 
\[
\int_Q (B(s^i)-B(s))(b(s^i)-b(s))
\geq \ell \int_Q \|e^i_s\|^2.
\]
Inserting this into~\eqref{eq:final_estimate_new} we have
\begin{align}
    &\int_0^T \|s^{i+1} -s\|^{2}
    +\frac{1}{L}\left\| \nabla \int_0^T\left (L\big(s^{i+1} - s^i\big) + B(s^i)-B(s)\right)\right\|^{2}\leq 
    \left (1-\ell \left(\frac{2-L}{L}\right)\right)\int_0^T \|s^{\,i}-s\|^{2}.
\end{align}
This shows that the iterative method is contracting the $L^2(Q)$ error in $s$ for every iteration step.


\subsubsection{Numerical study on L-scheme convergence}\label{sec:numerical_validation}
Our goal is to develop an adaptive space-time iterative scheme, based on a posteriori estimates and employing the $L$-type scheme introduced through Problem~\ref{pb:2}. Before proceeding with the adaptivity and a posteriori estimates, we assess the behaviour of this linearization by investigating the decay of the iteration errors. To this end, we compare it with the $N_{\mathrm{reg}}$-scheme
\cite{mitra2019modified,javed2025robust}, which can be seen as a regularized variant of the Newton scheme; see \eqref{eq:M-scheme}. For both schemes, the space-time discretization presented in \Cref{sec:apost} is used on a fixed mesh. The experiments are for a one-dimensional setting with Barenblatt data described in \Cref{sec:1DPME_1DST}.
Figure~\ref{fig:iter-combined} illustrates the evolution of the iteration error with respect to the iteration number for different types of nonlinearities, including degenerate ($b'$, $B'$ vanishes) and non-degenerate cases (when either $b'$ or $B'$ is bounded away from 0). The first row reports the error measured in the norm $\|b(s^{i+1})-b(s^i)\|_{L^2(Q)}$, while the second row shows the error measured in the norm $\|\partial_x(B(s^{i+1})-B(s^i))\|_{L^2(Q)}$. The regularized Newton ($N_{\mathrm{reg}}$) scheme with regularization parameter $\e=0.1$ is employed for the degenerate and single-degenerate cases (left and center columns) since lower values of $\e$ result in the iterations diverging, whereas $\e=0$ is used in the non-degenerate case (right column).

\begin{figure}[h]
\centering

\begin{subfigure}[t]{0.34\textwidth}
  \centering
  \includegraphics[width=\linewidth]{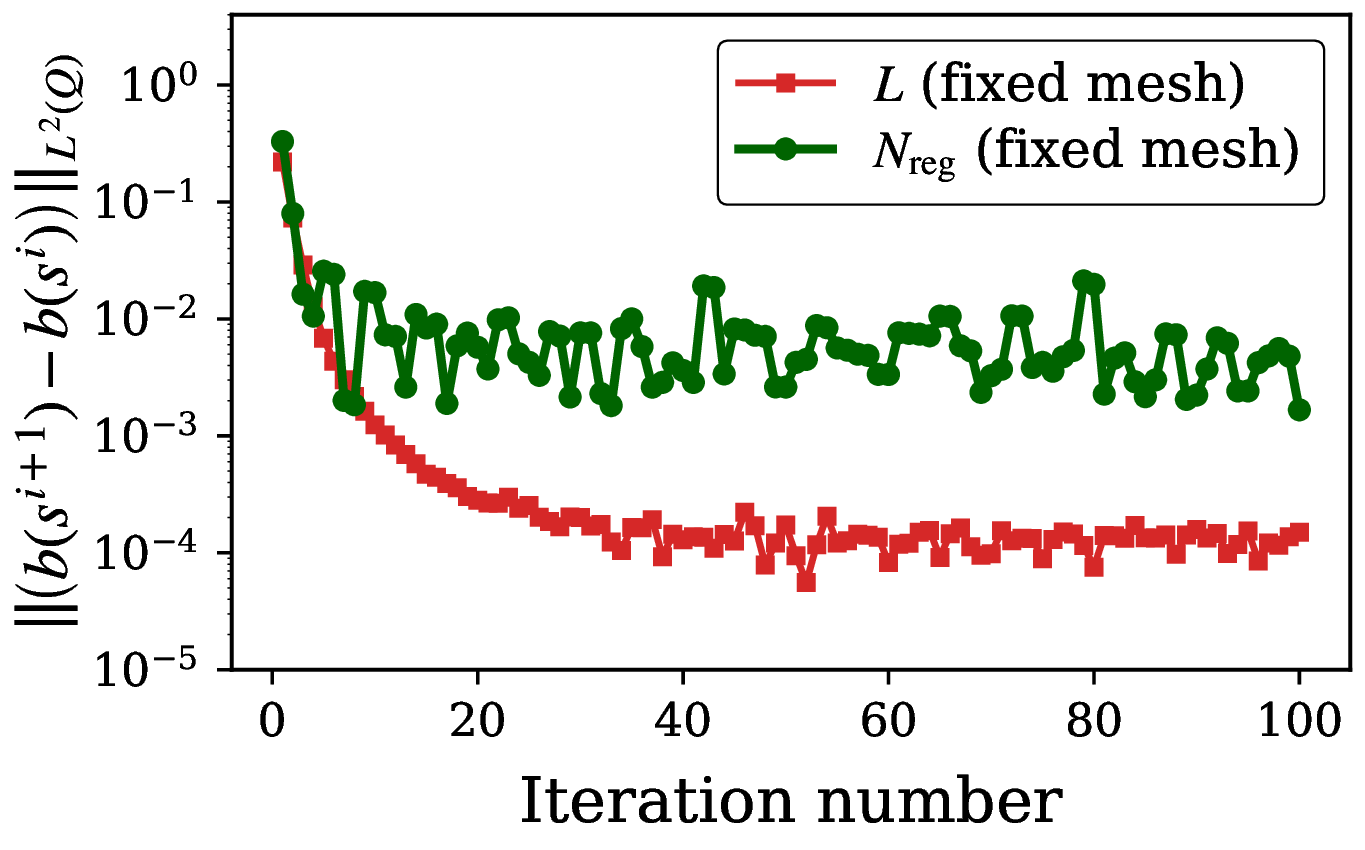}
  \caption{Degenerate}
\end{subfigure}\hfill
\begin{subfigure}[t]{0.33\textwidth}
  \centering
  \includegraphics[width=\linewidth]{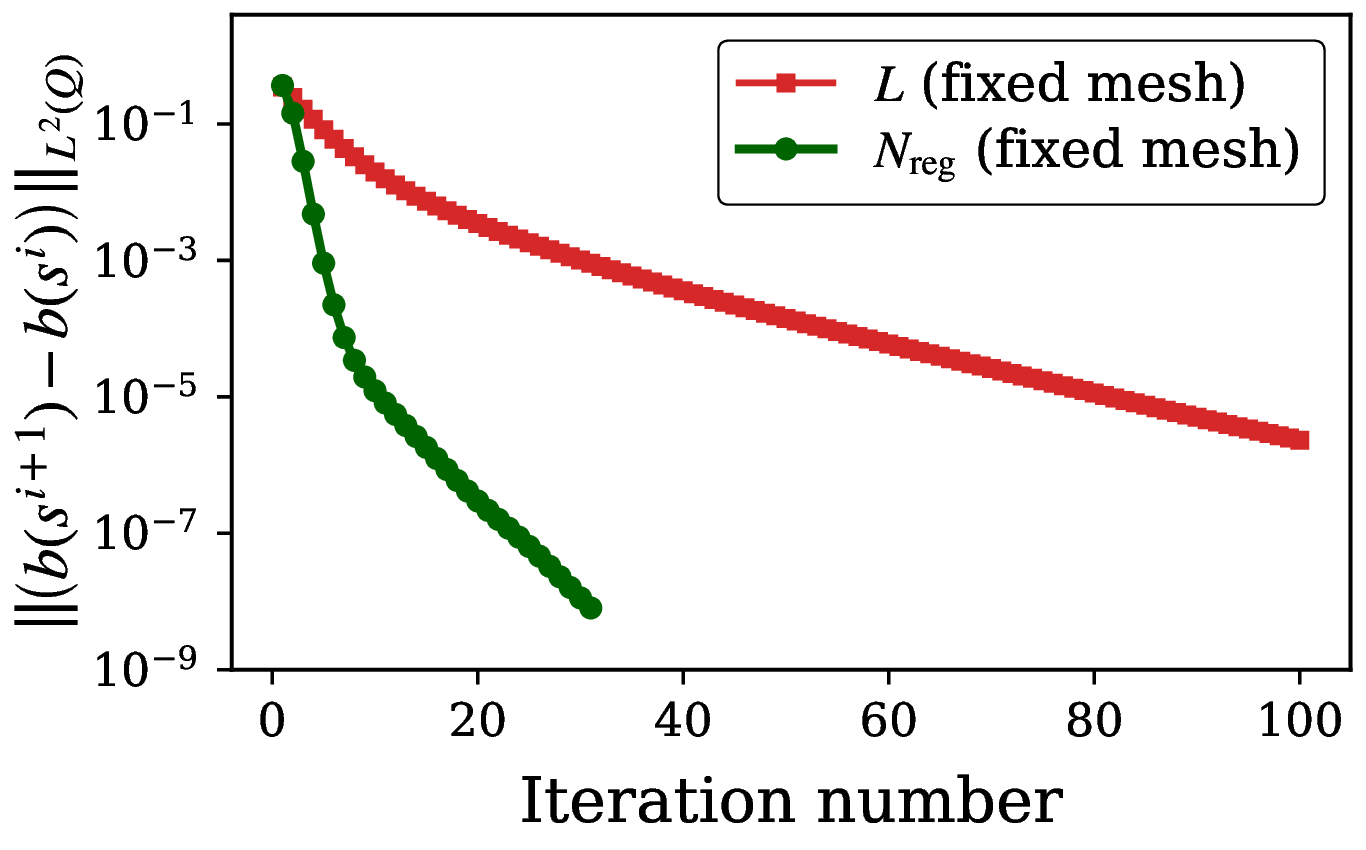}
 \caption{\centering Single-degenerate.}
 \end{subfigure}\hfill
\begin{subfigure}[t]{0.33\textwidth}
  \centering
  \includegraphics[width=\linewidth]{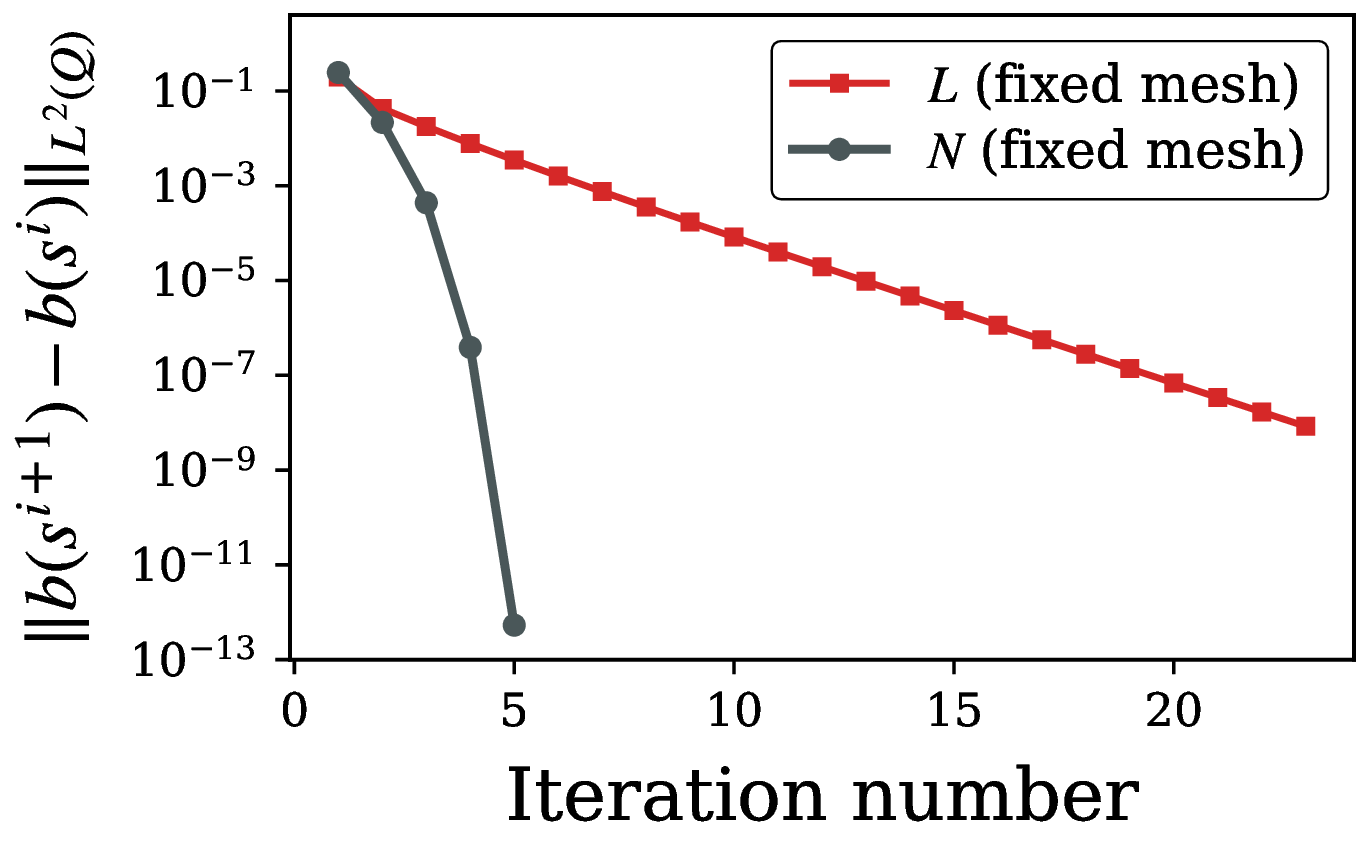}
  \caption{\centering Non-degenerate}
\end{subfigure}
  \begin{subfigure}[t]{0.34\textwidth}
    \centering
    \includegraphics[width=\linewidth]{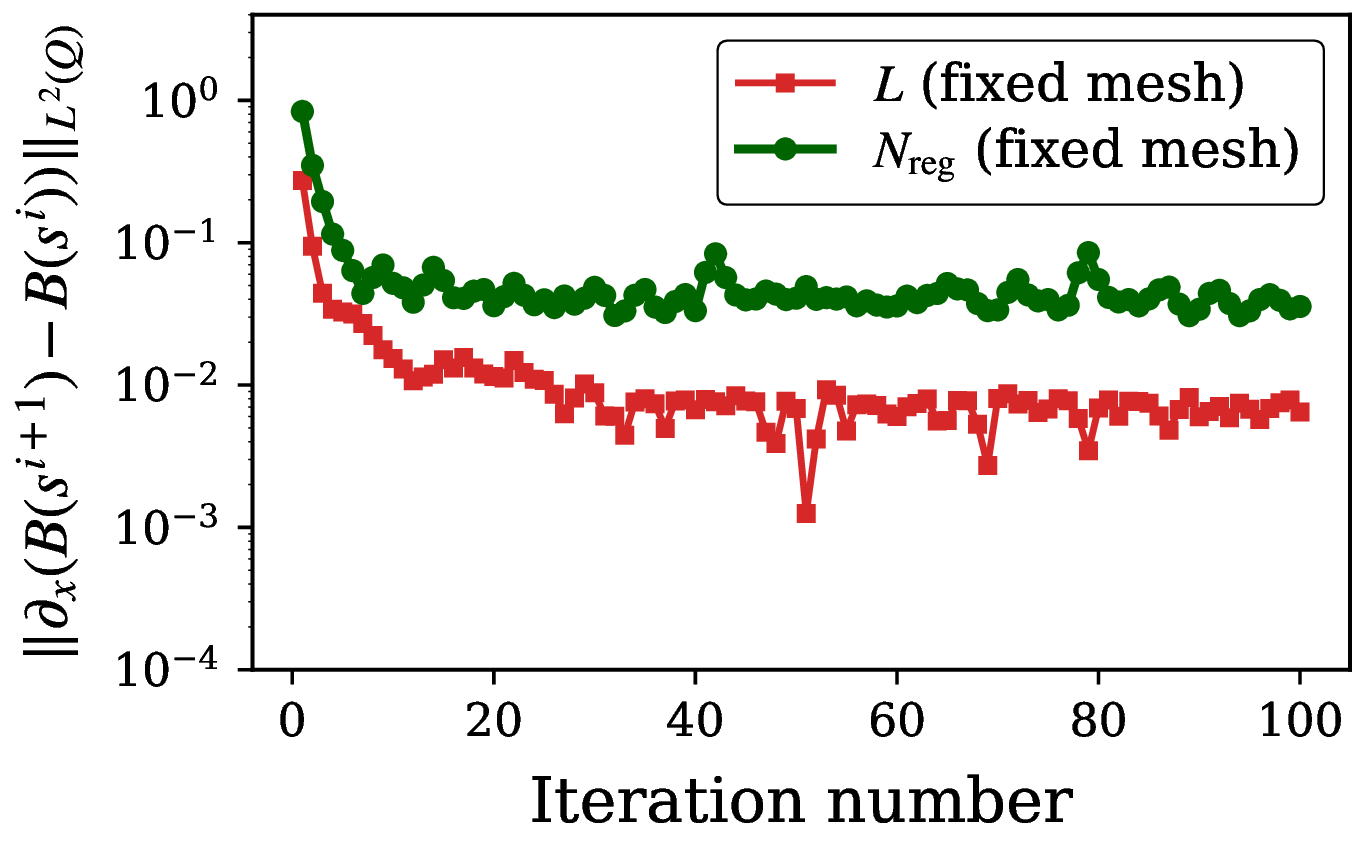}
    \caption{Degenerate }
  \end{subfigure}\hfill
  \begin{subfigure}[t]{0.33\textwidth}
    \centering    \includegraphics[width=\linewidth]{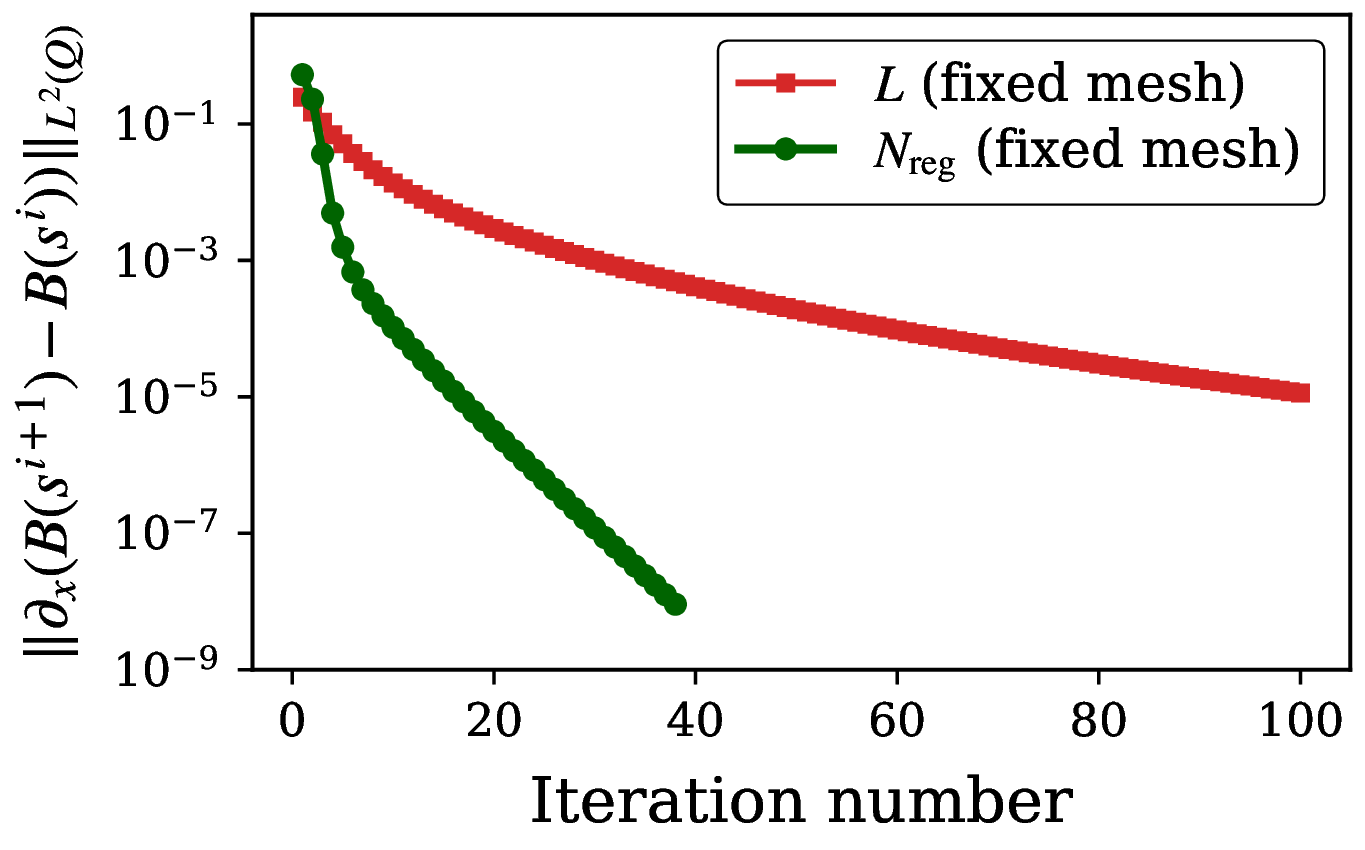}
 \caption{\centering Single-degenerate}
 \end{subfigure}\hfill
  \begin{subfigure}[t]{0.33\textwidth}
    \centering    \includegraphics[width=\linewidth]{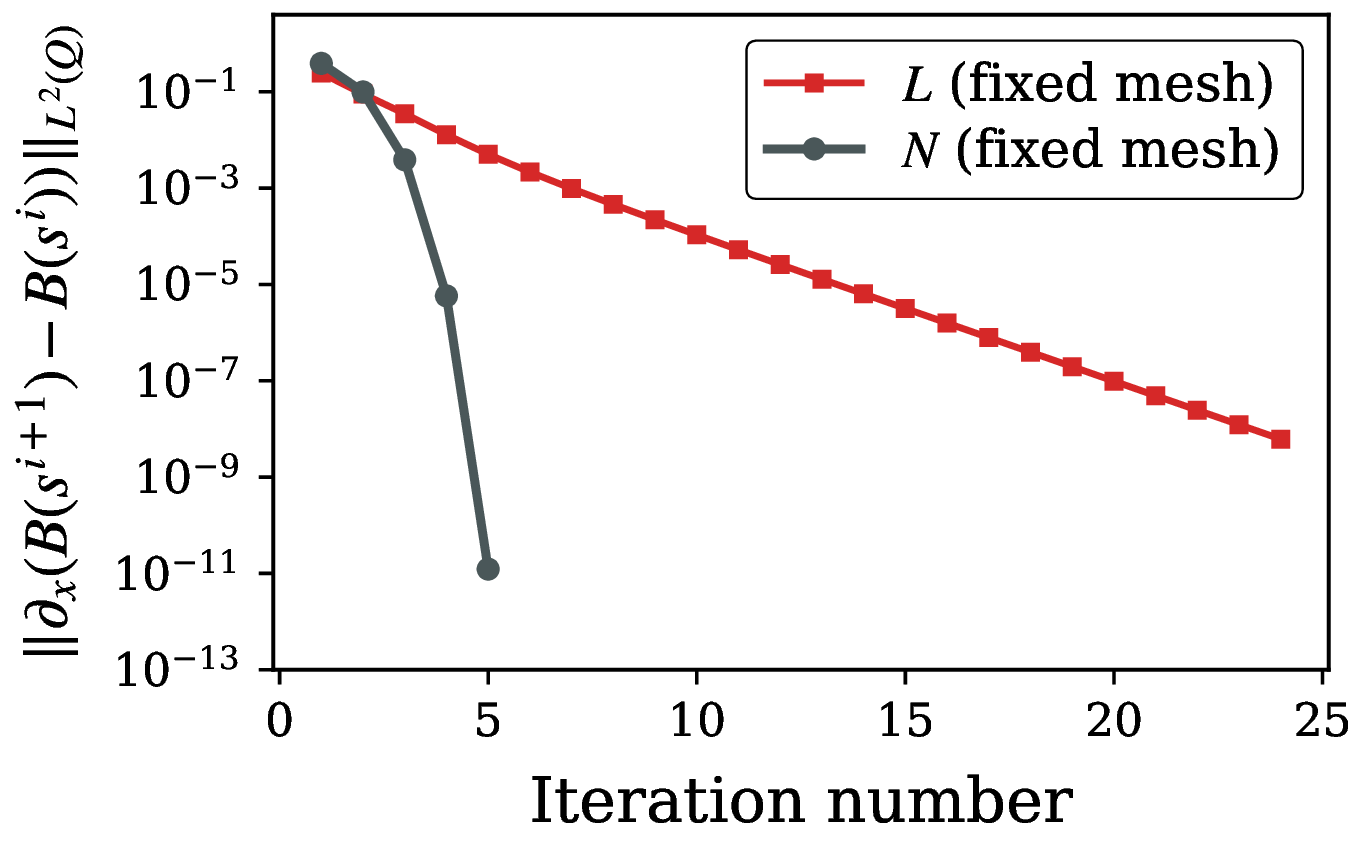}
  \caption{\centering Non-degenerate}
  \end{subfigure}
\caption{[\Cref{sec:numerical_validation}] Decay of the linearization error versus the number of iterations for the \(L\)-scheme, and the regularized Newton scheme \eqref{eq:M-scheme} for a fixed mesh. \textbf{Top row:} error measured as $\|b(s^{i+1})-b(s^i)\|_{L^2(Q)}$. \textbf{Bottom row:} error measured as $\|\partial_x(B(s^{i+1})-B(s^i))\|_{L^2(Q)}$. From left to right, the columns correspond to degenerate (porous medium equation, nonlinearity as in \eqref{eq:bB_PME}), single-degenerate ($B(s)=\frac{1}{2}s^2+\frac{1}{2}s,\; b(s)=s^2$ so that $\inf_{\mathbb{R}} B'>0$), and non-degenerate diffusion ($B(s)=\frac{1}{2}s^2+\frac{1}{2}s, \; b(s)=\frac{1}{6}s^3+\frac{1}{2}s$ so that $\inf_{\mathbb{R}} B',\, \inf_{\mathbb{R}} b'>0$) models. The regularized Newton scheme is used in the degenerate and single-degenerate cases with $\e=0.1$, while the classical Newton scheme ($\e=0$) is used in the non-degenerate case.}
\label{fig:iter-combined}
\end{figure}

For the single-degenerate case when $\inf_{\mathbb{R}} B'>0$ and non-degenerate cases, a significantly faster decay of the iteration error is observed for the regularized Newton ($N_{\mathrm{reg}}$) scheme compared to the \(L\)-scheme. In particular, for the non-degenerate case, the Newton scheme exhibits quadratic convergence (since $\e=0$). However, when \(B'(s)\) vanishes, the error decay for the regularized Newton scheme becomes irregular and eventually stagnates, even in the presence of the regularization parameter \(\e>0\), which appears insufficient to compensate for the vanishing diffusion. In contrast, for the degenerate diffusion case, the \(L\)-scheme exhibits a faster initial decay of the iteration error and reaches a much lower error level before eventually stagnating. Better error behaviour is seen for the error metric  $\|b(s^{i+1})-b(s^i)\|_{L^2(Q)}$ due to convergence \eqref{eq:L_conv_2} predicted in \Cref{theo:converegence}.

Nevertheless, the main strength of the L-scheme comes from the fact that it reduces the nonlinear degenerate problem to solving a set of heat equations. This allows us to combine it with adaptive mesh refinement with precise a posteriori error control obtained through the error-residual relationship derived below.

\section{Decomposition of error into discretization and linearization components}\label{subsec:decomposition}
The main idea in the a posteriori analysis is in the decomposition of the error into two components, linearization and discretization. To do so, with $s \in \cal S$ and $s^{i+1} \in \mathcal{V}$ solving Problems~\ref{eq:reformulated_weak_form} and \ref{pb:2}, respectively, we define $\calE_{\mathrm{lin}}$ and $\calE_{\mathrm{disc}}^{\rm H}$. 

Let $\calR$  and $\calR_{\rm lin}$ be the residuals associated with Problems~\ref{eq:reformulated_weak_form} and \ref{pb:2}, as defined in \eqref{eq:residual} and \eqref{eq:defRlin}, respectively. We define the total error of $s^i$ as a candidate solution to Problem \ref{eq:reformulated_weak_form} as
\begin{align}\label{eq:total_indicator}
\mathcal{E}(s^i) 
:=\left(\left\vert\!\left\vert \mathcal{R}(s^i) \right\vert\!\right\vert\!_{L^2(0,T;H^{-1}(\Omega))} ^2 + L^2 \|s^{i}(0)-s_0\|^2\right)^{\frac{1}{2}}.
\end{align}
Now, for a given $s^i$, the next iteration $s^{i+1}$ is a solution to Problem \ref{pb:2}. It is approximated by a fully discrete one $s^{i+1}_{h\tau}$. The corresponding  discretization error is defined as
\begin{align}\label{eq:disc_error}
\calE_{\mathrm{disc}}^{\rm H}(s^{i+1}_{h\t})
:= \left(\left\vert\!\left\vert \mathcal{R}_{\mathrm{lin}}(s_{h\tau}^{i+1}, s^i) \right\vert\!\right\vert\!_{L^2(0,T;H^{-1}(\Omega))} ^2 + L^2 \|s^{i+1}_{h\t}(0)-s_0\|^2\right)^{\frac{1}{2}}.
\end{align}
Note that, $\calE_{\mathrm{disc}}^{\rm H}$  vanishes if and only if $s^{i+1}_{h\t}$ solves Problem \ref{pb:2}. It can be seen as the discretization error since $s^{i+1}_{h\t}$ is typically a finite dimensional approximation of $s^{i+1}$. This association will become clearer in Proposition~\ref{lemma:SP_Iterative}. Furthermore, the linearization error is defined as 
\begin{align}\label{err:linearization}
\calE_{\mathrm{lin}}^i 
:= \|s^{i+1}_{h\t}- s^i\|_\mathcal{V}.
\end{align}
It inculcates the linearization component, vanishing if the iterations converge (even for a fixed discretization). These errors are related as follows from the following result.

\begin{theorem}[Decomposition of Error]\label{theo:decomp}
Let $L\in (1,2)$. Given $s^i\in (\calS \cap \calV)$, let $s^{i+1}_{h\t}\in \calV$ denote an approximation of $s^{i+1}\in \calV$, the solution to Problem \ref{pb:2}. With the linearization and discretization errors, $\mathcal{E}^i_{\mathrm{lin}}$ and $\mathcal{E}^{\mathrm{H}}_{\mathrm{disc}}$ defined in \eqref{err:linearization} and \eqref{eq:disc_error}, respectively, the total error defined in \eqref{eq:total_indicator} satisfies the reliability and efficiency bounds:
\begin{subequations}
    \begin{align}
\mathcal{E}(s^i)
&\leq \calE_{\mathrm{disc}}^{\mathrm{H}}(s^{i+1}_{h\t}) + L\calE_{\mathrm{lin}}^i
\quad &&\textup{(Reliability)}, \label{reliability_bound}\\
\calE_{\mathrm{disc}}^{\mathrm{H}}(s^{i+1}_{h\t})
&\leq \mathcal{E}(s^i) + L\calE_{\mathrm{lin}}^i
\quad &&\textup{(Efficiency)}.\label{eq:efficiency}
\end{align}
\end{subequations}
\end{theorem}

\textit{What makes \Cref{theo:converegence} quite interesting is as follows}: Since $s^i$ and $s^{i+1}_{h\t}$ are known (unlike the infinite dimensional object $s^{i+1}$), $\calE_{\mathrm{lin}}^i$ is easily estimable using \eqref{err:linearization}. In the next section, we will call this estimation $\eta_{\mathrm{lin}}^i$. On the other hand,  $\calE_{\mathrm{disc}}^{\rm H}$ corresponds to estimating the residual of a heat equation. This makes it possible to provide highly accurate, locally-efficient estimates that are robust with respect to the mesh-size, time-step size, and the spatial and temporal polynomial degree; see \cite{ErnSmearsVohralik2017}. We will call this estimator $\eta_{\mathrm{disc}}^{\rm H}$. The details regarding this estimator will be given below.

Now, as the iterations converge, upon meeting an adaptive stopping criteria (i.e., when $\eta_{\mathrm{lin}}^i\ll \eta_{\mathrm{disc}}^{\rm H}$), we can provide \emph{guaranteed, efficient a posteriori error bounds for the double degenerate problem that are not only robust with respect to discretization but also with respect to the nonlinearities}. This is the underlying idea of the adaptive (space-time) linear iterative scheme proposed in this work.

\subsection{Proof of \Cref{theo:decomp}}
The proof relies on the results stated below, which can be obtained straightforwardly from \Cref{lemma_2.2} and \eqref{eq:defRlin}-\eqref{eq:l_b}. First, given $s^i\in (\calS \cap \calV)$, if $s^{i+1}$ solves Problem \ref{pb:2} one has

\[
\calR(s^i) =-L\resH(s^{i+1}-s^i),
\]
and
\[
\calR_{\mathrm{lin}}(s^{i+1}_{h\t},s^i)= L\resH(s^{i+1}_{h\t}-s^i) + \calR(s^i)= L\resH(s^{i+1}_{h\t}- s^{i+1}).
\]
The results below involve the errors introduced in   \eqref{eq:total_indicator}, \eqref{eq:disc_error} and \eqref{err:linearization}.


\begin{proposition}[Equivalent expressions of the dual norms of the residual and linearized residual]\label{lemma:SP_Iterative}
One has the following 
\begin{align}
[\mathcal{E}(s^i)]^2 &=L^2 \|s^{i+1}-s^i\|_{\calV}^2, \\
[\calE_{\mathrm{disc}}^{\rm H}(s^{i+1}_{h\t})]^2 &= L^2  \|s^{i+1} - s^{i+1}_{h\t}\|_{\calV}^2.
\end{align}
\end{proposition}
These equalities are a direct consequence of \Cref{lemma_2.2}, and of the initial conditions $s^{i+1}(0)=s^i(0)=s_0$.
With the proposition above, the proof of Theorem~\ref{theo:decomp} reduces to a straightforward application of the triangle inequality. First, by Proposition~\ref{lemma:SP_Iterative},
\[
\mathcal{E}(s^i)=L \|s^{i+1}-s^i\|_{\calV}\leq L \|s^{i+1} - s^{i+1}_{h\t}\|_{\calV} + L \|s^{i+1}_{h\t} - s^{i}\|_{\calV}=  \calE_{\mathrm{disc}}^{\rm H}(s^{i+1}_{h\t}) + L\calE_{\mathrm{lin}}^i.
\]
This is nothing but \eqref{reliability_bound}. The proof of \eqref{eq:efficiency} is similar, starting from $\calE_{\mathrm{disc}}^{\rm H}(s^{i+1}_{h\t})
= L\|s^{i+1}-s^{i+1}_{h\t}\|_{\calV}$. The triangle inequality applied to $L\|s^{i+1}_{h\t}-s^{i}\|_{\calV}$ reveals the other direction.\qed \\[2pt]

We use now the result in \Cref{subsec:decomposition} for establishing a posteriori estimates for a space time discretization of Problem~\ref{pb:2} and ultimately, to develop an adaptive, linear iterative scheme for Problem~\ref{eq:reformulated_weak_form}..

\section{Space-time Finite Elements, A Posteriori Analysis, and Adaptivity}\label{sec:apost}
For the finite element discretization of Problem~\ref{pb:2}, we consider a regular shape partition of $Q$ into  $\mathcal{T}_{h\tau}$, a family of simplicial space-time elements of ( $d+1$)-dimensions, that are uniformly shape-regular. We denote by $(\boldsymbol{x},t) \in \mathbb{R}^{d+1}$ a generic space-time point, where $\boldsymbol{x} \in \Omega$ is the spatial component and $t\in(0,T]$ the temporal component.  Let $\mathcal{N}_{h\tau}$ denote the set of vertices of the triangulation $\mathcal{T}_{h\tau}$, and $\mathfrak{F}_\htt$ the set of faces. For  $a\in \mathcal{N}_{h\t}$, let $\calT_\htt^a$ denote the set of elements that have $a$ as a vertex, $\phi_a$ denote the corresponding hat function (P1-Lagrange element basis), and $\oma$ the space-time patch, i.e., support of $\phi_a$. Let $h_K$ be the diameter of element $K$, and $h_\t:=\max_{K\in \calT_{h\t}} h_K>0$ denote the mesh-size of $\calT_\htt$. We define the broken polynomial space of degree at most $p \ge 0$ by $\mathcal P_p(\mathcal T_{h\tau}):=\{ v \in L^2(Q)|\; v|_K \in \mathcal P_p(K), \; \text{for all}\; K \in \mathcal T_{h\tau} \},$ where $\mathcal P_p(K)$ denotes the space of polynomials of total degree at most $p$ defined on $K$. Then, the conforming space-time finite element space is defined as
\begin{align}\label{eq:def_Vht}
    \calV_{h\t}:= \mathcal{P}_1(\calT_{h\t}) \cap H^1(Q)\subset \calS\cap \calV. 
\end{align}
The elementwise $L^2$-orthogonal projection $\Pi^p_{h\tau} : L^2(Q) \to \mathcal{P}_p(\mathcal{T}_{h\tau})$ is defined by
\begin{equation}\label{eq:L2_projection}
\int_Q (\Pi^p_{h\tau} v)\, q 
=
\int_Q v\, q 
\qquad \text{for all}\quad q \in \mathcal P_p(\calT_{h\t}).
\end{equation}
and $\bm{\Pi}^p_{h\t}: L^2(Q)^d\to \calP_p(\calT_{h\t})^d$ be the component-wise projection for vector-valued functions.
\subsection{Space-time finite elements}

\begin{theorem}[Space-time finite elements]\label{theo:lsp}
Let $s^0\in (\calS\cap \calV)$ be given and $\{s^{i}\}_{i\in\N}\subset \calV$ the sequence of solution to the $L$-scheme, obtained by solving Problem \ref{pb:2}.   Then, there exists a unique sequence of finite element solutions $\{s_{h\t}^{i}\}_{i\in \N}\subset \mathcal V_{h\tau}\subset (\calS\cap \calV)$ that satisfies $s^i_{h\t}(0)= \Pi^1_{h\t} s_0$, $s^{0}_{h\t}= \Pi^1_{h\t} s^0$, and
  \begin{align*}
      \int_0^T \!\left\langle \partial_t \left( L(s^{i+1}_{h\t} - s^i_{h\t}) + b(s^i_{h\t}) \right), \varphi_{h\t} \right\rangle \,  
+ \int_0^T \!\left( \nabla \left( L(s^{i+1}_{h\t} - s^i_{h\t}) + B(s^i_{h\t}) \right), \nabla \varphi_{h\t} \right) 
= \int_0^T \!( f, \varphi_{h\t} )
  \end{align*}
{for all $\f_{h\t}\in \mathcal V_{h\tau}$, or equivalently, in terms of the linear functional defined in \eqref{eq:defRlin},
\begin{align}\label{eq:lsp}
\calR_{\rm lin}(s^{i+1}_{h\t},s^i_{h\t}) = 0 .
\end{align}}
Further, let $i\in \N$ be fixed and $\bar{s}^{i+1}\in \calV$ solve $\calR_{\rm lin}(\bar{s}^{i+1},s^i_{h\t})=0$ with $\bar{s}^{i+1}(0)=s_0$. Then, 
\begin{equation}\label{eq:apriori_est_lsp}
    \|s^{i+1}_{h\t}-\bar{s}^{i+1}\|_{L^2(0,T;H^1_0(\Om))}\lesssim \inf_{\f_{h\t}\in \calV_{h\t}}\|\bar{s}^{i+1}-\f_{h\t}\|_\calV \to 0 \qquad \text{ as } \qquad h_\t \to 0.
\end{equation}
\end{theorem}

\begin{proof}
From \eqref{eq:defRlin} and \eqref{eq:lsp} we get $ L\resH(s^{i+1}_{h\t})(\f_{h\t})= L\resH(s^{i}_{h\t})(\f_{h\t}) -\calR(s^{i}_{h\t})(\f_{h\t}) $ for all $\f_{h\t}\in \calV_{h\t}$. Observe that $s^{i}_{h\t}\in \calV_{h\t}\subset (\calS\cap \calV)$ implies that $L\resH(s^{i}_{h\t}) -\calR(s^{i}_{h\t})\in L^2(0,T;H^{-1}(\Om))$. Thus, from Theorem 3.2 of \cite{steinbach2015space}, given $s_{h\tau}^i\in \calV_{h\t}$, there exists a unique $s^{i+1}_{h\t}\in \calV_{h\t}\subset (\calS\cap \calV)$, such that \eqref{eq:lsp} holds true, and one proceeds with the inductive logic. The estimate \eqref{eq:apriori_est_lsp} follows from Theorem 3.2 of \cite{steinbach2015space}.
\end{proof}
\begin{remark}
Particularly in the degenerate cases, we do not expect any higher regularity of $\bar{s}^{i+1}\in \calV$. Therefore, an explicit order of convergence with respect to $h_\t$ cannot be provided. As such, \eqref{eq:apriori_est_lsp} shows consistency, but is less useful for the main goal of this work, namely the a posteriori analysis and the development of an adaptive scheme.

\end{remark}

\subsection{A posteriori estimates}\label{sec:a-posteriori}
There are several works developing reliable and efficient estimators for nonlinear diffusion equations. We refer to \cite{dolejvsi2024error,congreve2026residual,DolejsiRoskovec2021} for discontinuous Galerkin space-time methods (on either mesh-dependent norms, or efficiency dependent on the nonlinearities), and \cite{ErnSmearsVohralik2017,ErnSmearsVohralik2019} for conforming finite elements. In these works, the local adaptation on time is not addressed. However, for problems with local, evolving features such as free boundaries, an efficient discretization scheme should employ space-time adaptivity. Here, we introduce an \emph{easy-to-implement, fully and parallelly computable a posteriori estimator that works in combination with linear iterations, specifically the L-scheme or the regularized Newton and is robust with respect to nonlinearities/degeneracies}. We will also show the reliability of this estimator, which will mainly be used in the adaptive algorithm. For this purpose, an equilibrated flux $\bm{\sigma}_h$ needs to be constructed, which is mass-conservative. 

\subsubsection{Space of fluxes}
Concerning fluxes, we start by noting that, given $\nhat\in \R^d$ and $e\in\mathbb{R}$, then $t+\nhat\cdot \bm{x}=e$ is the equation of a hyperplane inside $Q'$. For any $\bm{v}\in H(\mathrm{div};Q')$, $\bm{v}\cdot \nhat$ remains conforming accross $\mathcal{H}$. However, it does not include any time component, and thus, does not require conformity in the time-direction.

\begin{definition}[Raviart-Thomas spaces for space-time problems]\label{def:Qhtspace}
For a triangulation $\calT_\htt$, its vertex $a\in \calN_\htt$ with the corresponding space-time patch $\omega_a$, and a polynomial degree $p\geq 2$, we define the spaces
\[
\calQ_\htt:= H(\mathrm{div},Q)\cap \calP_p(\calT_\htt),\qquad  \calQ^a_\htt:=\{\boldsymbol{v}^a_h \in H(\mathrm{div},\oma)\cap \calP_p(\calT_\htt^a) \;|\; \bm{v}^a_h\cdot \nhat|_{\p\oma \setminus \Gamma_T}=0 \}.
\]
In the above, $\Gamma_T:=\p \Om \times [0,T]$ and $\nhat$ denotes the restriction to an outward normal of $\p\oma$ in the space direction. As above, the equation of a face $E$ that is part of $\p\oma$ is represented by $t+ \nhat\cdot \x=e$, where $t$ is the time component of $a$ and $e\in \mathbb{R}$ suitably chosen.
\end{definition}

The non-emptyness of the spaces will follow from \Cref{mass_balance_property}.

\subsubsection{Equilibrated flux}\label{sec:equilibrated_fluxes}
For the construction of the equilibrated flux, we rewrite \eqref{eq:lsp} as 
\begin{align}\label{eq:source_flx}
 \int_0^T ( \Shti, \varphi_{h\tau} ) \, dt
 + \int_0^T ( \Fhti, \nabla \varphi_{h\tau} ) \, dt = 0,
 \quad \;\text{for all}\quad\; \varphi_{h\tau} \in \mathcal{V}_{h\tau},
\end{align}
where
\begin{align}\label{eq:source_flx_Lscheme}
    \Shti:= \partial_t \left( L(s^{i+1}_{h\t} - s^i_{h\t}) + b(s^i_{h\t}) \right), \qquad\;\text{and}\; \qquad \Fhti:= \nabla \left( L(s^{i+1}_{h\t} - s^i_{h\t}) + B(s^i_{h\t}) \right).
\end{align}
Then, we can define the equilibrated flux.
\begin{definition}[Equilibrated flux $\flx$]\label{def:equilibrated}
    Let \eqref{eq:source_flx} hold. For a vertex $a\in \calN_\htt$, let the local flux $\flxa\in \calQ_\htt^a$ be defined by
    \begin{align}\label{eq:optimization}
        \flxa:= \underset{\substack{\bm{v}^a_h\in\calQ_\htt^a\\
       \nabla\cdot \bm{v}^a_h = -\Pi_h^0 (\phi_a \Shti) - \nabla\phi_a \cdot \bm{\Pi_h^0} \Fhti}}{\mathrm{argmin}} \left\| \bm{v}^a_h + \phi_a \Fhti\right\|_{L^2(\oma)}
    \end{align}
    Here $\phi_a$ is the $P_1$-lagrange element basis satisfying $\phi_a(a)=1$. After extending $\flxa$ by zero to $Q\setminus \oma$, the equilibrated flux $\flx\in\calQ_\htt$ is defined as
    \[
    \flx:=\sum_{a\in \calN_\htt} \flxa.
    \]
\end{definition}
\begin{proposition}[Well-posedness and mass balance property of $\flx$]\label{mass_balance_property}
    The equilibrated flux $\flx \in \calQ_\htt$ is well defined and satisfies 
    \[
    \nabla\cdot \flx= -\Pi_h^0(\Shti) \qquad \text{ in all } K\in \calT_\htt.
    \]
\end{proposition}
\begin{proof}
First, by counting degrees of freedom for a given $a\in \calN_\htt$ we show that there are indeed $\bm{v}_h^a\in \calQ_\htt^a $ such that $\nabla\cdot \bm{v}^a_h = -\Pi_h^0 (\phi_a \Shti) - \nabla\phi_a \cdot \bm{\Pi_h^0} \Fhti\in \calP_0(\calT_\htt)$. \textcolor{black}{For a given vertex $a \in \mathcal{N}_{h\tau}$ such that $a \notin \Gamma_T$, let $N_a \in \mathbb{N}$ denote the number of elements in $\omega_a$. This gives $N_a$ constraints for satisfying the divergence conditions. Observe that $\partial\omega_a$ contains $N_a$ faces. Since the elements in $\mathcal{T}_{h\tau}$ are $(d+1)$-simplices, each of the remaining $(d+1)N_a$ faces is shared by two adjacent elements, giving in total $((d+1)N_a/2)$ internal faces}. 
For $\bm{v}_h^a\in \calQ_\htt^a$, the trace $\bm{v}^a_h\cdot \nhat$ on each face $E$ gives a polynomial of degree $p$ with $d$ variables (since $t=e-\nhat\cdot \x$ can be substituted as the equation of the face). This implies at most $C_p^{p+d}$ unknown coefficients per  face, thus $((d+1)/2 +1) N_a C_p^{p+d}$ coefficients in total. Now, the total degrees of freedom of $\bm{v}_h^a$ (which is $d$-dimensional) per element are $dC^{p+d+1}_p$. The total number of degrees of freedom needs to exceed the number demanded by constraints, which yields the relation
\[
dC^{p+d+1}_p N_a \geq \left[\left(\frac{d+1}{2} +1\right) C_p^{p+d} + 1\right]N_a,
\]
or, after cancelling terms,
\[
\frac{d(p+d+1)}{d+1}\geq \frac{d+3}{2} + \frac{1}{C_p^{p+d}}.
\]
Rearranging, we get the condition \[p\geq \frac{1}{2}(3-d)(d+1) + \frac{d+1}{d C_p^{p+d}}\]
which is satisfied for all $p\geq 2$ if $d\geq 2$. For $d=1$, this restriction is actually an overcount since $\bm{v}_h^a$ is scalar, and thus $\bm{v}^a_h\cdot \nhat=0$ implies $\bm{v}_h^a=0$. This condition has to be enforced in the boundary nodes of $\oma$, see \Cref{fig:P2nodes}. The value at the internal nodes, $(N_a+1)$ in number, can be freely chosen to satisfy $N_a$ divergence constraints (actually one less due to the compatibility condition). When $a\in \calN_\htt$ is on the boundary $\Gamma_T$, then the restrictions are even less, since no conditions are required on $\p\oma \cap \Gamma$. Thus, for $p\geq 2$ the space $\calQ_\htt^a$  has the required degrees of freedom to satisfy conditions for $\bm{v}_h^a$ in \eqref{eq:optimization}.

    \begin{figure}[h]
                \centering
                    \includegraphics[width=.4\linewidth]{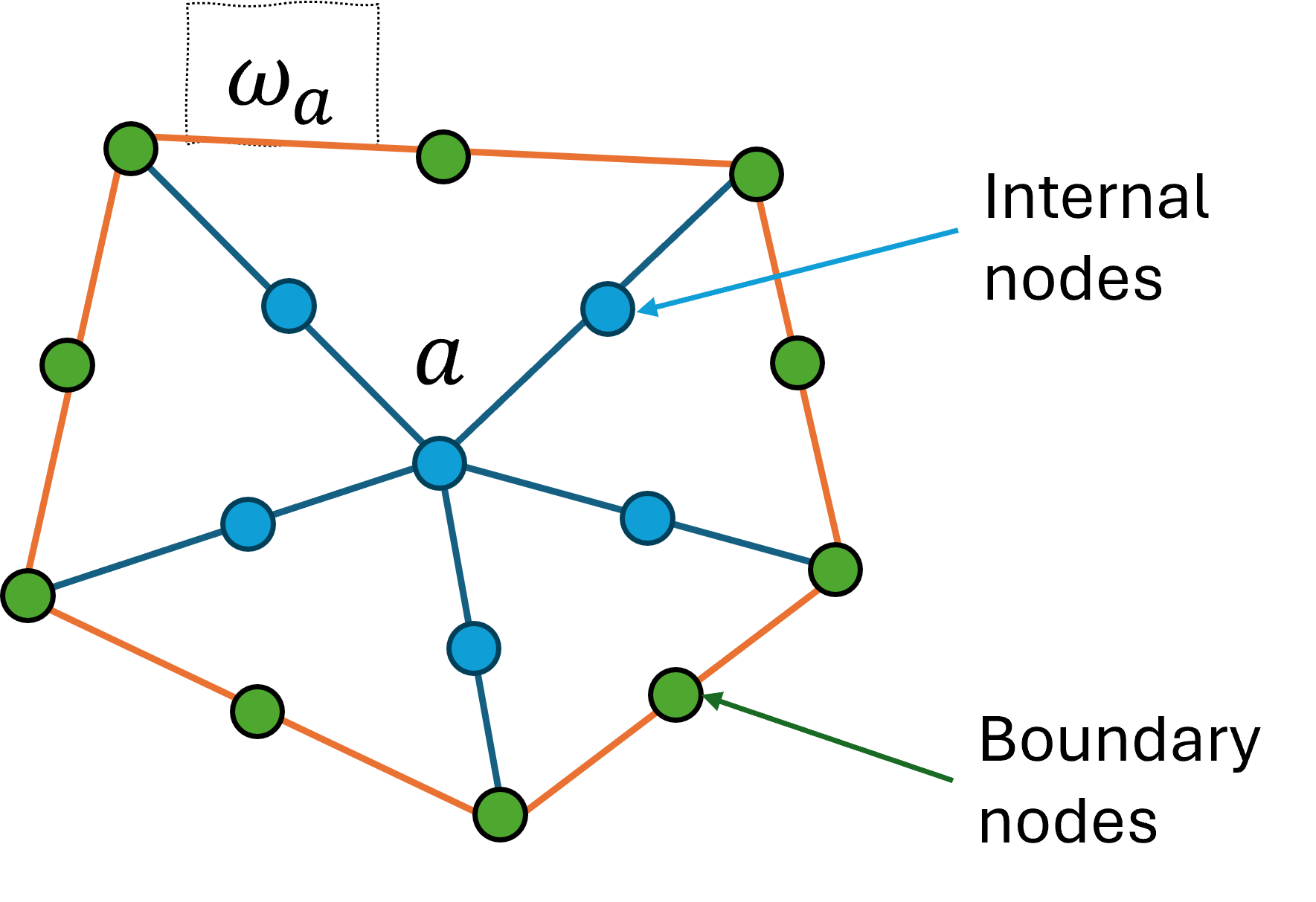}
                    \caption{For one space dimension ($d=1$), the  $P2$ nodal points of $\oma$.}
        \label{fig:P2nodes}
    \end{figure}

Next, we verify the compatibility criteria for $\bm{v}^a_h\in \calQ^a_\htt$. Using the divergence theorem on the field $(\bm{v}^a_h,0)$, we have
\begin{align*}
        0&=\int_{\p\oma} \bm{v}^a_h\cdot \nhat= \int_{\oma}  \nabla \cdot \bm{v}^a_h = - \left[\int_{\oma} \Pi_h^0 (\phi_a \Shti) + \nabla\phi_a \cdot \bm{\Pi_h^0} \Fhti\right]\\
        &= -\left[\int_0^T (\Shti, \phi_a ) \, dt
 + \int_0^T ( \Fhti, \nabla \phi_a ) \, dt\right]=0.
\end{align*}
The last two steps follow from the property of the projector $\Pi^0_h$, and from \eqref{eq:source_flx} by inserting $\f_\htt=\phi_a$. The Euler-Lagrange equation corresponding to the constrained minimization problem \eqref{eq:optimization} then becomes: find $(\flxa,\l^a_\htt)\in \calQ_\htt^a \times \calP_1(\calT_a)$ such that 
\begin{equation*}
    \left\{\begin{aligned}
        &(\flxa + \phi_a \Fhti, \bm{v}^a_h)_{L^2(\oma)} + (\l^a_\htt, \nabla \cdot \bm{v}^a_h)_{L^2(\oma)}= 0,\\
        &(l_h^a, \nabla \cdot \flxa)_{L^2(\oma)}= -(l_h^a\phi_a, \Shti)_{L^2(\oma)}- (l_h^a \nabla \phi_a, \Fhti)_{L^2(\oma)},
    \end{aligned}\right.  
\end{equation*}
for all $(\bm{v}^a_h,l_h^a)\in \calQ_\htt^a \times \calP_1(\calT_a)$. The existence-uniqueness of a solution to this system follows from the standard inf-sup condition since it is a saddle-point problem \cite[Chapter 4]{boffi2013mixed}.

Finally, we get by the partition of unity property $\sum_{a\in \calN_{h\t}} \phi_a=1$ that
\begin{align*}
    &\nabla \cdot \flx = \sum_{a\in \calN_{h\t}} \nabla \cdot \flxa= -\sum_{a\in \calN_{h\t}} \left[\,\Pi_h^0 (\phi_a \Shti) + \nabla\phi_a \cdot \bm{\Pi_h^0} \Fhti\right]\\
    &= - \left[\Pi_h^0 \left((\Sigma_{a\in \calN_{h\t}} \phi_a) \Shti\right) + \nabla(\Sigma_{a\in \calN_\htt}\phi_a ) \cdot \bm{\Pi_h^0} \Fhti\right]= - \Pi_h^0 \Shti.
\end{align*}
\end{proof}

\subsubsection{Guaranteed reliability estimate}
The equilibrated fluxes constructed in \Cref{sec:equilibrated_fluxes} are used to provide local estimates for the error. For this, we use the residuals $\calR$  and $\calR_{\rm lin}$ defined in \eqref{eq:residual} and \eqref{eq:defRlin}, respectively. For $\calV_\htt$ defined in \eqref{eq:def_Vht}, let $\{s_{h\t}^{i}\}_{i\in \N}\subset \mathcal V_{h\tau}\subset (\calS\cap \calV)$ be the sequence of finite element solutions introduced in \Cref{theo:lsp}. With the equilibrated flux $\flx\in \calQ_\htt$ given in \Cref{def:equilibrated} with $\calQ_\htt$, $\Shti$, $\Fhti$ from \Cref{def:Qhtspace} and \eqref{eq:source_flx_Lscheme} respectively, we introduce the flux and data oscillation estimators for $Q'\subseteq Q$ as,

\begin{subequations}
   \begin{align}\label{flux_oscillation_estimator}
        \eta_{\mathrm{F},Q'}^i:= \|\flx + \Fhti\|_{L^2(Q')}, \qquad 
    \eta_{\mathrm{osc},Q'}^i:= C_{\,\!_{\Omega}}^{\mathrm{P}} \,h_\Om\,\|(\mathbb{I}-\Pi_h^0) (\p_t b(s^i_\htt) -f)\|_{L^2(Q')} 
\end{align}
along with the discretization and linearization estimators
 \begin{align}
     &\eta_{\mathrm{disc}}^{i}:= \left((\eta_{\mathrm{F},Q}^i + \eta_{\mathrm{osc},Q}^i)^2 + L^2\|s^i_{h\t}(0)-s_0\|^2\right)^{\frac{1}{2}},\\
    &\eta_{\mathrm{lin}}^{i}:= \left([C_{\,\!_{\Omega}}^{\mathrm{P}}\,h_\Om]^2\,\|\p_t (s^{i+1}_\htt - s^i_\htt)\|_{L^2(Q)}^2 + \|\nabla (s^{i+1}_\htt - s^i_\htt)\|_{L^2(Q)}^2 + \|(s^{i+1}_\htt - s^i_{h\t})(T)\|^2\right)^{\frac{1}{2}}.
 \end{align}
 \end{subequations}
 Let $ \eta^i:= \eta_{\mathrm{disc}}^{i} + L\eta_{\mathrm{lin}}^{i}$ be the total estimator. Then we have the following.
\begin{theorem}[Guaranteed a posteriori upper bound on residual]\label{thm:apost-reliability-i}
The errors $\calE$, $\calE_{\mathrm{disc}}^{\rm H}$, and $\calE_{\mathrm{lin}}^{i}$ defined in \eqref{eq:total_indicator}--\eqref{err:linearization} satisfy the reliability estimates
 \begin{align}\label{eq:reliability_estimate}
\mathcal{E}_{\mathrm{disc}}^{\rm H}(s^{i+1}_{_\htt}) \leq \eta_{\mathrm{disc}}^{i}, \qquad \mathcal{E}^i_{\rm lin} \leq \eta^i_{\rm lin},\;\qquad \text{and} \;\qquad \mathcal{E}(s^i_{_\htt}) \leq \eta^i.
 \end{align}
\end{theorem}
We stipulate that proving the efficiency of this estimator (lower bound of $\calE(s^i_{h\t})$ by $\eta_{\mathrm{disc}}^{i}$) might also be possible with minor modifications, see \Cref{remark:efficiency}. However, since our main target is adaptivity, we skip this for brevity.
\begin{proof}
Let $\varphi \in L^2(0,T;H^1_0(\Omega))$ with $\|\nabla\varphi\|_{L^2(Q)} = 1$.
We start from the residual functional \eqref{eq:source_flx},
\[
 \mathcal{R}_{\rm lin}(s^{i+1}_{h\tau}, s^i_{h\tau})(\varphi)
= \int_0^T\!\Big[(\mathcal{S}^i_{h\tau},\varphi) + (\mathbf{F}^i_{h\tau},\nabla\varphi)\Big]\,dt.
\]
The following term will be added to the above. Note that this term vanishes. More precisely, with $\hat{\underline{\bm{n}}}_{t,x}$ being the unit normal on $\p Q$, by the divergence theorem we have
\begin{align*}
    &\int_{Q} \left[\f\nabla \cdot \flx + \nabla \f\,\cdot \flx\right]= \int_{Q}\nabla \cdot (\f\flx) = \int_{\p Q} \f (\flx,0)\cdot \hat{\underline{\bm{n}}}_{t,x}\\
    &= \int_{\p \Om\times [0,T]} \f \,\flx\cdot \hat{\bm{n}}_{x} +  \int_{\Om\times \{0\}} \f \,\flx\cdot (\bm{0},-1) +  \int_{\Om\times \{T\}} \f\, \flx\cdot (\bm{0},1)=0.
\end{align*}
The first term on the right vanishes since  $\varphi\in L^2(0,T;H^1_0(\Om))$. Adding this term yields
\begin{align*}
 \mathcal{R}_{\rm lin}(s^{i+1}_{h\tau}, s^i_{h\tau})(\varphi)
&= \int_0^T\!\Big[(\mathcal{S}^i_{h\tau} + \nabla\cdot\boldsymbol{\sigma}^i_{h\tau},\varphi)
+ (\mathbf{F}^i_{h\tau} + \boldsymbol{\sigma}^i_{h\tau},\nabla\varphi)\Big]\,\dd t.
\end{align*}
By \Cref{mass_balance_property}, $\nabla\cdot\boldsymbol{\sigma}^i_{h\tau} = -\Pi^0_h(\mathcal{S}^i_{h\tau})$ in all $K\in\mathcal{T}_{h\tau}$, so that
\[
\mathcal{S}^i_{h\tau} + \nabla\cdot\boldsymbol{\sigma}^i_{h\tau}
= (\mathbb{I}-\Pi^0_h)\mathcal{S}^i_{h\tau}.
\]
Recalling \eqref{eq:source_flx_Lscheme}, we expand
\[
(\mathbb{I}-\Pi^0_h)\mathcal{S}^i_{h\tau}
= (\mathbb{I}-\Pi^0_h)\partial_t\!\left(L(s^{i+1}_{h\tau}-s^i_{h\tau})\right)
+ (\mathbb{I}-\Pi^0_h)\left(\partial_t b(s^i_{h\tau})-f\right).
\]
Since $s^{i+1}_{h\tau}, s^i_{h\tau}\in\mathcal{V}_{h\tau}\subset\mathcal{P}_1(\mathcal{T}_{h\tau})$, the term $\partial_t(L(s^{i+1}_{h\tau}-s^i_{h\tau}))$ is piecewise constant on $\mathcal{T}_{h\tau}$ and thus, $(\mathbb{I}-\Pi^0_h)\partial_t\!\left(L(s^{i+1}_{h\tau}-s^i_{h\tau})\right) = 0$.
Substituting into the residual identity, we obtain
\begin{align}\label{eq:res-identity}
\mathcal{R}_{\rm lin}(s^{i+1}_{h\tau}, s^i_{h\tau})(\varphi)
&= \int_0^T\!\Big[((\mathbb{I}-\Pi^0_h)(\partial_t b(s^i_{h\tau})-f),\varphi)
+ (\mathbf{F}^i_{h\tau} + \boldsymbol{\sigma}^i_{h\tau},\nabla\varphi)\Big]\,dt.
\end{align}
Applying Cauchy-Schwarz to each term of \eqref{eq:res-identity},
\begin{align*}
\left| \mathcal{R}_{\rm lin}(\varphi)\right|
&\leq \|(\mathbb{I}-\Pi^0_h)(\partial_t b(s^i_{h\tau})-f)\|_{L^2(Q)}\,
\|\varphi\|_{L^2(Q)}+\|\mathbf{F}^i_{h\tau}+\boldsymbol{\sigma}^i_{h\tau}\|_{L^2(Q)}\,
\|\nabla\varphi\|_{L^2(Q)}.
\end{align*}
Since $\varphi(\cdot,t)\in H^1_0(\Omega)$ a.e.\ $t\in [0,T]$ and $\|\nabla \f\|_{L^2(Q)}=1$, using the Poincar\'e-Friedrichs inequality gives 
\begin{align}
    \|\varphi(t)\|_{L^2(Q)}\leq C_{\,\!_{\Omega}}^{\mathrm{P}}\,h_\Omega\, \|\nabla\varphi(t)\|_{L^2(Q)}=C_{\,\!_{\Omega}}^{\mathrm{P}}\, h_\Omega.
\end{align}
Therefore
\begin{align}\label{eq:Hminus1-R}
&\|\mathcal{R}_{\rm lin}(s^{i+1}_{h\tau},s^i_{h\tau})\|_{L^2(0,T;H^{-1}(\Omega))} :=\sup_{\|\nabla \f\|_{L^2(Q)}=1}\left| \mathcal{R}_{\rm lin}(\varphi)\right|\nonumber
\\
&\leq \Big[C_{\,\!_{\Omega}}^{\mathrm{P}}\,h_\Omega\,
\|(\mathbb{I}-\Pi^0_h)(\partial_t b(s^i_{h\tau})-f)\|_{L^2(Q)}
+ \|\mathbf{F}^i_{h\tau}+\boldsymbol{\sigma}^i_{h\tau}\|_{L^2(Q)}\Big]\overset{\eqref{flux_oscillation_estimator}}{=}\;
\eta^i_{\rm osc} \;+\; \eta^i_{\rm F}.
\end{align}
It remains to bound the total error. From the reliability bound 
\eqref{reliability_bound} of \Cref{theo:decomp},
\[
\mathcal{E}(s^i_{h\tau}) 
\leq \mathcal{E}^{\rm H}_{\rm disc}(s^{i+1}_{h\tau}) + L\,\mathcal{E}^i_{\rm lin}.
\]
The first term, defined in \eqref{eq:disc_error}, is bounded by \eqref{eq:Hminus1-R}, giving $\mathcal{E}^{\rm H}_{\rm disc}(s^{i+1}_{h\tau})\leq \eta^i_{\rm disc}$.
Applying \eqref{eq:defHinv_norm} we have $\|\p_t(s^{i+1}_{h\tau}-s^i_{h\tau})\|_{L^2(0,T;H^{-1}(\Om))}\leq C_{\,\!_{\Omega}}^{\mathrm{P}}\, h_\Om\|\p_t(s^{i+1}_{h\tau}-s^i_{h\tau})\|_{L^2(Q)}$. This implies for the linearization error by \eqref{err:linearization} that
\[\mathcal{E}^i_{\rm lin} = \|s^{i+1}_{h\tau}-s^i_{h\tau}\|_{\mathcal{V}} \leq  \eta^i_{\rm lin}.\]
Combining these two bounds, we obtain the estimate \eqref{eq:reliability_estimate}.
\end{proof}

\begin{remark}[Effectivity of the estimator $\eta^i$] The estimator $\eta^i$ is fully and parallelly computable, and does not involve any unknown constants (Poincar\'e constant is generally known, e.g., for convex domains, this is $C_{\,\!_{\Omega}}^{\mathrm{P}}=1/\pi$, \cite{payne1960optimal}). 
However, generally the flux estimator $\eta_{\mathrm{F}}^i$ is much greater than $\eta_{\mathrm{osc}}^i$ (unless the process is multiscale), particularly since $\p_t b(s^i_\htt)|_K=\p_t s^i_\htt|_K \in \calP_0(K)$ if $s^i_\htt\leq u^*$ in $K\in \calT_\htt$ from \Cref{lemma:b_B_construction}. Moreover, the linearization estimator $\eta_{\rm lin}^i$ vanishes as $i\to 0$. Thus, overall the estimators are expected to be quite accurate. This is verified in \Cref{fig:true_error_vs_estimator}, which shows that the effectivity index (error/estimator) is of order 1 in the non-degenerate regions ($s>0$) for the porous medium equation. However, the figure also reveals that proving local efficiency of the estimator is practically impossible at the degenerate regions where $s=0$, since the error is essentially 0. It is also possible to increase accuracy of the estimator further by taking higher degree projections $\Pi^p_h$ by increasing the degree $p$ of $\calQ_\htt$ in \Cref{def:Qhtspace}. 
\end{remark}

\begin{remark}[Global Poincaré constants in \Cref{theo:lsp} and efficiency estimate]\label{remark:efficiency} The use of the global Poincarè constant becomes necessary for the term $\int_0^T ((\mathbb{I}-\Pi_h^0) (\p_t b(s^i_\htt) - f),\f)$ since $\f\in L^2(0,T;H^1_0(\Om))$ only. However, if $\f\in H^1(Q)$, then we could have estimated this term by 
\begin{align*}
    &\int_0^T ((\mathbb{I}-\Pi_h^0) (\p_t b(s^i_\htt) - f),\f)=\sum_{K\in \calT_\htt} \left((\mathbb{I}-\Pi_h^0) (\p_t b(s^i_\htt) - f),\f-\tfrac{1}{|K|}\smallint_K \f \right)_K\\
    &\leq \sum_{K\in \calT_\htt}\|(\mathbb{I}-\Pi_h^0) (\p_t b(s^i_\htt) - f)\|_{L^2(K)}\left\|\f-\tfrac{1}{|K|}\smallint_K \f\right\|_{L^2(K)}\\
    &\leq \sum_{K\in \calT_\htt} \frac{h_K}{\pi} \|(\mathbb{I}-\Pi_h^0) (\p_t b(s^i_\htt) - f)\|_{L^2(K)}\left\|\f\right\|_{H^1(K)} \leq [\hat{\eta}_{\mathrm{osc}}^i] \|\f\|_{H^1(Q)}
\end{align*}
where $\hat{\eta}_{\mathrm{osc}}^i:=\left(\sum_{K\in \calT_\htt} \frac{h_K^2}{\pi^2}\|(\mathbb{I}-\Pi_h^0) (\p_t b(s^i_\htt) - f)\|_{L^2(K)}^2\right)^{\frac{1}{2}} $.
This could be useful in obtaining an global efficiency bound on $\calE(s^i_\htt)$ since $\|\mathcal{R}_{\rm lin}\|_{L^2(0,T;H^{-1}(\Omega))} \geq \|\mathcal{R}_{\rm lin}\|_{H^1(Q)^*}$.
\end{remark}

\subsubsection{Space-time adaptive algorithm}
We outline an adaptive solution strategy based on the estimators derived in \Cref{thm:apost-reliability-i}, and the $\texttt{SOLVE} \;\longrightarrow\; \texttt{ESTIMATE} \;\longrightarrow\; \texttt{MARK} \;\longrightarrow\; \texttt{REFINE}$ principle \cite{carstensen2014axioms}.
On a given mesh $\mathcal{T}_{h\tau}^\ell$ ($\ell\in \N$), the nonlinear problem is solved by the L-scheme linearization (Problem \ref{pb:2}). Once the linearization error $\eta^i_{\rm lin}$   is below a prescribed tolerance, the equilibrated flux reconstruction $\boldsymbol{\sigma}^i_{h\tau}$ (see \Cref{def:equilibrated}) is performed, yielding discretization estimator $\eta^i_{\mathrm{disc}}$ given in \eqref{flux_oscillation_estimator}. If the linearization error $\eta^i_{\rm lin}$ is smaller than a fraction $\d_{\mathrm{lin}}\ll 1$ of the discretization error $\eta^i_{\mathrm{disc}}$, then a refinement step is done.
The element-wise error indicators $\eta^i_K$ for each element $K \in \mathcal{T}_{h\tau}$ are computed, based on \Cref{thm:apost-reliability-i}, comprising of flux, data oscillation, and initial condition contributions,
\begin{align}
    \eta^i_K := \left(2\left([\eta^i_{\mathrm{osc},K}]^2 + [\eta^i_{\mathrm{F},K}]^2\right) + L^2 \int_{K\cap (\{0\}\times \Om)} |s^i_{\htt}(0)-s_0|^2 \right)^{\frac{1}{2}} \text{ such that } \sum_{K\in \calT_\htt} [\eta^i_K]^2\geq [\eta^i_{\mathrm{disc}}]^2. \label{eq:etaK}
\end{align}
The elements that contribute the most to the total error are marked for refinement using the D\"{o}rfler criterion. The marked elements are refined using vertex bisection. After projecting to the new mesh $\calT^{\ell+1}_\htt$, the linear iterations are restarted. The resulting adaptive loop is summarized in Algorithm~\ref{alg:Lscheme-adapt}.

To evaluate the performance of the adaptive strategy, along with the adaptive algorithm, we also investigate a \emph{uniform refinement} strategy. In this case, a uniform refinement is done by a factor $\mathrm{r}_\mathrm{ref}=2$ if the adaptive criterion in Algorithm~\ref{alg:Lscheme-adapt},
\[\eta_{\mathrm{lin}}^i\le \delta_{\mathrm{lin}}\eta^i_{\mathrm{disc}},\]
is met. As evaluation criteria, we consider both the number of iterations, as well as the degree of freedom for each mesh. More precisely, for any mesh-level $\ell\in \N$, let $\overline{i}(\ell)\in \N$ be the iteration at which the adaptive criterion is met both for the adaptive and the uniform strategies. Letting $|\mathrm{DOF}|_\ell$ denote the degrees of freedom of the corresponding finite element space on $\calT^\ell_\htt$, to assess the adaptive strategy, we consider the total cost, respectively the total number of iterations. For the $i^{\mathrm{th}}$ iteration at the $\ell^{\rm th} \text{level}$ we define
\begin{equation}
  \label{eq:cost}\mathrm{Cost}(i,\ell)
:= \sum_{\ell'=0}^{\ell-1}\sum_{i'=1}^{\overline{i}(\ell')}|\mathrm{DOFs}|_{\ell'} + \sum_{i'=1}^i |\mathrm{DOFs}|_{\ell},\qquad \mathrm{IterNo}(i,\ell):=\sum_{\ell'=0}^{\ell-1} \overline{i}(\ell') + i.
\end{equation}

\textcolor{black}{Here, $\mathrm{Cost}(i,\ell)$ represents the accumulated degrees of freedom over all iterations at previous mesh levels $\ell'=0,\ldots,\ell-1$ until the stopping criterion is met, plus those at the current level $\ell$ up to iteration $i$. Similarly, $\mathrm{IterNo}(i,\ell)$ is the total number of iterations performed across all mesh levels through iteration $i$ at level $\ell$.}

\begin{algorithm}[h!]
\caption{\vspace{0.4em}Iterative scheme with equilibrated--flux a posteriori adaptivity\vspace{0.5em}}
\label{alg:Lscheme-adapt}
\begin{minipage}{\linewidth}
\begin{algorithmic}[1]
\Require initial mesh $\mathcal{T}_{h\tau}^0$, initial guess $s^0$, and parameters: $L>1$, tolerance $\texttt{Tol}\ll 1$, maximum iteration $i_{\max}$, maximum level $\ell_{\max}$, D\"orfler parameter 
$0<\texttt{$\theta$}<1$, refinement factor $\texttt{$\mathrm{r}_\mathrm{ref}$}>1$, balance factor
$0<\texttt{$\delta_{\mathrm{lin}}$}<1$, 
and an entry threshold 
$0<\texttt{$\varepsilon_{\mathrm{adapt}}$}\le 0.1$.
\State \textbf{set} $\texttt{Error}=1,\;\;$ $i=0,\;\;$, $\ell=0$
\While{$i<\texttt{$i_{\texttt{max}}$}\;\;$ or $\;\;\texttt{Error}\ge\texttt{Tol}\;\;$ or  $\;\;\ell<\ell_{\max}$\;\;} \Comment{nonlinear (L scheme) loop}
  \State Solve linearized problem for $s^{i+1}_\htt$ on $\mathcal T_{h\tau}^\ell$ \hfill \Comment{(\texttt{SpaceTime} solve)(see~\eqref{eq:linearization_gen})}
  \State Compute $\eta_{\mathrm{lin}}^i$ \Comment{linearization error}

  \If{$\eta^i_{\mathrm{lin}}<\e_{\mathrm{adapt}}$}\hfill  \Comment{enter adaptive (refinement) loop}
      \State Reconstruct $\boldsymbol{\sigma}^{\,i}_{h\tau}$ 
      \hfill (local patch problem, see \Cref{def:equilibrated})
     \State Compute $\eta^i_K$, $\eta^i_{\mathrm{disc}}$, and $\eta^i\mapsto \texttt{Error}$
      \hfill (see~\Cref{thm:apost-reliability-i},\eqref{eq:etaK})
     \If{$\eta_{\mathrm{lin}}^i\le \delta_{\mathrm{lin}}\cdot\eta^i_{\mathrm{disc}}$} 
     \Comment{Mesh adaptivity gate}
      \State \textbf{Dörfler mark:} identify minimal $\mathcal{M}\subset\mathcal{T}_{h\tau}^\ell$ with
       $\sum_{K\in\mathcal M}(\eta^i_K)^2 \ge \theta^2\sum_{K\in\mathcal{T}_{h\tau}^\ell}(\eta^i_K)^2$
        \State \textbf{Refine} marked elements 
               $K\in\mathcal{M}$ by setting 
               $h_K^{\mathrm{new}} = h_K / \texttt{$\mathrm{r}_\mathrm{ref}$}$
        \State \textbf{Update mesh} $\mathcal T_{h\tau}^{\ell+1}\gets\textsc{adaptmesh}(\mathcal T_{h\tau}^\ell,\{h_K^{\mathrm{new}}\}_{\mathcal{M}})$  \Comment{generate new mesh}
         \State \textbf{Post-refinement:} $\ell\gets (\ell+1)$. Project $s^{i+1}_{h\tau}, s^{i}_{h\tau}$ to the new $\calV_\htt^{\ell}$ on $\calT_{h\tau}^\ell$
     \EndIf
  \EndIf
  \State Update $s^{i}\gets s^{i+1}$, $i\gets i+1$  \Comment{next iteration}
  \EndWhile
\end{algorithmic}
\end{minipage}
\end{algorithm}
\noindent
In the following section, we compare the performance of these refinement strategies with respect to $\mathrm{Cost}(i,\ell)$ and $\mathrm{IterNo}(i,\ell)$ across several numerical test cases, demonstrating their robustness and efficiency. For cases when the exact solution $u_{\mathrm{exact}}$ is known, we would define the element-wise error and local effectivity index of the estimator as
\begin{align}\left\{\begin{aligned}
        &\|e\|_K^2:= \|b(s^i_\htt)-u_{\mathrm{exact}}\|_K^2 + \|\nabla (B(s^i_\htt)-\Phi(u_{\mathrm{exact}}))\|_K^2, \\
    &(\text{Effectivity index})_K:= \|e\|_K/\eta_K^i,
\end{aligned}
\right. \qquad \text{ for all } K\in \calT_\htt.\label{eq:effectivity}
\end{align}

\section{Numerical Study of the Adaptive Space-Time Scheme}\label{sec:Numerical_results}

This section illustrates the performance of the space-time linearization scheme combined with the adaptive algorithm of Section~\ref{sec:apost}. As a baseline for comparison, we consider both the L-scheme and a regularized Newton scheme, accompanied by adaptive or uniform refinements.

\subsubsection{A modified Newton ($N_{\mathrm{reg}}$) scheme}
For the numerical tests, we also consider a linear iterative scheme that can be seen as a regularized Newton ($N_{\mathrm{reg}}$) scheme. This provides a baseline for comparisons with the L-scheme. The iteration in the regularized Newton scheme can be interpreted as the solution to the following problem.
\begin{problem}
Let $\epsilon\geq0$ be fixed. For $i\in \N$ and $s^i\in (\calS\cap \calV)$, find $s^{i+1}\in \calV$ solving
\begin{align}\label{eq:M-scheme}
\int_0^T \!\left\langle \partial_t \left( L^i_b(s^{i+1} - s^i) + b(s^i) \right), \varphi \right\rangle \,  
+ \int_0^T \!\left( \nabla \left( L^i_B(s^{i+1} - s^i) + B(s^i) \right), \nabla \varphi \right) 
= \int_0^T \!( f, \varphi ),
\end{align}
\end{problem}
with $L^i_b:= b'(s^i) + \varepsilon$ and $L^i_B:=B'(s^i) +\varepsilon$.

Observe that the $\varepsilon=0$ case corresponds to the Newton scheme. However, below we keep $\e>0$ since, for the degenerate test cases investigated, the $\e=0$ option always diverges. This is either due to $b'$ or $B'$ being 0, which makes the time-derivative or space-derivative terms vanish, and therefore the Jacobian becomes singular.
Note that this regularization is similar to the $M$-scheme analyzed in \cite{mitra2019modified,javed2025robust}

The adaptive estimator $\eta^i$ can also be computed following the analysis in \Cref{sec:a-posteriori}. More precisely, one can change the $L$-scheme in \eqref{eq:source_flx} by considering
\begin{align}\label{eq:ShFh_M}
\Shti=
\p_t(L^i_b(s^{i+1} - s^i) + b(s^i) )-f, \qquad
\Fhti=\nabla (L^i_B (s^{i+1} - s^i) + B(s^i))
\end{align}
taking $L$ constant in \eqref{eq:source_flx_Lscheme}. Note that the proofs of \Cref{theo:converegence,theo:decomp,theo:lsp,thm:apost-reliability-i} do not extend to this scheme, so the convergence of the scheme and the reliability of the estimator cannot be guaranteed. However, it will serve as an important point of comparison.

In the following, we denote by \textbf{Adapt-L} and \textbf{Adapt-$N_{\mathrm{reg}}$} the adaptive mesh refinement strategies for the L and regularized Newton schemes, respectively; by \textbf{Uni-L} and \textbf{Uni-$N_{\mathrm{reg}}$} the uniform mesh refinement for L and regularized Newton schemes; and by \textbf{L (fixed mesh)} and \textbf{$N_{\mathrm{reg}}$ (fixed mesh)}, the L and regularized Newton schemes on a fixed mesh. 


\subsubsection{Numerical setup}\label{subsec:num-setup}
All numerical experiments are implemented in 
\texttt{FreeFem++} \cite{hecht2012new}.
We investigate the problems in the space-time domain $Q=\Omega\times(0,T)$, where $\Omega=[-8,8]^d$ with $d=1,2$ being the space dimension. The discretization parameters are listed in ~\Cref{tab:mesh_params_all}.
Unless stated otherwise, the adaptive algorithm (\Cref{alg:Lscheme-adapt}) uses the parameters $L = 1.0$, $\theta = 0.8$, $\varepsilon_{\mathrm{adapt}} = 0.1$, $r_{\mathrm{ref}} = 4$, $\delta_{\mathrm{lin}} = 0.2$, and $\texttt{Tol}=10^{-8}$. The maximum number of iterations is chosen as $i_{\max}\in\{50,100\}$ depending on the problem. Similarly, the maximum refinement level $\ell_{\max}\in\{8,\ldots,15\}$ varies across the considered test cases. For the regularized Newton scheme, the regularization parameter is fixed to $\varepsilon = 0.1$ throughout all simulations in this section.

\begin{table}[h!]
\centering
\caption{ Mesh and model parameters for space-time experiments on the domain $[-8,8]$ (or $[-8,8]^2$ in 2D) and time interval $[0,1]$.}

\resizebox{\textwidth}{!}{%
\begin{tabular}{llccc}
\toprule
 & \textbf{Mesh} & $n_{\text{space}}$ & $n_{\text{time}}$ & $m$ \\
\midrule

\multirow{3}{*}{1D}
& Adaptive ref. & 16 & 2 
& PME: $2,4,6$; Others: $2$ \\

& Uniform ref. & 16 & 2 
& PME: $2,4,6$; Others: $2$ \\

& Fixed mesh 
& 160 (PME), 80 Others 
& 20 (PME), 10 Others 
& PME: $2,4,6$; Others: $2$\\

\midrule

\multirow{2}{*}{2D}
& Adaptive ref. 
& $16 \times 16$ 
& 4 
& PME: $2,4,6$; Others: $2$ \\

& Fixed mesh 
& $32 \times 32$ 
& 4 
& PME: $2,4,6$; Others: $2$ \\

\bottomrule
\end{tabular}%
}

\label{tab:mesh_params_all}
\end{table}

\vspace{2mm}
In $d$ spatial dimensions ($d=1$ or $d=2$), we employ the Barenblatt solution \cite{vazquez2007porous} at $t=0$ as our initial condition for all test cases. For $m>1$ and for $\nu = \frac{1}{(m-1 + \frac{2}{d})}$, this solution is given by
\begin{align}\label{eq:Barenblatt}
  u_{\,\!_{\rm BB}}(\bm{x},t) = (1 + t)^{-\nu} \left[\max\left(\gamma - \frac{\nu(m-1)|\bm{x}|^2}{2dm(t+1)^{2\nu/d}}, 0\right)\right]^{1/(m-1)}.
\end{align}
Homogeneous Dirichlet boundary conditions are imposed on $\Gamma_T=\partial\Omega\times(0,T]$.
For the porous medium equation with exponents $m$ (see \eqref{eq:PME}), $ u_{\,\!_{\rm BB}}$ is also the exact solution. This enables us to compute the exact error for this case. 

\subsection{Porous medium equation (PME)}\label{sec:porous_medium_equation}
We consider the porous medium equation
\begin{equation}\label{eq:PME}
\partial_t u = \Delta u^m \quad \text{ for }\quad  m>1
\end{equation}
which corresponds to $\Phi(u)=u^m$. It shows degeneracy at $u = 0$ since $\Phi'$ vanishes. Following the $b$--$B$ decomposition of \Cref{lemma:b_B_construction} (introduced in \cite{javed2025robust}), we reformulate the PME as $\partial_t b(s) - \Delta B(s) = 0$. Let $u^*>0$ satisfy $\Phi'(u^*)=1$, that is, $u^*=\left(\frac{1}{m}\right)^{\frac{1}{m-1}}$. Then the functions $b$ and $B$ admit the explicit representations
\begin{align}\label{eq:bB_PME}
b(s)=
\begin{cases}
0, & s<0,\\[4pt]
s, & 0\le s\le u^*,\\[4pt]
\bigl(s-u^*+(u^*)^m\bigr)^{1/m}, & s>u^*,
\end{cases}
\qquad
B(s)=
\begin{cases}
s, & s<0,\\[4pt]
s^m, & 0\le s\le u^*,\\[4pt]
s-u^*+(u^*)^m, & s>u^*.
\end{cases}
\end{align}
 Since $u_{\,\!_{\rm BB}}$, the exact solution of the porous medium equation is known, we can determine $s$ exactly, and use it to assess the quality of the scheme.
\subsubsection{1D PME}\label{sec:1DPME_1DST}


\begin{figure}[h]
\centering


  \includegraphics[width=0.75\linewidth]{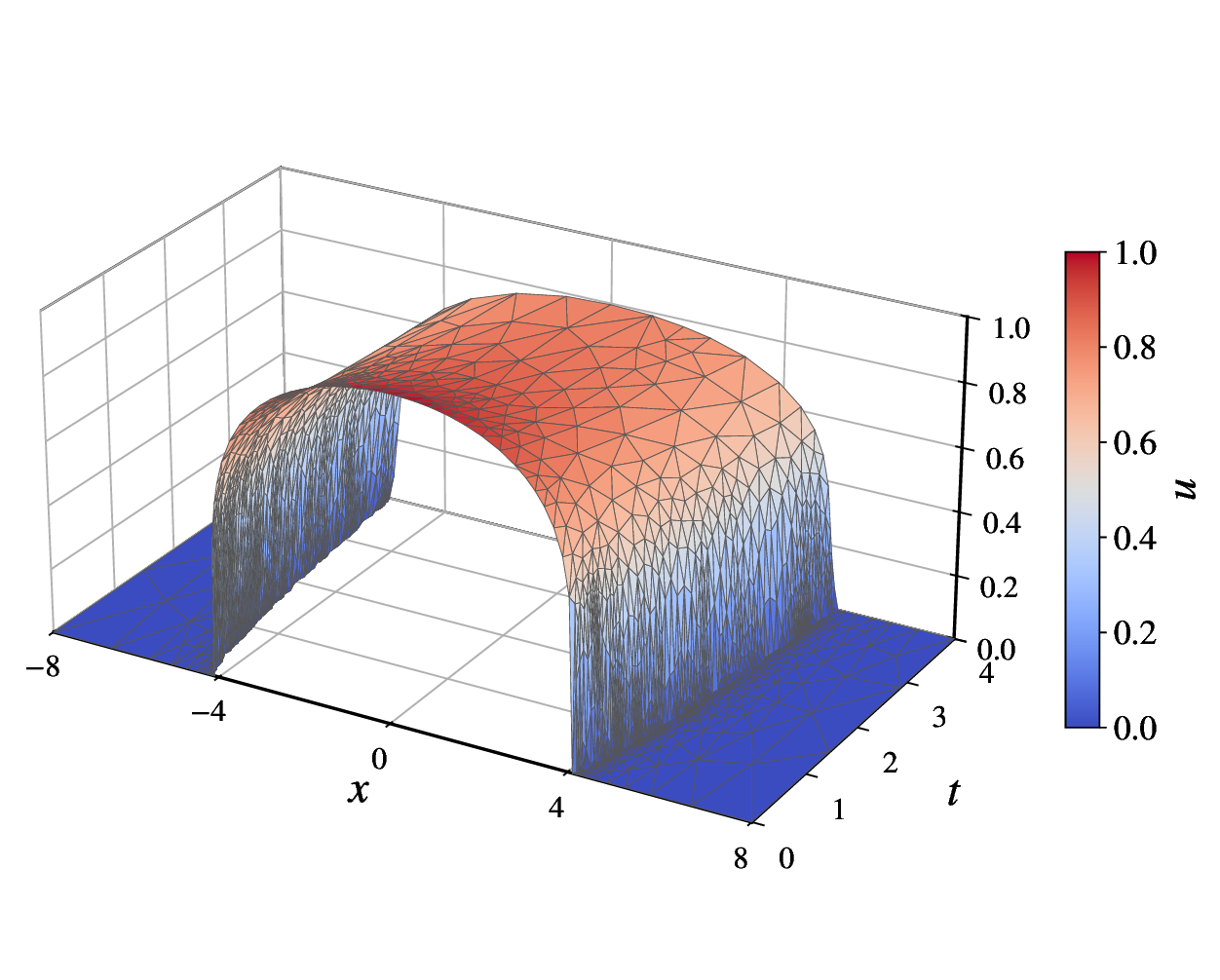}
\caption{[Section~\ref{sec:1DPME_1DST}, $m=6$, $\gamma=1.0$, $T=4$] Three-dimensional view of the numerical solution $u = b(s)$ of the porous medium equation on the final adapted space-time mesh.}
\label{fig:mesh_refinement_summary_1d}
\end{figure}
Figure~\ref{fig:mesh_refinement_summary_1d} illustrates the behaviour of the solution for the 1D PME evaluated using the adaptive refinement strategy. It is seen that the estimator $\eta^i_K$ of \eqref{eq:etaK} drives refinement in regions where the spatial gradients and temporal variations are high, i.e., at the free boundaries and close to $t=0$. 
This is further confirmed in Figure~\ref{fig:true_error_vs_estimator}, which compares the element-wise true error (Left) and the effectivity indices (Right) introduced in \eqref{eq:effectivity} over successive adaptive refinement levels. It is seen that errors are concentrated near the free boundary with much lower concentration at the degenerate $s=0$ region. The refinement happens accordingly, localized to the free-boundary and the initial condition regions. The errors decrease systematically as the mesh is refined. The effectivity indices are of the order of 1 in the region $\{s>0\}$ including the free-boundaries. However, in the region $s=0$ they become very small since the errors $\|e\|_K$ are practically zero.

\noindent
\noindent
\begin{center}
\setlength{\tabcolsep}{2pt}
\setlength{\LTpre}{0pt}
\setlength{\LTpost}{0pt}
\begin{longtable}{cc}

 \multicolumn{1}{c}{\textbf{$\log_{10}(\|e\|_K)$ -- Element-wise true error}} &
\multicolumn{1}{c}{\textbf{$\log_{10}(\textit{Eff. Index})$ -- Effectivity index}}\\[4pt]


\includegraphics[width=0.49\textwidth]{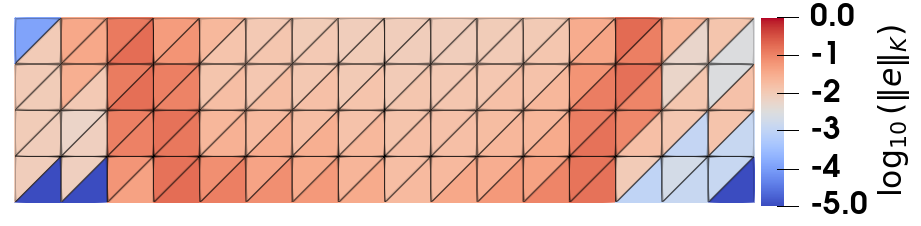} &\includegraphics[width=0.49\textwidth]{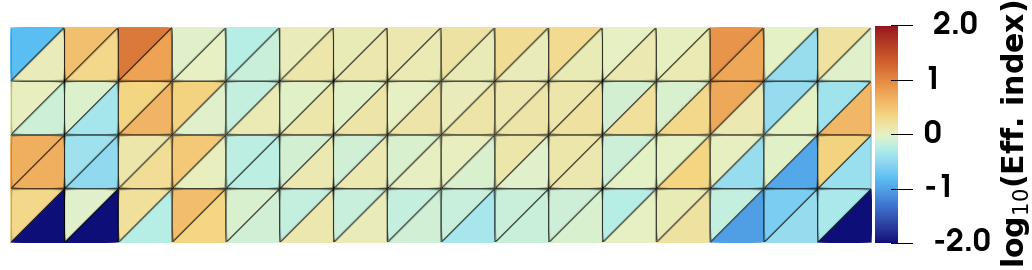} \\
{\small $\log_{10}(\|e\|_K)$ --- level 0} 
&{\small $\textit{Eff. Index} :={\|e\|_K}/\eta_K $,\; level 0} %
\\[6pt]
\includegraphics[width=0.49\textwidth]{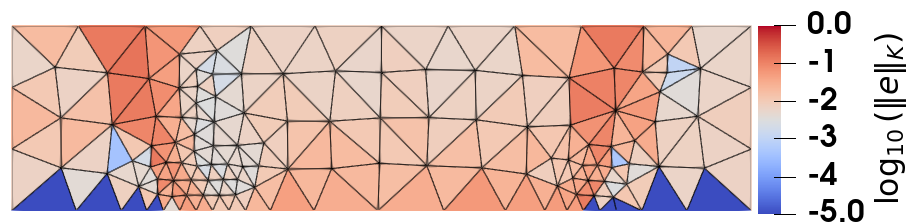}
& \includegraphics[width=0.49\textwidth]{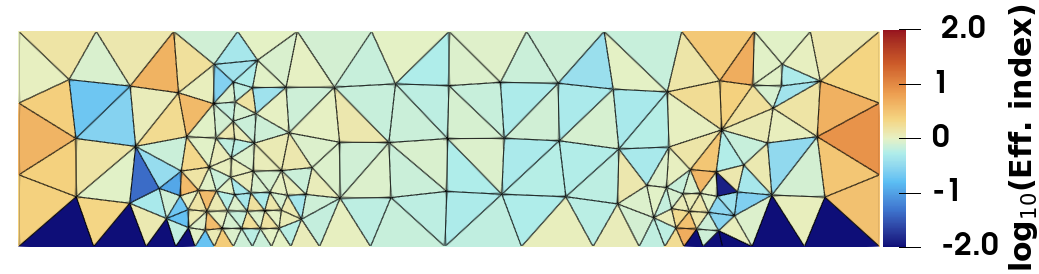}\\
{\small $\log_{10}(\|e\|_K)$ --- level 1}
&{\small $\textit{Eff. Index} :={\|e\|_K}/\eta_K $,\; level 1}\\[6pt]

\includegraphics[width=0.49\textwidth]{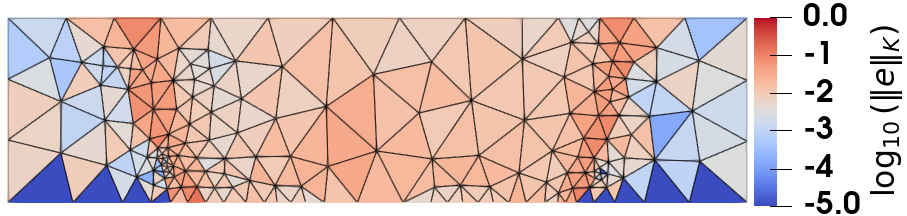} &\includegraphics[width=0.49\textwidth]{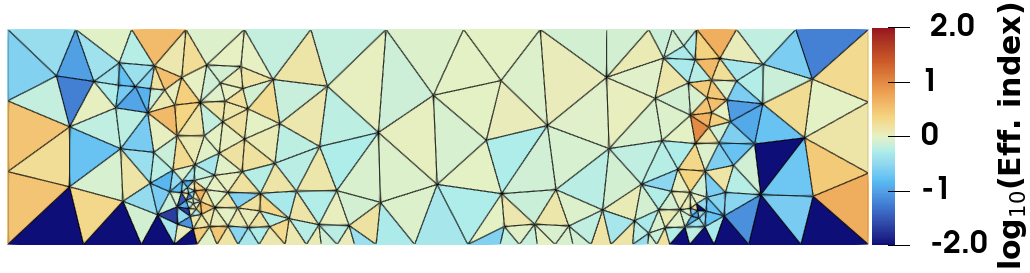} \\
{\small $\log_{10}(\|e\|_K)$ --- level 2} &
{\small $\textit{Eff. Index} :={\|e\|_K}/\eta_K $,\; level 2}\\[6pt]


\includegraphics[width=0.49\textwidth]{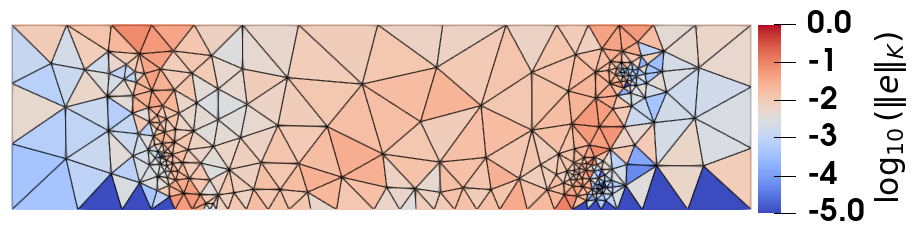} 
& \includegraphics[width=0.49\textwidth]{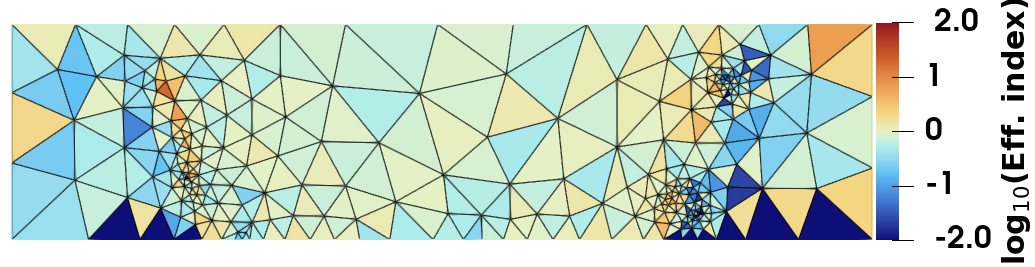}\\
{\small  $\log_{10}(\|e\|_K)$ --- level 3} 
&{\small $\textit{Eff. Index} :={\|e\|_K}/\eta_K $,\; level 3}
\\[6pt]
\includegraphics[width=0.49\textwidth]{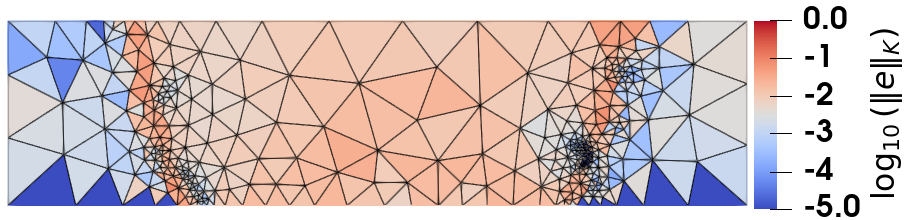} &
\includegraphics[width=0.49\textwidth]{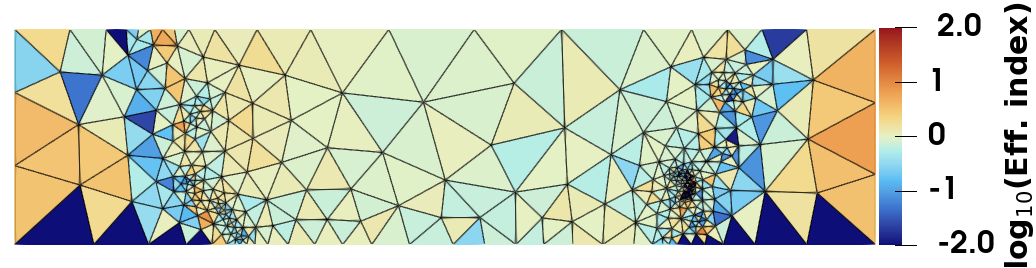} \\
{\small $\log_{10}(\|e\|_K)$ --- level 4} &
{\small $\textit{Eff. Index} :={\|e\|_K}/\eta_K $,\; level 4}
\\[6pt]


\includegraphics[width=0.49\textwidth]{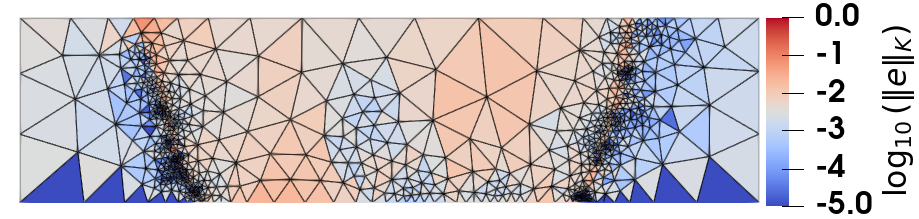} &
\includegraphics[width=0.49\textwidth]{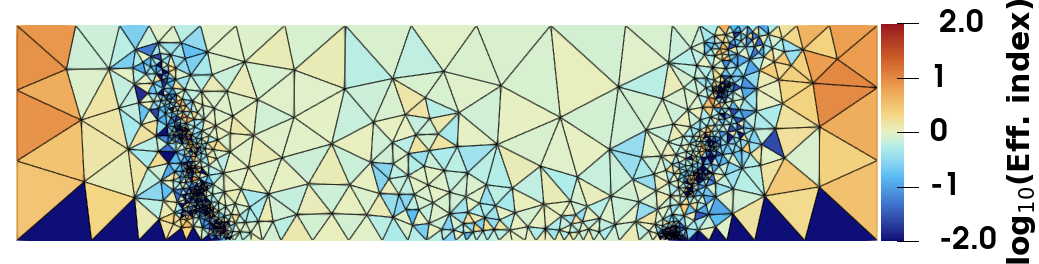} 
 \\
{\small $\log_{10}(\|e\|_K)$ --- level 6} 
&{\small $\textit{Eff. Index} :={\|e\|_K}/\eta_K $,\; level 6}\\[6pt]



\includegraphics[width=0.49\textwidth]{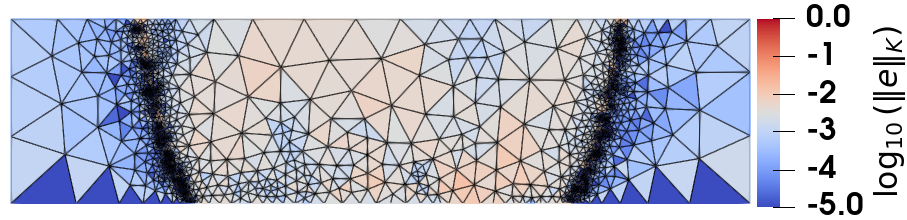} &
\includegraphics[width=0.49\textwidth]{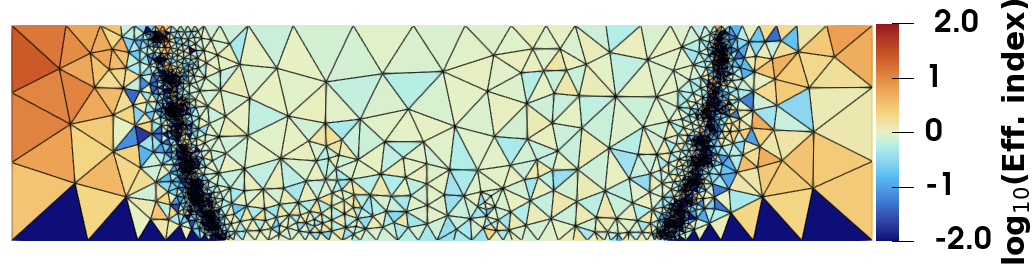} \\
{\small $\log_{10}(\|e\|_K)$ --- level 9} &
{\small $\textit{Eff. Index} :={\|e\|_K}/\eta_K $,\; level 9}
\\[6pt]
\end{longtable}
\captionof{figure}{[Section~\ref{sec:1DPME_1DST}, $m=6$, $\gamma=1.0$, $T=4$] \textbf{Left column:} 
    Element-wise error $\|e\|_K$ defined in \eqref{eq:effectivity} in $\log_{10}$-scale over multiple levels of refinement. \textbf{Right column:} Element-wise effectivity indices in $\log_{10}$-scale. }

\label{fig:true_error_vs_estimator}
\end{center}

\Cref{Figure:1DPME_m=2_dof} shows the error decays with the cost $\mathrm{Cost}(i,\ell)$ defined in \eqref{eq:cost} comparing the adaptive, uniform, and fixed mesh strategies for the $L$-scheme with $m=2$. Both adaptive and uniform strategies reach lower error levels, and the adaptive strategy performs and scales more favorably, albeit slightly for $m=2$ (better performance is observed for higher $m$ values depicted later). \Cref{Figure:1DPME_m=2_iterationno} shows the error decay in different metrics with respect to total iteration number from \eqref{eq:cost}. The adaptive strategy requires more iterations due to the projection process after each refinement step which introduces additional errors. We will see similar behaviour later on. \Cref{Figure:1DPME_Newton_m=2_dof} shows the error decay for the Adapt-$N_{\mathrm{reg}}$, Uni-$N_{\mathrm{reg}}$, and $N_{\mathrm{reg}}$ (fixed mesh) schemes. Similar to \Cref{fig:iter-combined}, these schemes exhibit significantly more oscillations compared to the \(L\)-schemes, even in the $L^2$-norm. Even for uniform refinements, the oscillations persist, as opposed to the \(L\)-schemes, see \Cref{fig:compasion-Uni-Adapt} (Left).  Thus, as concluded also in \cite{mitra2023guaranteed}, the faster convergence of the linearization scheme does not necessarily translate to faster overall convergence in terms of total error. Henceforth, we will focus only on the metric $\|b(s_{h\tau}^{(i)}) - u_{\text{exact}}\|_{L^2(Q)}$ since \Cref{theo:converegence} predicts this metric to vanish as the iterations progress, and also since \Cref{fig:iter-combined,Figure:1DPME_m=2_dof,Figure:1DPME_m=2_iterationno,Figure:1DPME_Newton_m=2_dof} show the errors to oscillate less in this measure.  

\begin{figure}[h]
    \centering
    \begin{subfigure}[b]{0.49\textwidth}
        \includegraphics[width=\textwidth]
{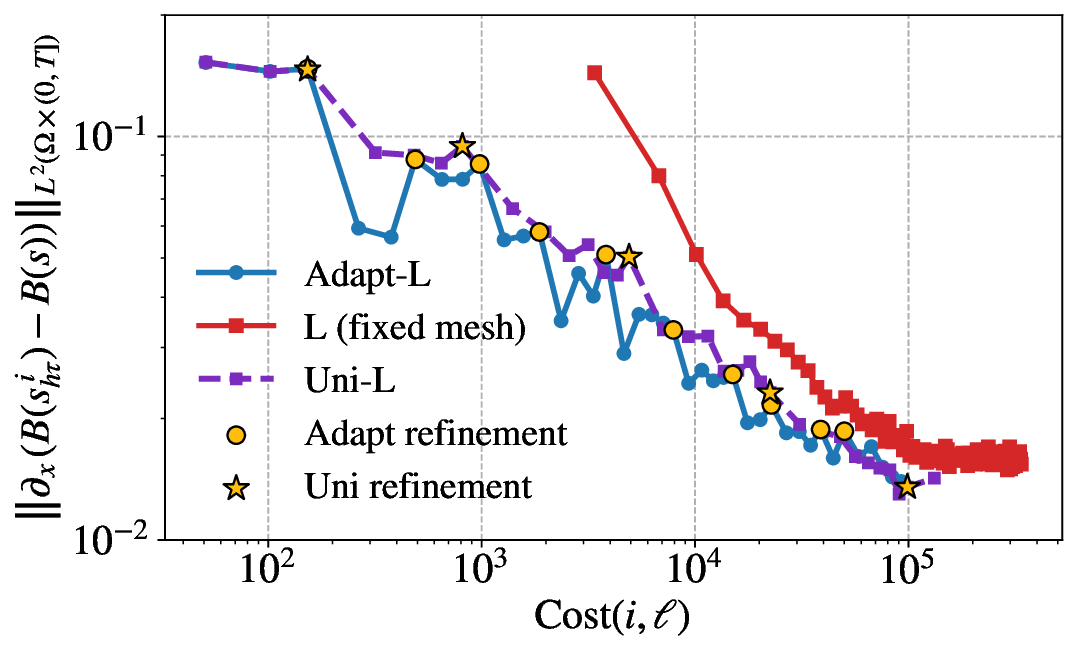}       
    \end{subfigure}
    \hfill
    \begin{subfigure}[b]{0.49\textwidth}
        \includegraphics[width=\textwidth]
{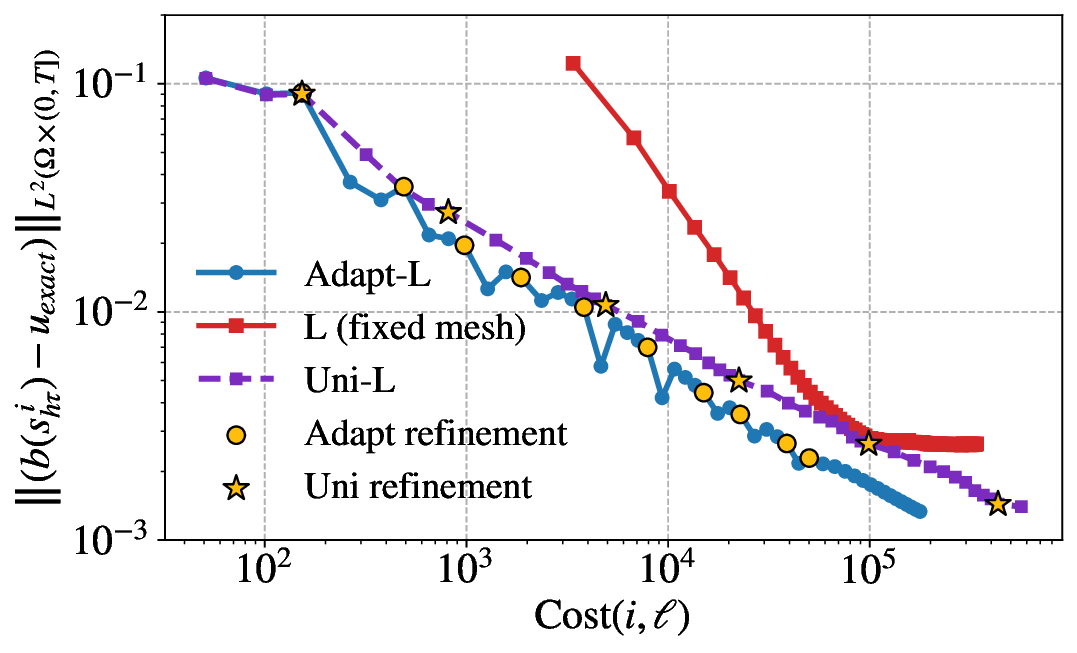}
    \end{subfigure}
    \centering
    \caption{[\Cref{tab:mesh_params_all}, Section \ref{sec:1DPME_1DST}, $m=2$] Comparison of errors versus $\mathrm{Cost}(i,\ell)$ \eqref{eq:cost} for fixed, uniform, and adaptive $L$-schemes. 
\textbf{Left}: Error in terms of $\|\nabla\!\left(B(s_{h\tau}^{(i)}) - B(s)\right)\|_{L^2(Q)}$. 
\textbf{Right}: Error in terms of $\|b(s_{h\tau}^{(i)}) - u_{\text{exact}}\|_{L^2(Q)}$. 
}
\label{Figure:1DPME_m=2_dof}
\end{figure}

\begin{figure}[h]
    \centering
    \begin{subfigure}[b]{0.49\textwidth}
        \includegraphics[width=\textwidth]       {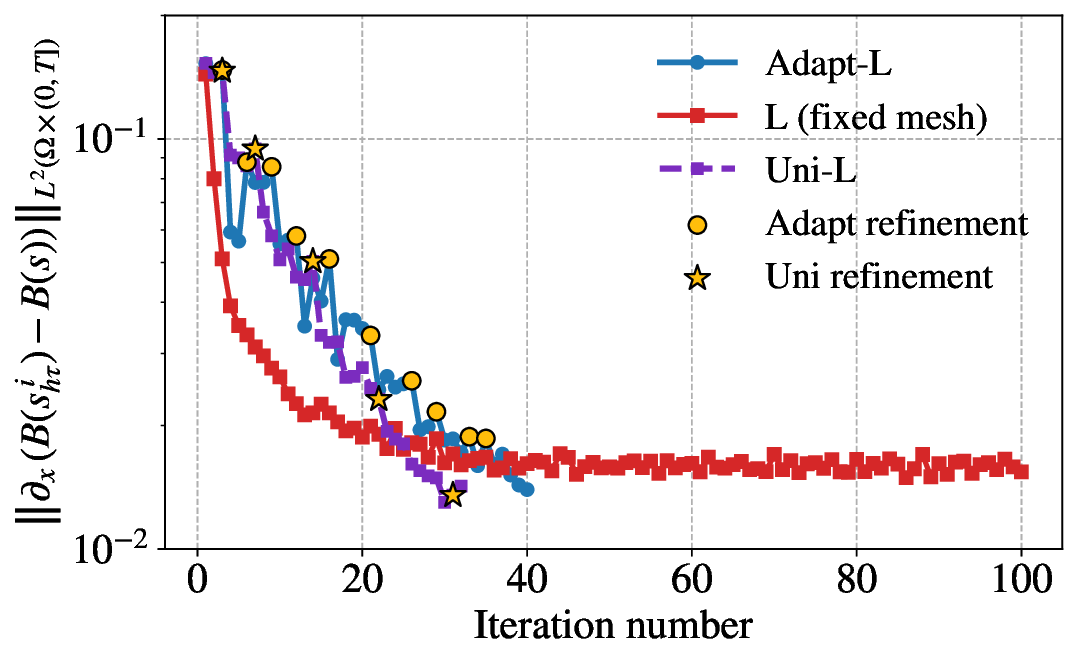}       
    \end{subfigure}
    \hfill
    \begin{subfigure}[b]{0.49\textwidth}
        \includegraphics[width=\textwidth]
{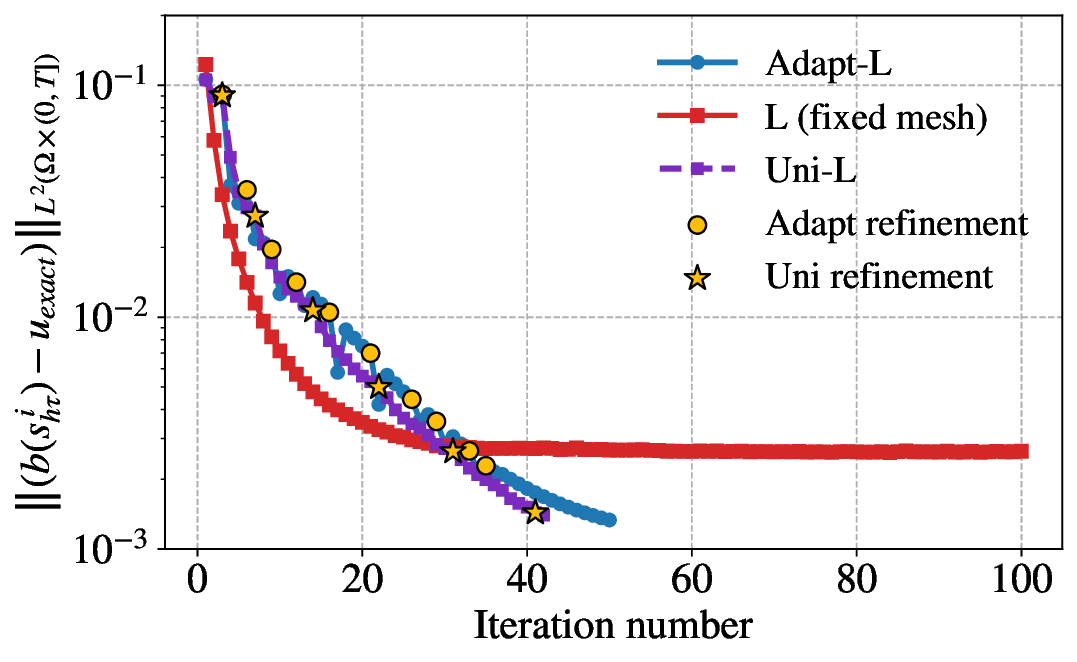}

    \end{subfigure}
    \centering
     \caption{[\Cref{tab:mesh_params_all}, Section \ref{sec:1DPME_1DST}, $m=2$] Comparison of errors versus $\mathrm{IterNo}(i,\ell)$ \eqref{eq:cost} for fixed, uniform and adaptive $L$-schemes.
\textbf{Left}: Error in terms of $\|\nabla\!\left(B(s_{h\tau}^{(i)}) - B(s)\right)\|_{L^2(Q)}$. 
\textbf{Right}: Error in terms of $\|b(s_{h\tau}^{(i)}) - u_{\text{exact}}\|_{L^2(Q)}$. 
}
\label{Figure:1DPME_m=2_iterationno}
\end{figure}

\begin{figure}[h]
    \centering
    \begin{subfigure}[b]{0.49\textwidth}
        \includegraphics[width=\textwidth]{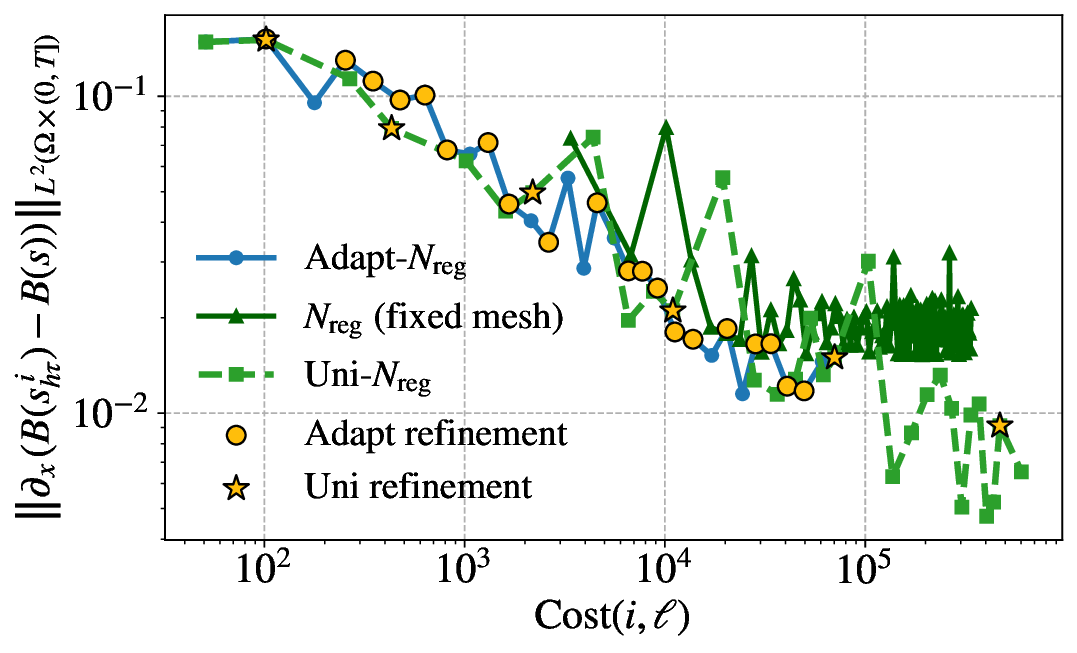}
    \end{subfigure}
    \hfill
    \begin{subfigure}[b]{0.49\textwidth}
        \includegraphics[width=\textwidth]
{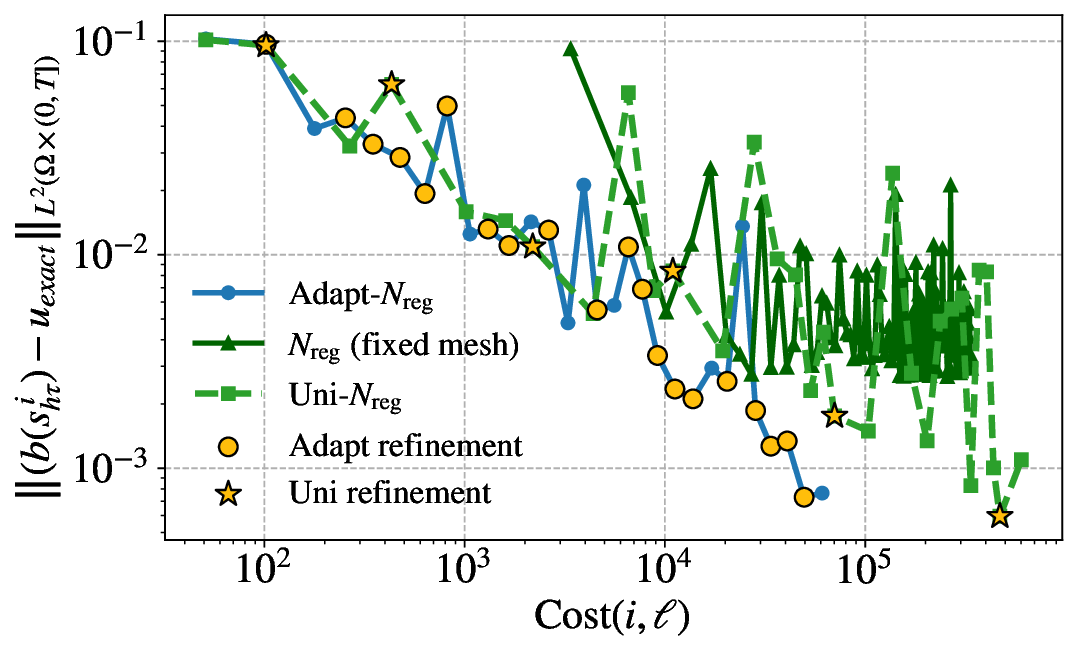}
    \end{subfigure}
    \centering
    \caption{[\Cref{tab:mesh_params_all}, Section \ref{sec:1DPME_1DST}, $m=2$, $\delta_{\mathrm{lin}}=1.0$] Comparison of errors versus $\mathrm{Cost}(i,\ell)$ for fixed, uniform and adaptive $N_{\mathrm{reg}}$-schemes.
\textbf{Left}: Error in terms of  $\|\nabla_x\!\left(B(s_{h\tau}^{(i)}) - B(s)\right)\|_{L^2(Q)}$. 
\textbf{Right}: Error in terms of  $\|b(s_{h\tau}^{(i)}) - u_{\text{exact}}\|_{L^2(Q)}$. 
}
\label{Figure:1DPME_Newton_m=2_dof}
\end{figure}

\begin{figure}[h]
    \centering
   \includegraphics[width=0.47\linewidth]{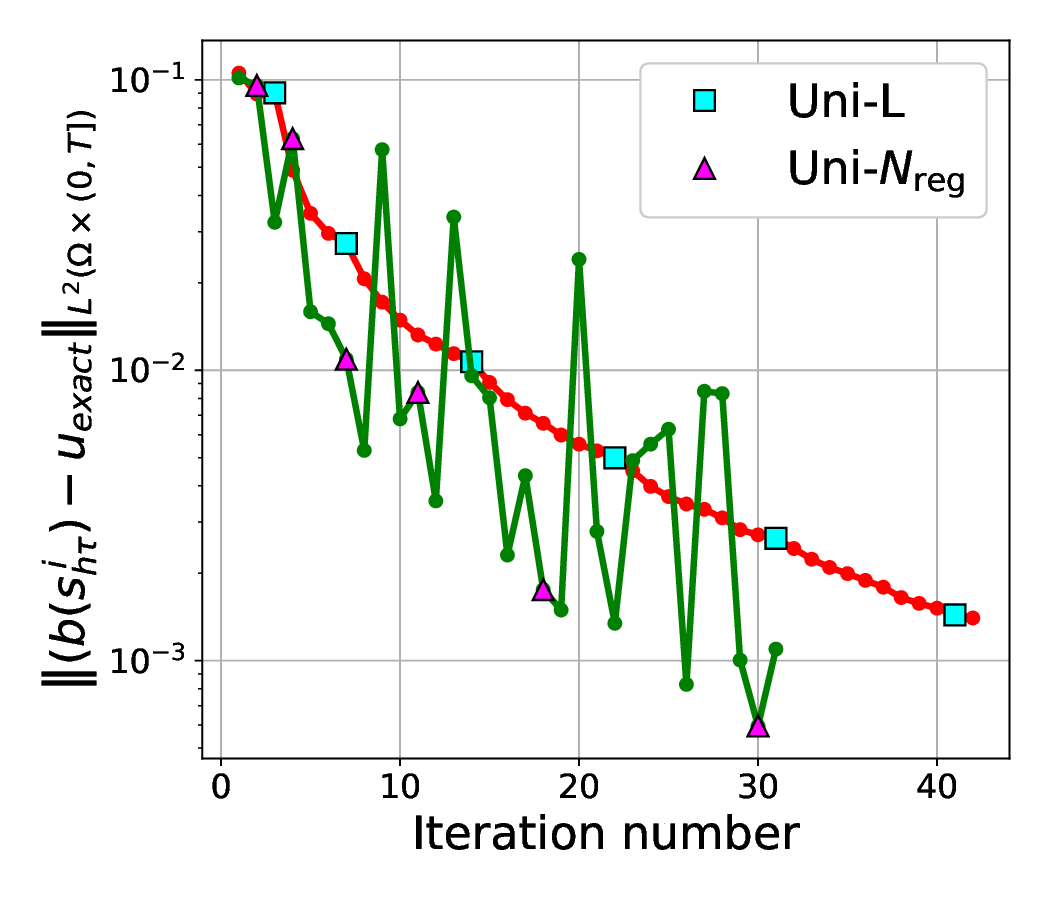}   
    \includegraphics[width=0.49\linewidth]{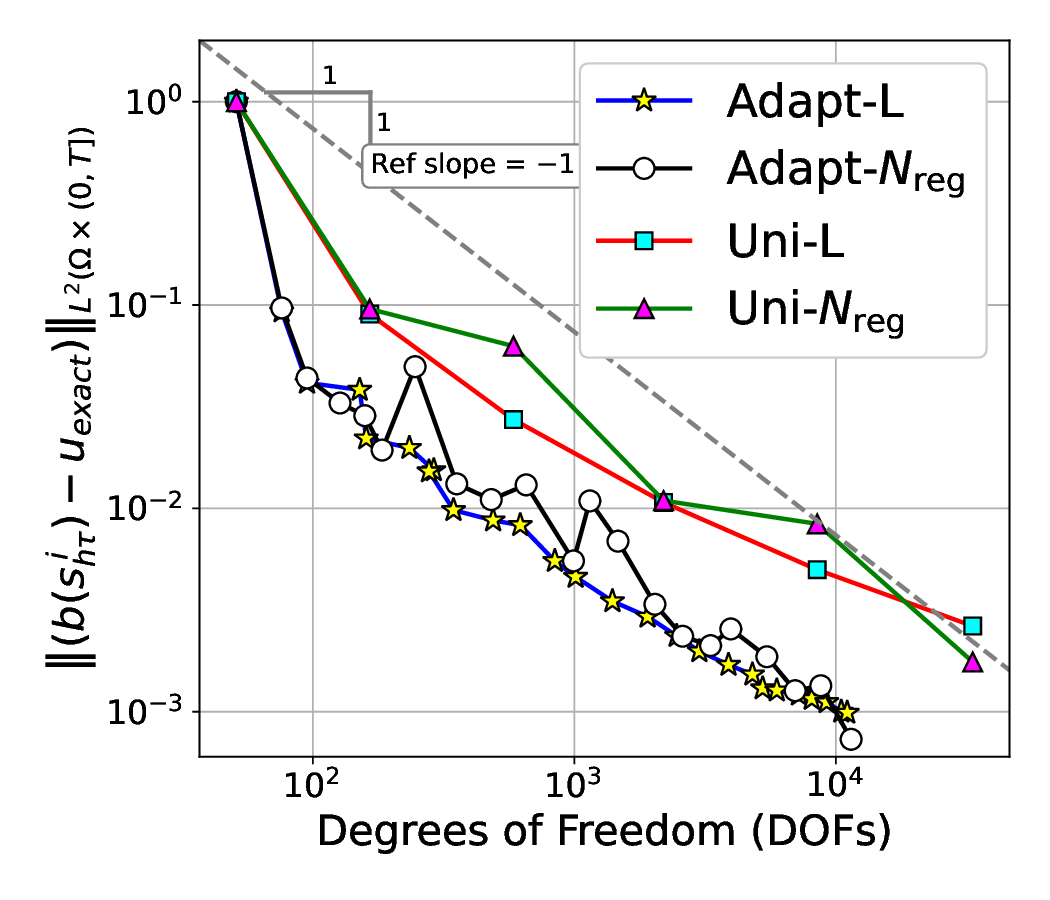}
    \caption{[\Cref{tab:mesh_params_all}, Section \ref{sec:1DPME_1DST}, $m=2$] \textbf{Left}: Comparison of errors versus the number of iterations under successive uniform mesh refinement. Each marker corresponds to a refinement step where the mesh size is halved.
\textbf{Right}: Comparison of errors versus the degrees of freedom ($|\mathrm{DOF}|_\ell$) for uniform and adaptive space-time meshes $\calT^\ell_\htt$ as the iterations progress. 
}
    \label{fig:compasion-Uni-Adapt}
\end{figure}

\Cref{fig:compasion-Uni-Adapt} (Right) compares the error decay with respect to the degrees of freedom associated with each mesh level $\ell$ visited, for both the uniform and adaptive strategies. Again, the adaptive strategy shows better performance, demonstrating almost first-order convergence. \Cref{Figure:1DPME_m=4_dof,Figure:1DPME_m=6_dof}
show the decay of errors with respect to costs (Left) and iteration number (Right) for porous medium exponents $m=4$ and $6$ respectively. The gap between adaptive and uniform strategies is even more pronounced in these cases, with the adaptive scheme incurring costs one order of magnitude less. This is due to the sharp interfaces being more localized for higher $m$ values, resulting in the adaptive scheme being more effective.


\begin{figure}[h]
    \centering
    \begin{subfigure}[b]{0.49\textwidth}
        \includegraphics[width=\textwidth]
{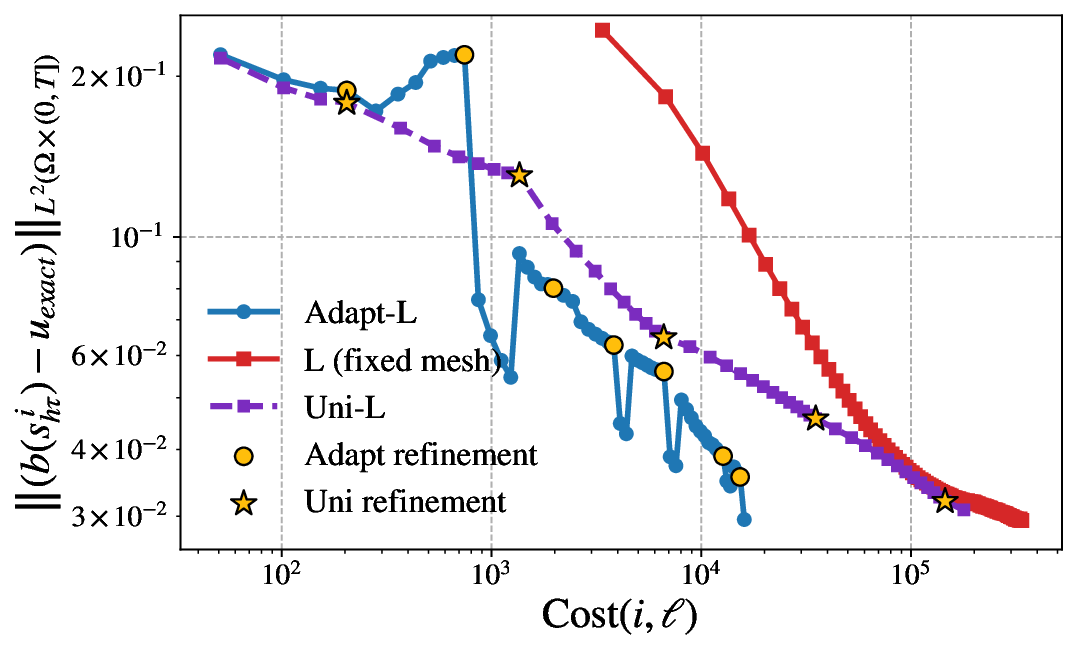}
    \end{subfigure}
   \hfill
   \begin{subfigure}[b]{0.49\textwidth}
\includegraphics[width=\textwidth]
{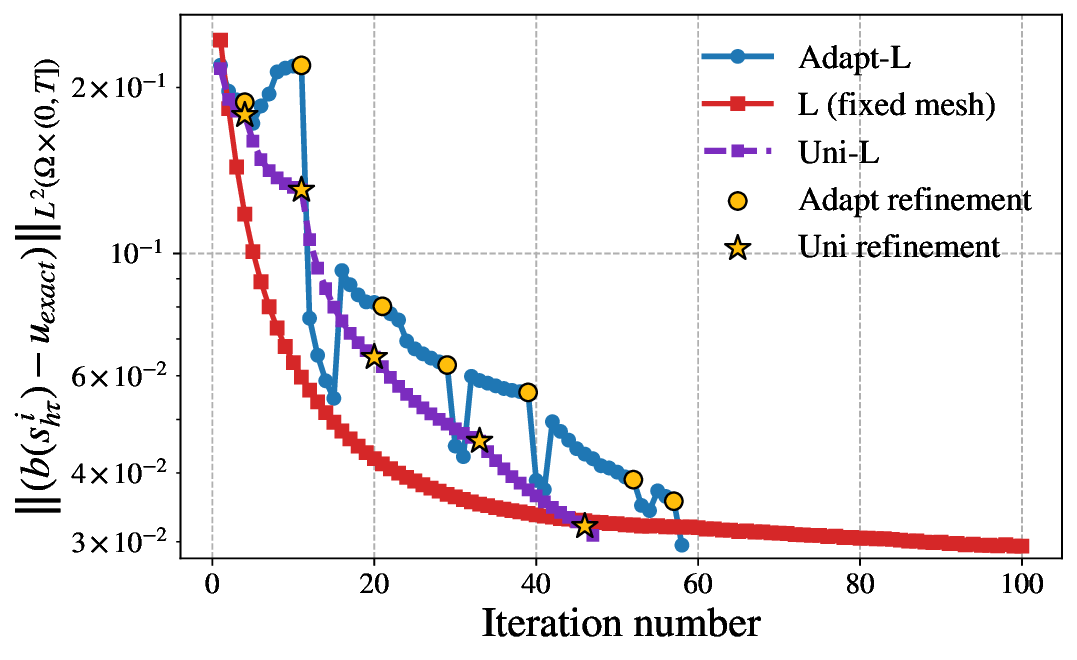}
   \end{subfigure}
    \centering
    \caption{[\Cref{tab:mesh_params_all}, Section \ref{sec:1DPME_1DST}, $m=4$]  Comparison of  error $\|b(s_{h\tau}^{(i)}) - u_{\mathrm{exact}}\|_{L^2(Q)}$ 
for fixed, uniform, and adaptive $L$-schemes. 
\textbf{Left:} error versus $\mathrm{Cost}(i,\ell)$. 
\textbf{Right:} error versus $\mathrm{IterNo}(i,\ell)$. 
}
\label{Figure:1DPME_m=4_dof}

\end{figure}

\begin{figure}[h]
    \centering
    \begin{subfigure}[b]{0.49\textwidth}
        \includegraphics[width=\textwidth]{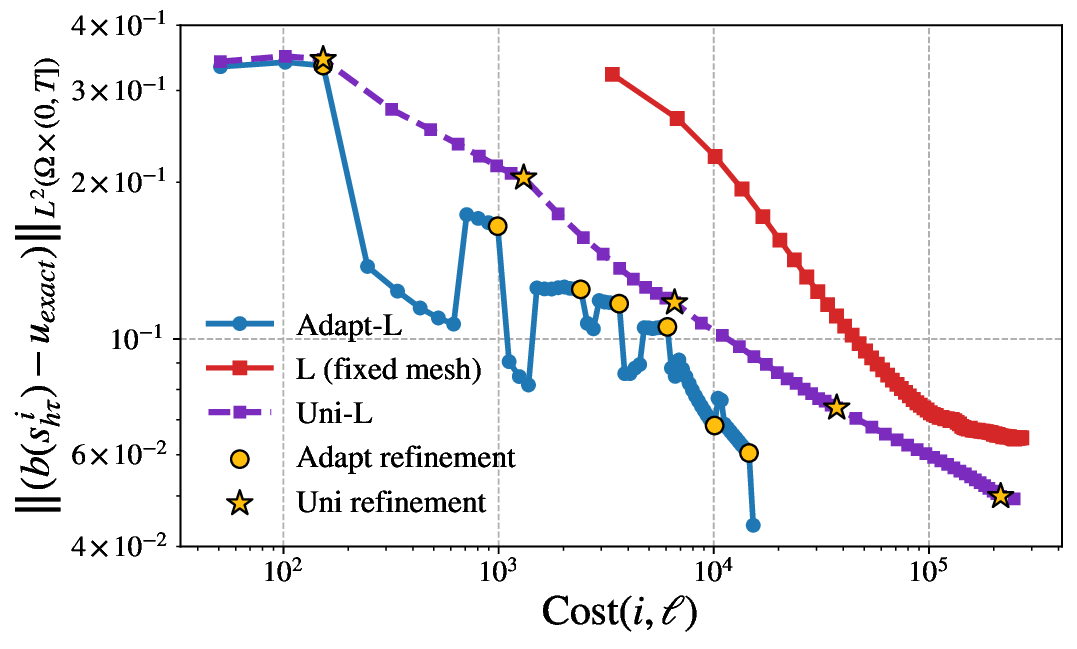}
    \end{subfigure}
    \hfill
    \begin{subfigure}[b]{0.49\textwidth}
        \includegraphics[width=\textwidth]
{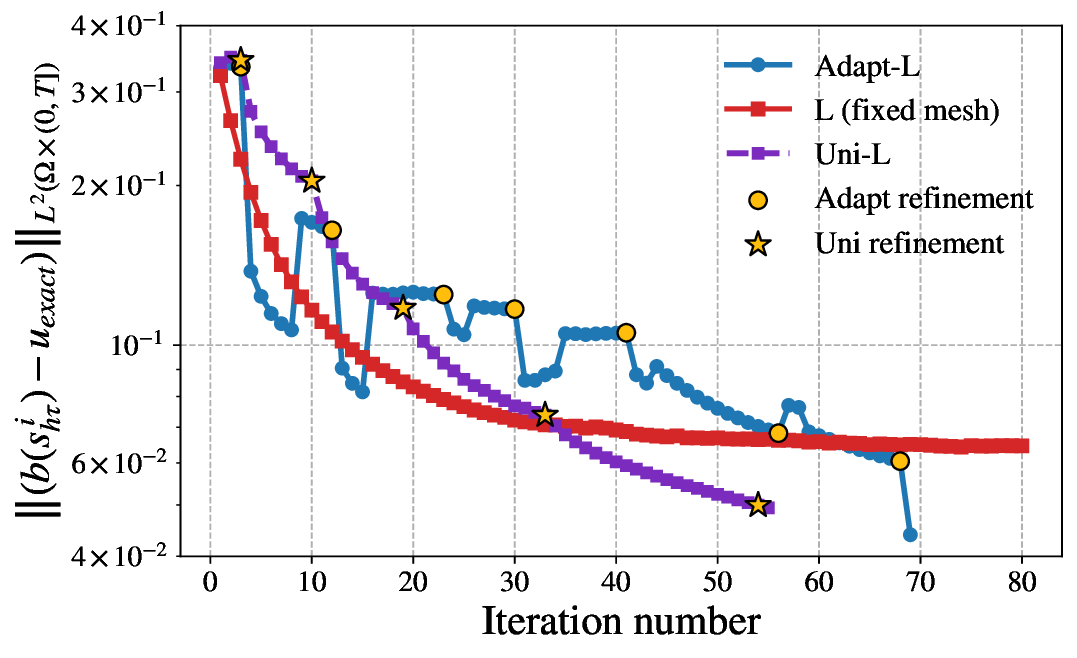}
    \end{subfigure}
 \caption{[\Cref{tab:mesh_params_all}, Section \ref{sec:1DPME_1DST}, $m=6$] Comparison of  error $\|b(s_{h\tau}^{(i)}) - u_{\mathrm{exact}}\|_{L^2(Q)}$ 
for fixed, uniform, and adaptive $L$-schemes. 
\textbf{Left:} error versus $\mathrm{Cost}(i,\ell)$. 
\textbf{Right:} error versus $\mathrm{IterNo}(i,\ell)$. 
}
\label{Figure:1DPME_m=6_dof}
\end{figure}
\subsubsection{2D PME}\label{2DPME_2DST}

Next, we look at numerical results for the $L$-scheme in two space dimensions, i.e., $d=2$. Starting with parameters from \Cref{tab:mesh_params_all} and the Barenblatt initial condition \eqref{eq:Barenblatt},
\Cref{fig:cube_mesh_refinement_3d} (a) shows the solution $u=b(s)$ plotted over the final adapted mesh. 
The time-stacked slices in \Cref{fig:cube_mesh_refinement_3d} (b) show that refinement remains focused around the free boundary at all time levels. The spatial slices along the $x$-axis show that the mesh is fine only in a thin region where the solution has steep changes. The efficiency of the adaptive approach is further confirmed in \Cref{fig:2DPME_DOF_INTERATIONNO_M6}, where we compare the error against the cost and total number of iterations for fixed and adaptive strategies.  To achieve a prescribed accuracy, the adaptive method requires significantly fewer DoFs than both the fixed-mesh and the uniform refinement strategies. It also reaches much lower error levels. Note that uniform refinement becomes prohibitively more expensive for higher dimensional computations. Overall, we find the performance of the adaptive space-time scheme to be excellent, particularly for higher $m$ values and in higher space dimensions.

\begin{figure}[h]
\centering
\begin{subfigure}[t]{0.34\textwidth}
    \centering
    \includegraphics[width=\linewidth]{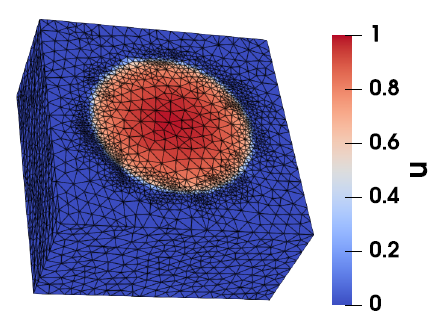}
    \caption{}
\end{subfigure}
\hfill
\begin{subfigure}[t]{0.32\textwidth}
    \centering
    \includegraphics[width=\linewidth]{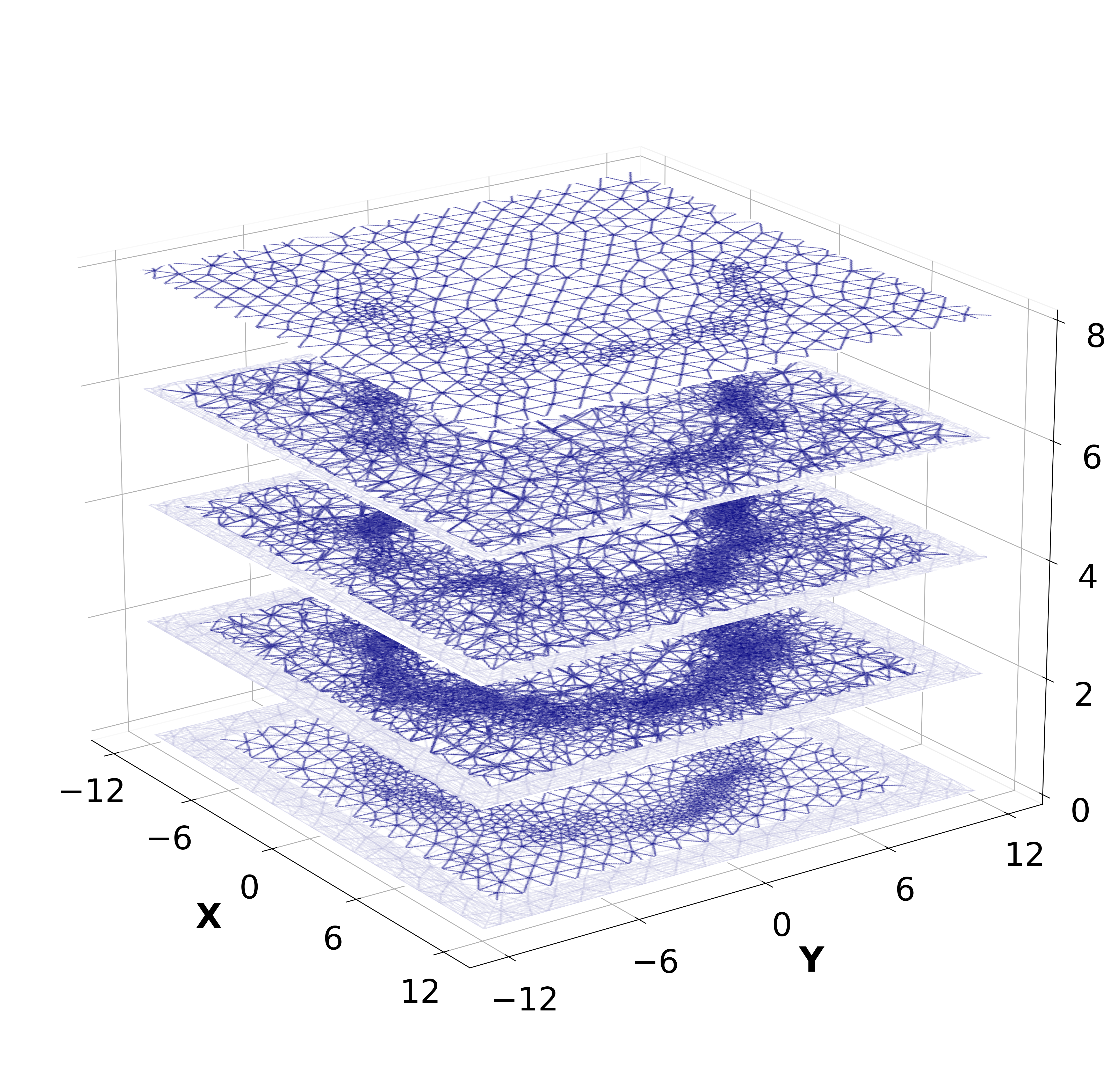}
    \caption{}
\end{subfigure}
\hfill
\begin{subfigure}[t]{0.32\textwidth}
    \centering
    \includegraphics[width=\linewidth]{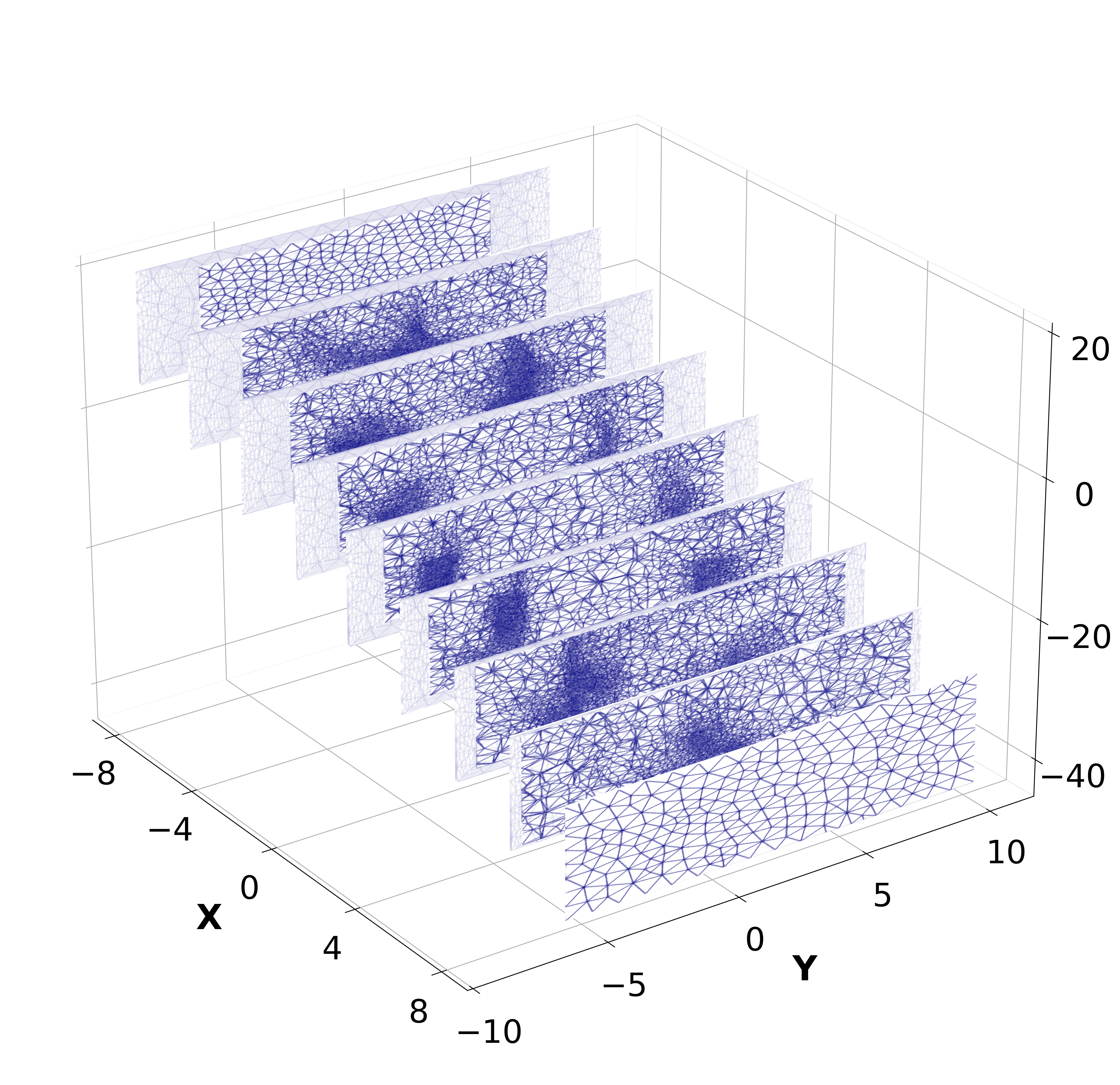}
    \caption{}
\end{subfigure}
\caption{[\Cref{tab:mesh_params_all,2DPME_2DST}, $m=6$, $T=8$] Numerical results for 2D PME. \textbf{(a)} Numerical solution $u=b(s)$ on the final adapted mesh in the $(x,y,t)$-domain. \textbf{(b)} Mesh slices stacked along the $t$-axis, illustrating the adaptive 
refinement across time levels. \textbf{(c)} Spatial mesh slices stacked along the $x$-axis.}
\label{fig:cube_mesh_refinement_3d}
\end{figure}

\begin{figure}[h]
    \centering
    \begin{subfigure}[b]{0.49\textwidth}
        \includegraphics[width=\textwidth]{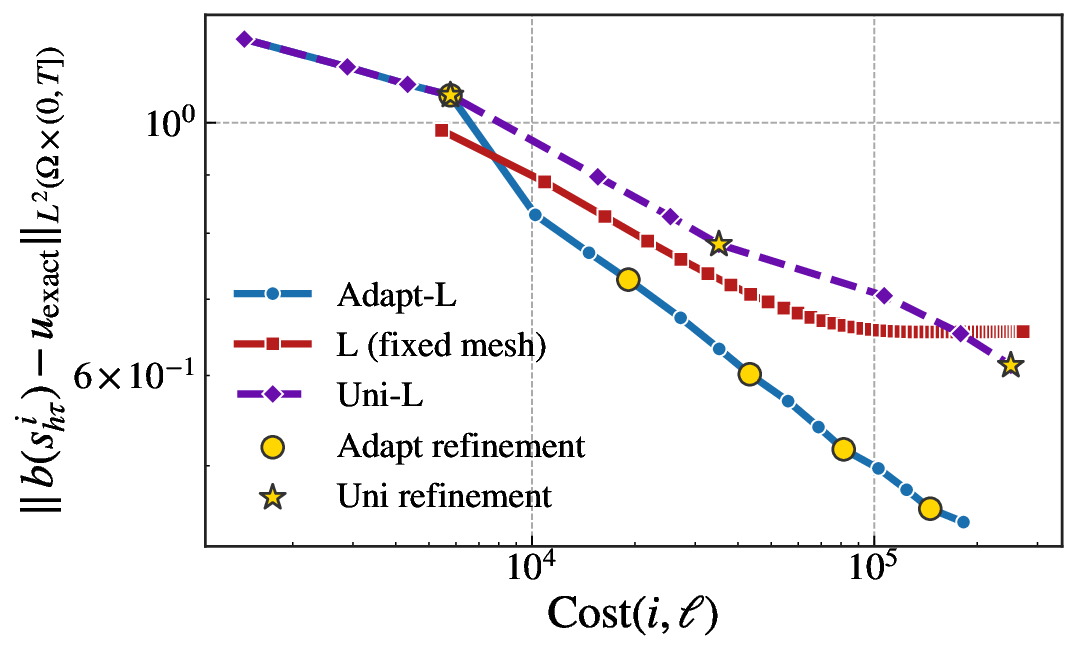}
    \end{subfigure}
     \hfill
    \begin{subfigure}[b]{0.49\textwidth}
        \includegraphics[width=\textwidth]
{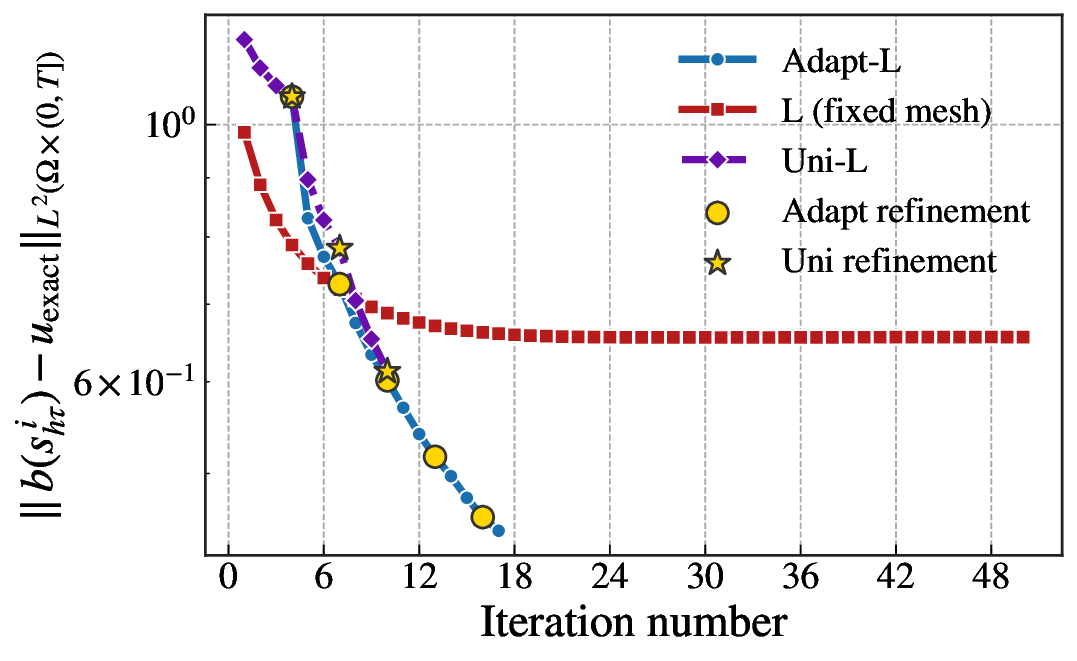}
    \end{subfigure}
    \centering
    \caption{[\Cref{tab:mesh_params_all}, Section~\ref{2DPME_2DST}, $m=6$]
Comparison of error $\|b(s_{h\tau}^{(i)}) - u_{\mathrm{exact}}\|_{L^2(Q)}$ for fixed and adaptive $L$-schemes. \textbf{Left:} error versus $\mathrm{Cost}(i,\ell)$. \textbf{Right:} error versus $\mathrm{IterNo}(i,\ell)$.
}
\label{fig:2DPME_DOF_INTERATIONNO_M6}

\end{figure}
\subsection{Biofilm growth model}\label{sec:Biofilm}
We consider the biofilm equation \cite{eberl2001new}
\begin{equation*}
\frac{\partial u}{\partial t} = \Delta \Phi(u) \qquad \text{ where } \qquad \Phi(u):=\left(\frac{u}{1-u}\right)^m \qquad \text{ with exponent }\; m>1.
\end{equation*}
We choose $m=2$ for this section. 
For the initial condition, $\g=0.8$ is additionally chosen in \eqref{eq:Barenblatt} so that $\max_\Om \Phi'(u_0)=200$. For these parameters, $u^* = 0.229083$ in \Cref{lemma:b_B_construction} is computed, satisfying $\Phi'(u^*)=1$. Then, the $b$, $B$ functions in \Cref{lemma:b_B_construction} become
\[
b(s)=
\begin{cases}
0, & s<0,\\
s, & 0\le s\le u^*,\\
\dfrac{(s-u^*+\Phi(u^*))^{1/m}}{1+(s-u^*+\Phi(u^*))^{1/m}}, & s>u^*,
\end{cases}
\quad
B(s)=
\begin{cases}
s, & s<0,\\
\dfrac{s^m}{(1-s)^m}, & 0\le s\le u^*,\\
s-u^*+\Phi(u^*), & s>u^*.
\end{cases}
\]


\begin{figure}[h]
\centering

\begin{subfigure}{0.42\textwidth}
  \centering
  \includegraphics[width=1.05\linewidth]{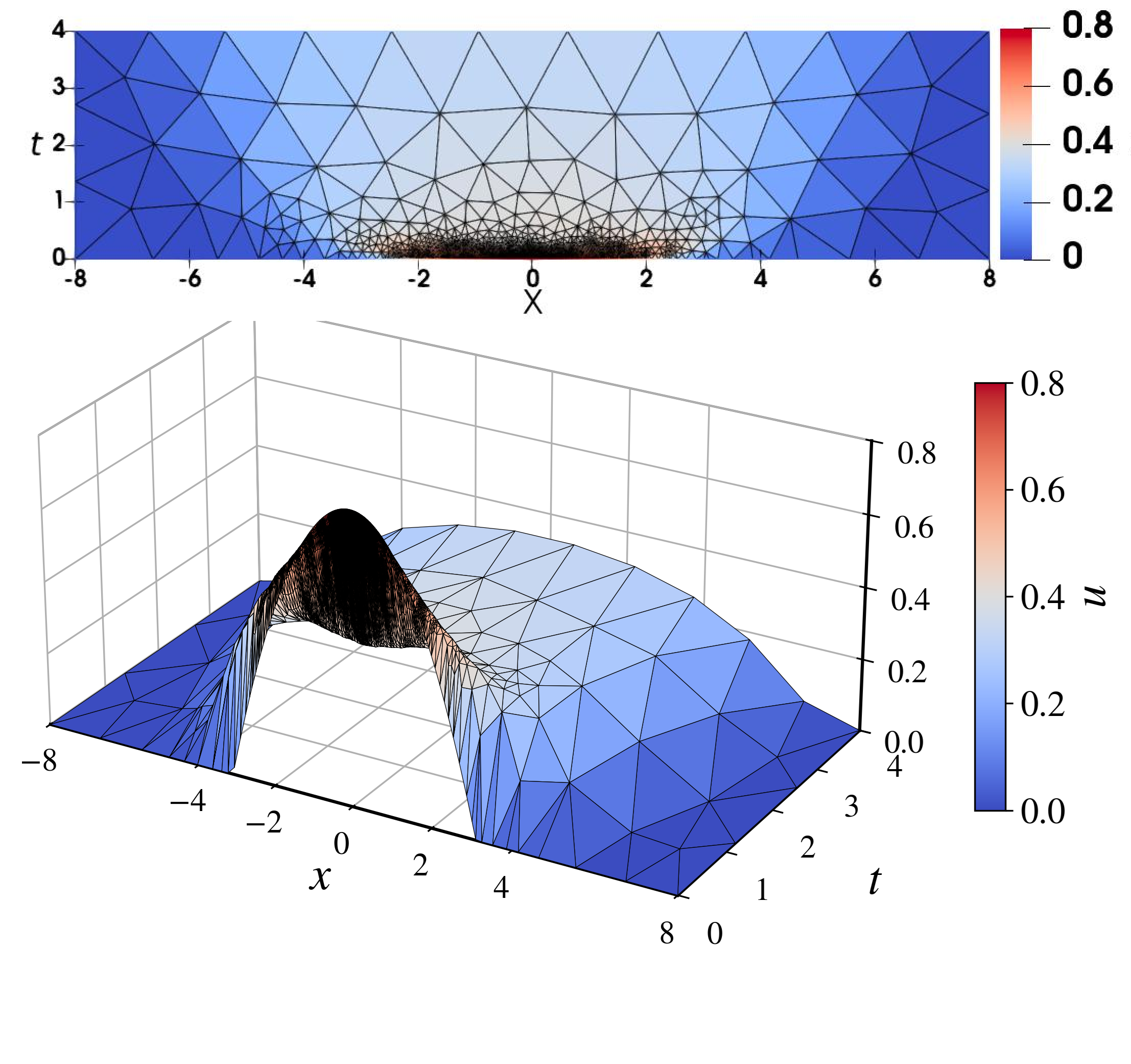}
\end{subfigure}
\hfill
\begin{subfigure}{0.49\textwidth}
\centering
\includegraphics[width=1.05\linewidth]{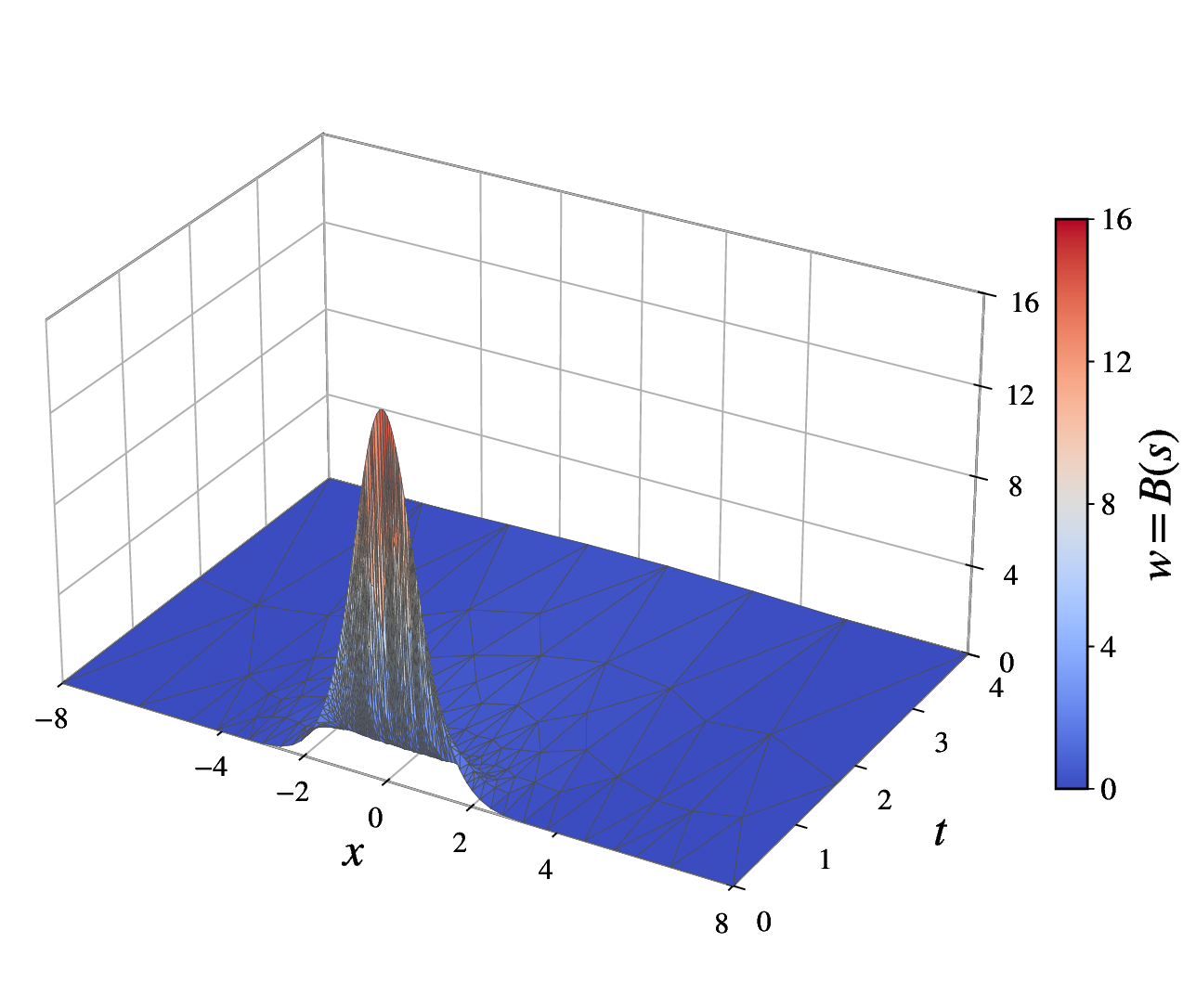}
\end{subfigure}
\caption{[\Cref{sec:Biofilm}, $d=1$, $m=2$, $\gamma=0.8$, $T=4$] 
Three-dimensional space-time profiles of the numerical solution of the 1D biofilm model on the final adapted space-time mesh. \textbf{Left:} $u = b(s)$, \textbf{Right:} $w = B(s)$.}
\label{fig:biofilm_1d_3D}
\end{figure}

\Cref{fig:biofilm_1d_3D} shows the numerical solution $(u,w)$ of the biofilm problem computed over the final adapted mesh. It is seen that most of the refinement happens close to $t=0$ due to steep $w$ gradients. The results in Figure~\ref{fig:adaptive_vs_uniform} illustrate the performance of the three strategies: fixed mesh (L), uniform refinement (Uni-L), and adaptive refinement (Adapt-L), in terms of the total error estimated by $\eta^i$ (since the exact solution is not known).
The fixed-mesh strategy shows a steady, smooth decrease in the total error throughout the computation. 
However, once the linearization converges, the total error reaches a plateau determined by the discretization error, the best approximation the fixed mesh can deliver. 
For Uni-L, the total error decreases smoothly between refinements. However, at each global refinement step (marked by a star), the solution is interpolated onto the new, finer mesh, temporarily increasing the linearization error. The iterations then converge again on the refined mesh. This produces the visible spikes in the curve at each refinement step. This is substantiated by \Cref{fig:adaptive_vs_uniform_L2norm} which shows error decay when the linearization estimator $\eta_{\mathrm{lin}}^i$ is expressed by $\|s^{i+1}-s^i\|_{L^2(Q)}$ which is less effected by the projection operation. Being devoid of the $\|\nabla (s^{i+1}-s^i)\|_{L^2(Q)}$ and $\|\p_t (s^{i+1}-s^i)\|_{L^2(Q)}$ components causes the total error to decay monotonically in \Cref{fig:adaptive_vs_uniform_L2norm}. For the Adapt-L strategy, a similar behavior is observed. Local refinements are triggered more frequently, as indicated by the circles, see \Cref{fig:adaptive_vs_uniform}. The projection operation causes the spikes at each refinement step. In terms of accuracy for a given computational cost, the Uni-L and Adapt-L strategies are quite similar for this test case.

\begin{figure}[h]
    \centering
    \includegraphics[width=0.48\textwidth]{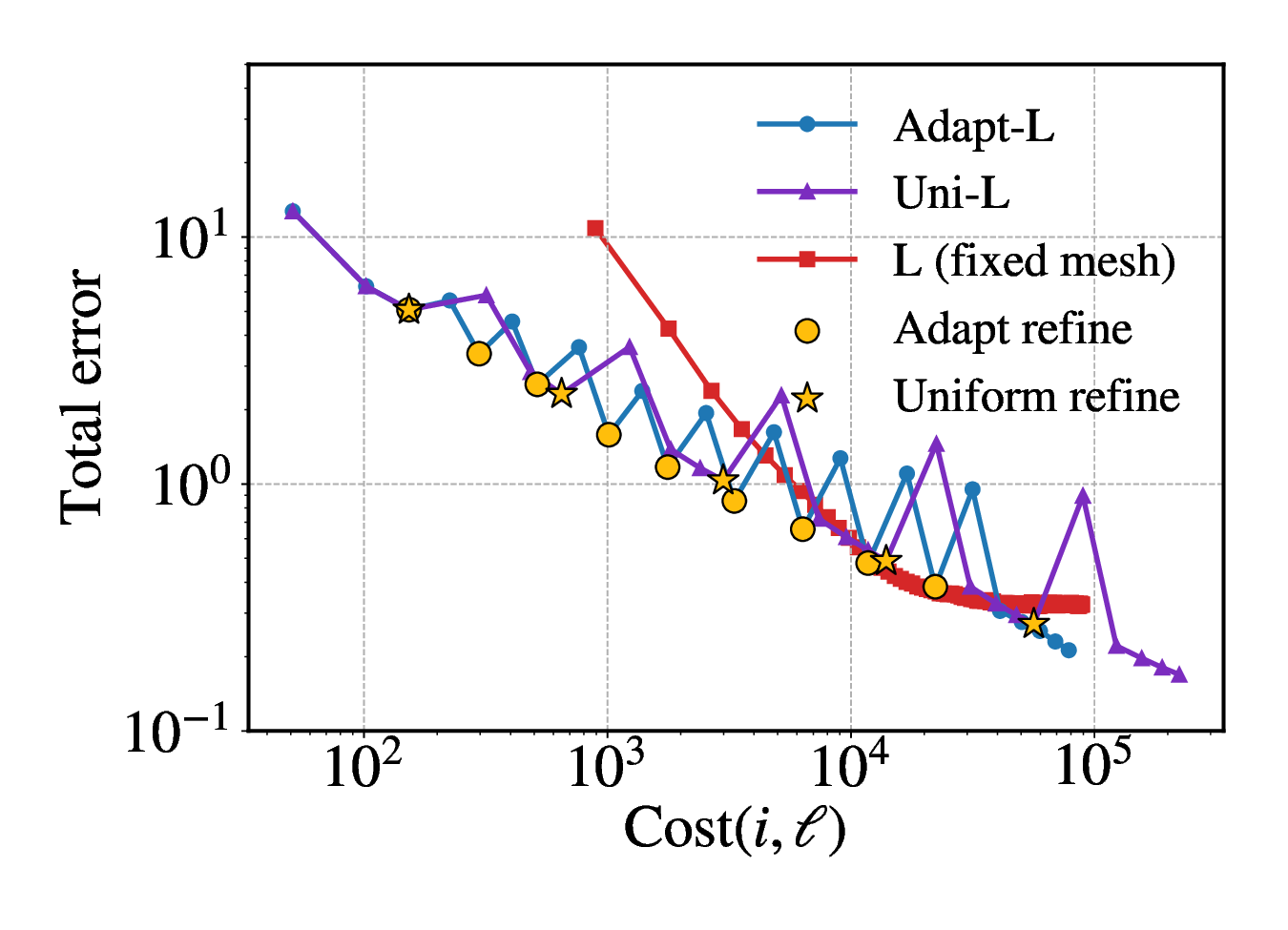}
    \hfill
    \includegraphics[width=0.48\textwidth]{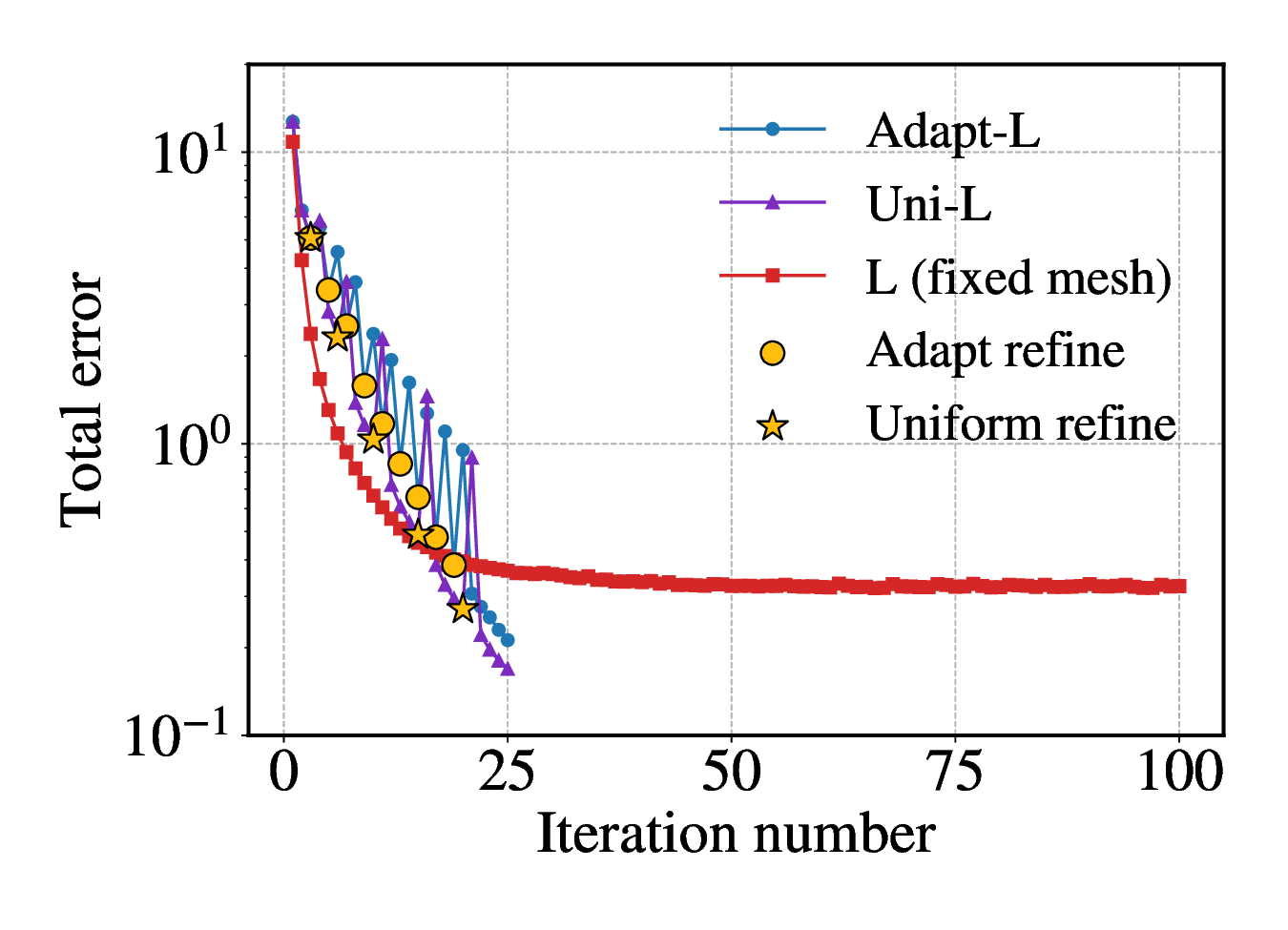}
\caption{[\Cref{tab:mesh_params_all,sec:Biofilm}, $d=1$] 
Comparison of total error estimator $\eta^i = \eta^i_{\mathrm{disc}} + L\,\eta^i_{\mathrm{lin}}$ 
for fixed, uniform, and adaptive $L$-schemes. 
\textbf{Left:} $\eta^i$ versus $\mathrm{Cost}(i,\ell)$. 
\textbf{Right:} $\eta^i$ versus $\mathrm{IterNo}(i,\ell)$. }
\label{fig:adaptive_vs_uniform}
\end{figure}



\Cref{fig:adaptive_vs_uniform_Mschemes} shows the error decay for the $N_{\mathrm{reg}}$-schemes. A similar error level is reached as compared to the $L$-schemes, both for the adaptive and uniform strategies. More oscillations are present, similar to \Cref{Figure:1DPME_Newton_m=2_dof}. The spikes are seen again at each refinement step, which get removed if the linearization estimator does not use $L^2$-norms of derivatives. \Cref{fig:adaptive_Lschemes_Biofilm_2D} shows how L-schemes compare when the same experiment is run in two space dimensions ($d=2$). The Adapt-L scheme significantly outperforms the Uni-L scheme since extremely fine resolution is required to capture the very high error close to $t=0$. In fact, in this case, even the fixed mesh scheme shows better performance.

\begin{figure}[h]
    \centering
    \includegraphics[width=0.48\textwidth]{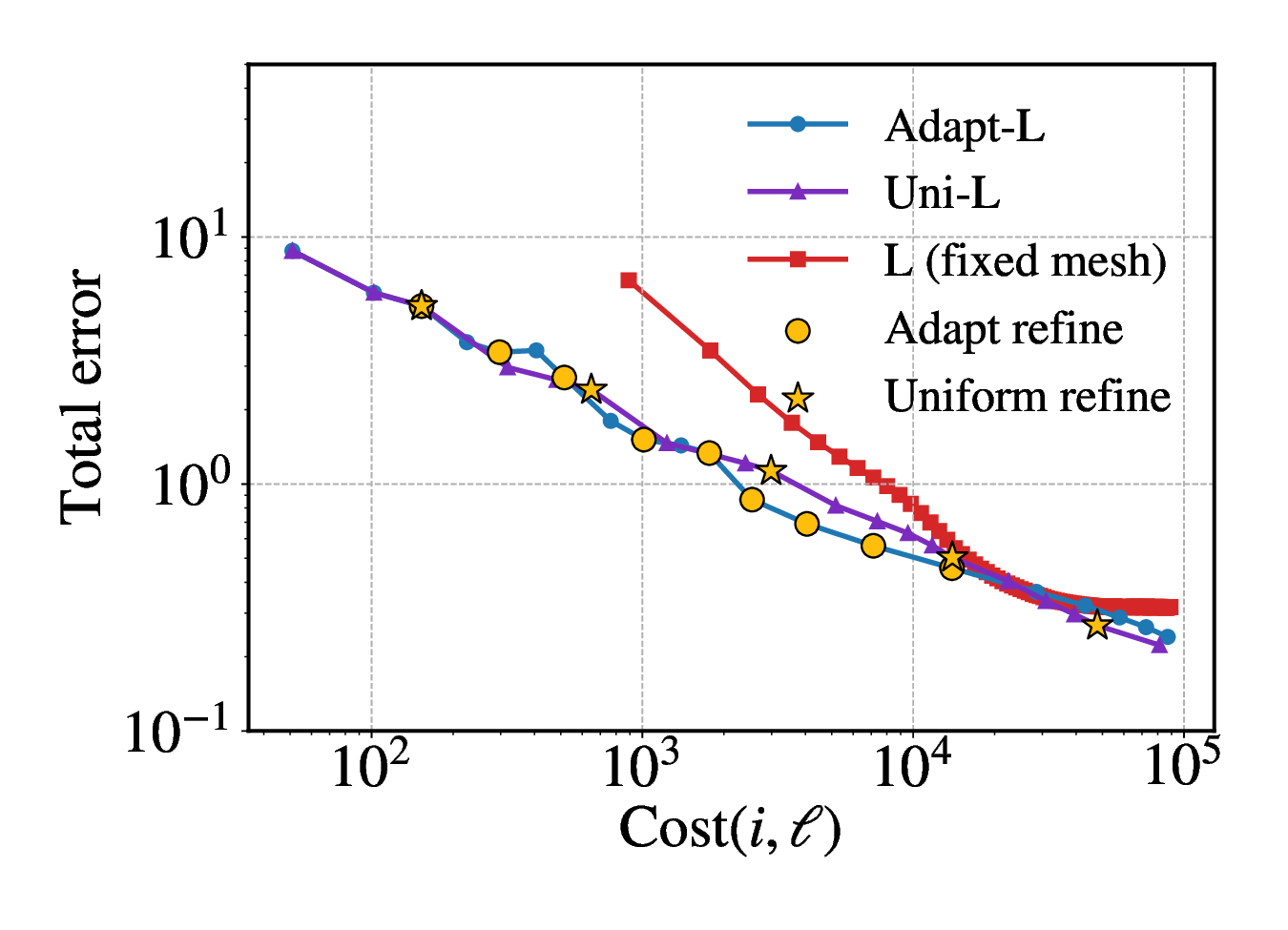}
    \hfill
    \includegraphics[width=0.48\textwidth]{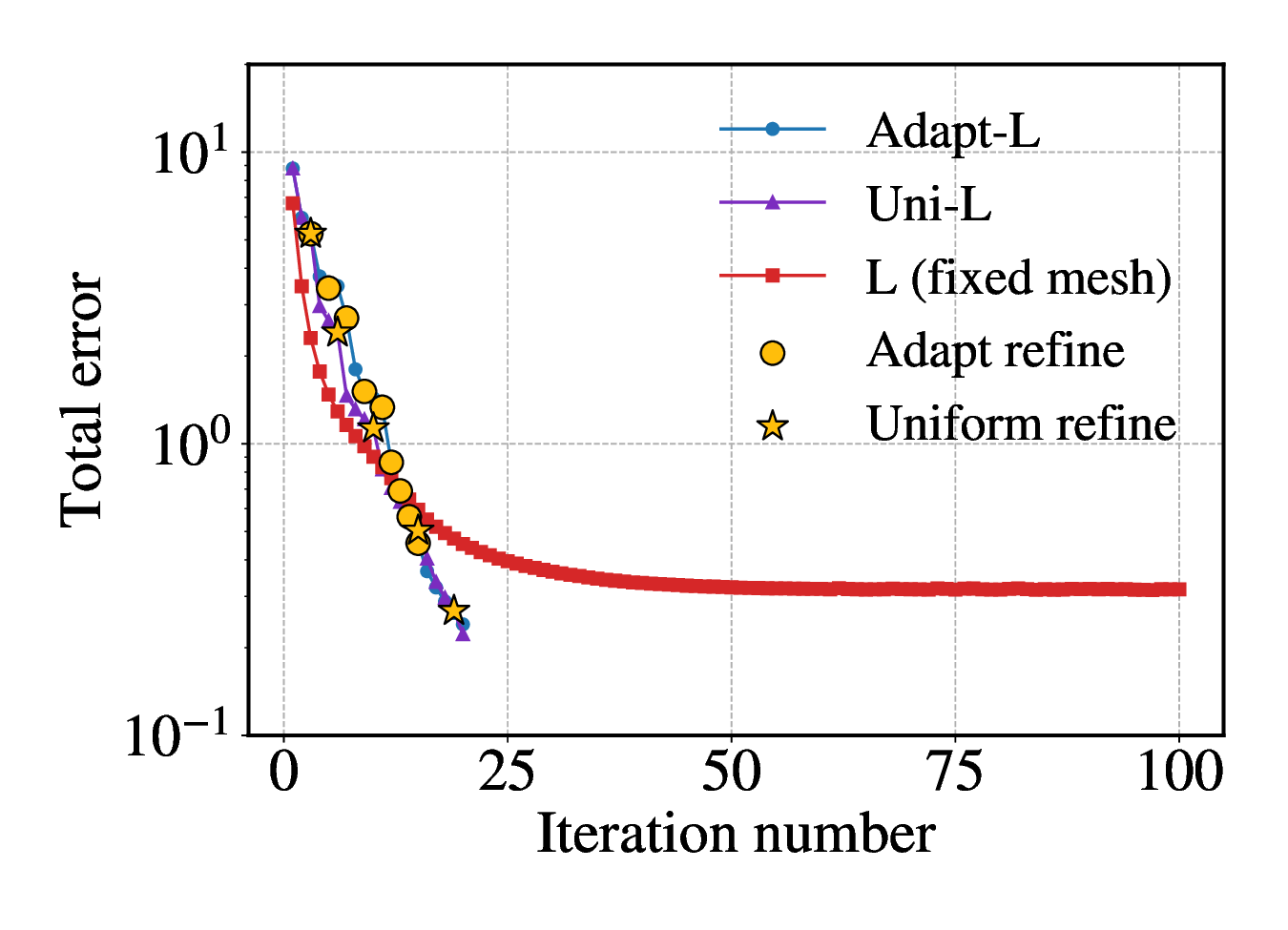}
\caption{[\Cref{tab:mesh_params_all,sec:Biofilm}, $d=1$] 
Comparison of total error estimator $\eta^i = \eta^i_{\mathrm{disc}} + L\,\eta^i_{\mathrm{lin}}$ 
for fixed, uniform, and adaptive $L$-schemes when linearization error is estimated by $\left\|s_h^{(i)} - s_h^{(i-1)}\right\|_{L^2(Q)}$. 
\textbf{Left:} $\eta^i$ versus $\mathrm{Cost}(i,\ell)$. 
\textbf{Right:} $\eta^i$ versus $\mathrm{IterNo}(i,\ell)$.}
\label{fig:adaptive_vs_uniform_L2norm}
\end{figure}

\begin{figure}[h]
\centering
\includegraphics[width=0.49\textwidth]{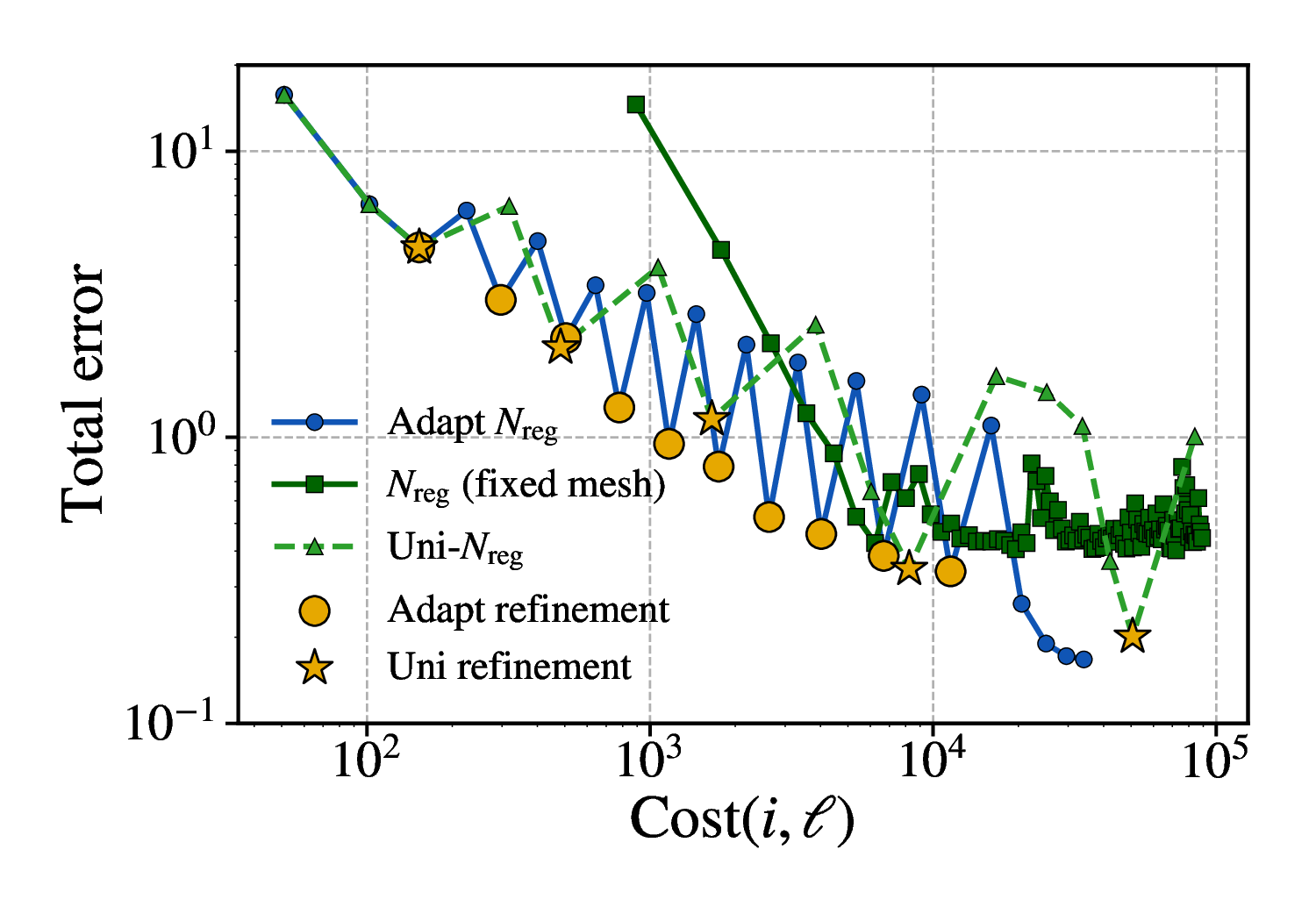}
\hfill
\includegraphics[width=0.49\textwidth]{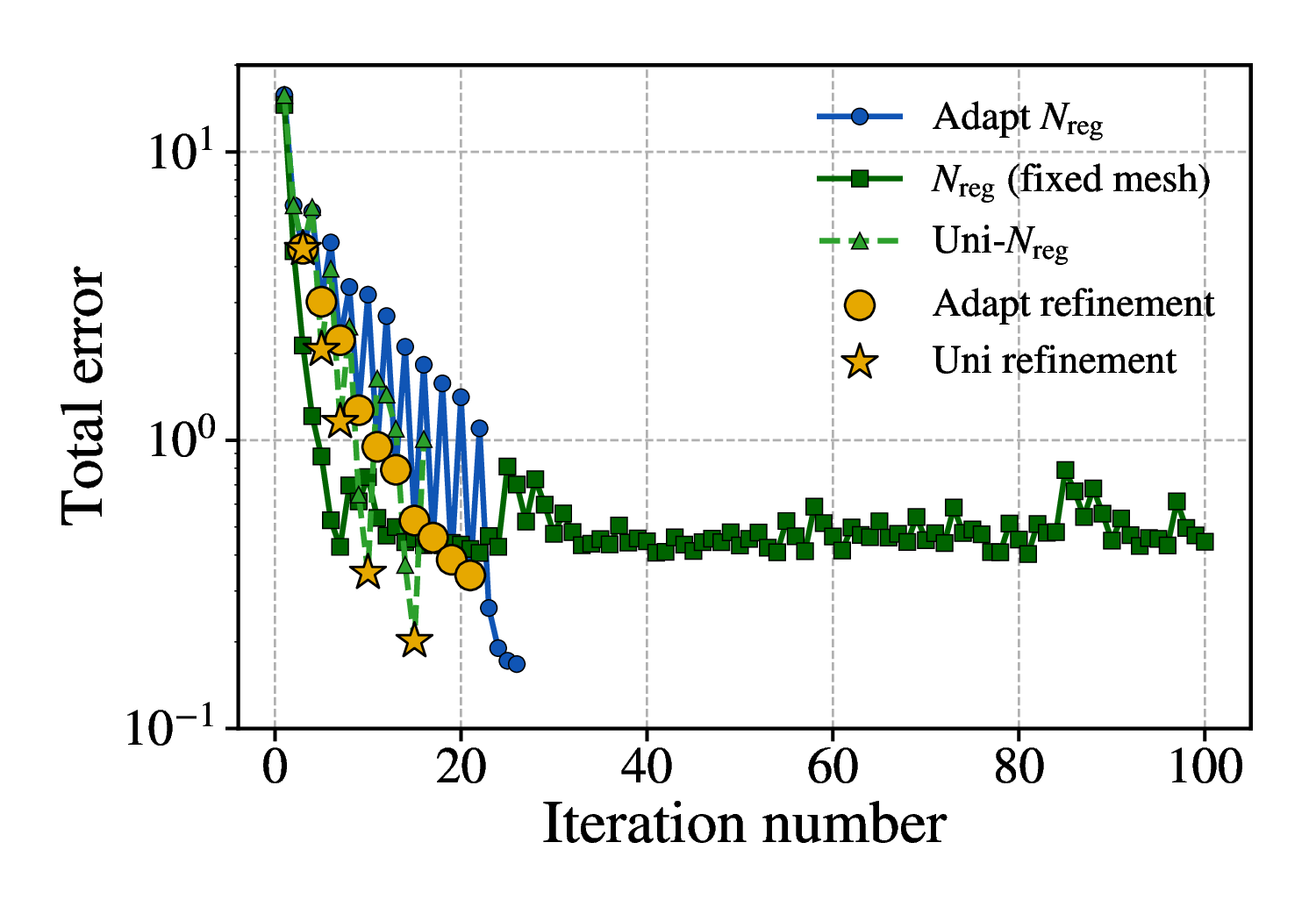}
\caption{
[\Cref{tab:mesh_params_all,sec:Biofilm}, $d=1$] 
Comparison of total error estimator $\eta^i = \eta^i_{\mathrm{disc}} + L\,\eta^i_{\mathrm{lin}}$ 
for fixed, uniform, and adaptive $N_{\mathrm{reg}}$-schemes. 
\textbf{Left:} $\eta^i$ versus $\mathrm{Cost}(i,\ell)$. 
\textbf{Right:} $\eta^i$ versus $\mathrm{IterNo}(i,\ell)$.}
\label{fig:adaptive_vs_uniform_Mschemes}
\end{figure}

\begin{figure}[h]
    \centering
    \begin{subfigure}[b]{0.5\textwidth}
        \includegraphics[width=1.01\textwidth]
{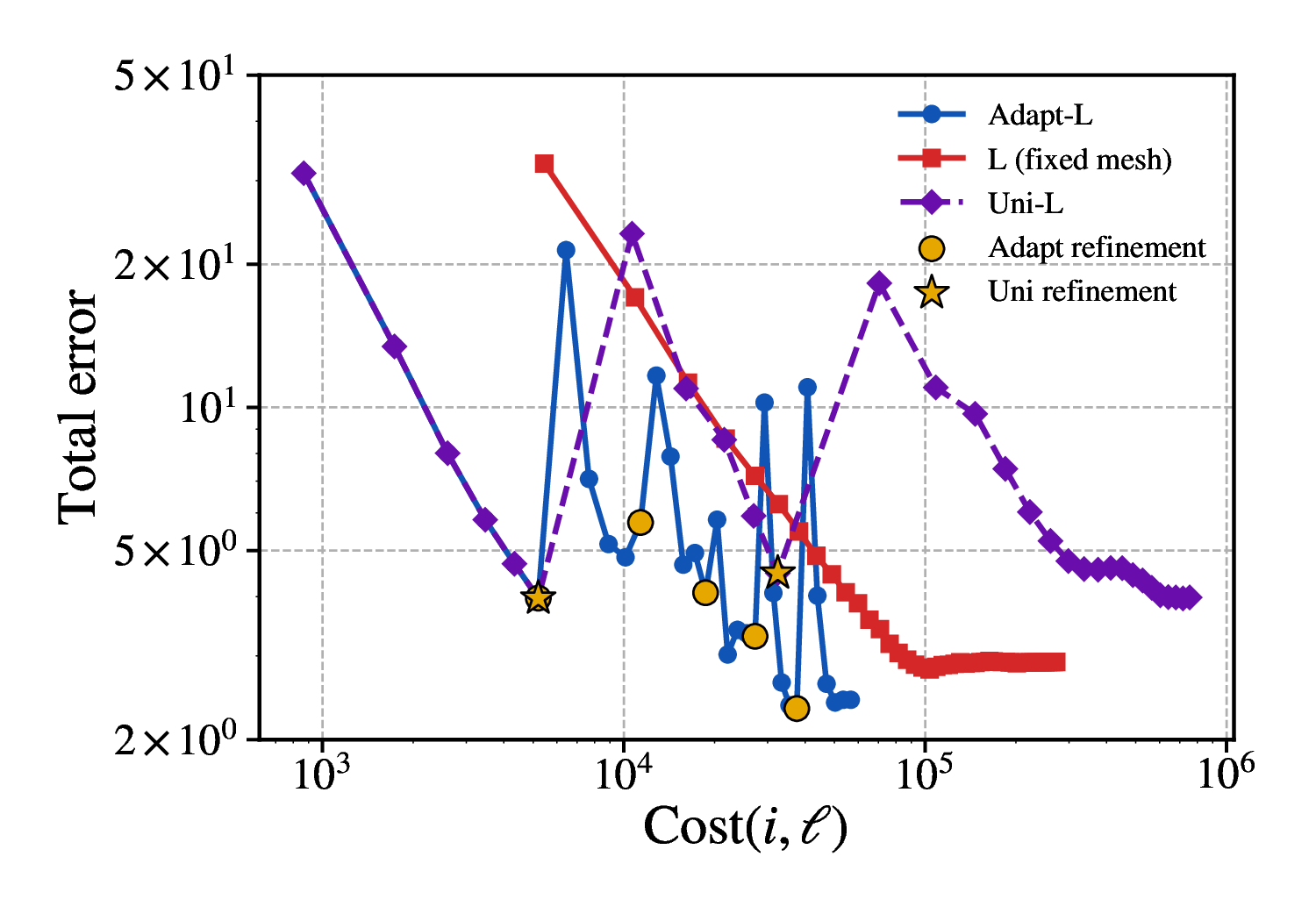}
    \end{subfigure}
    \hfill
    \begin{subfigure}[b]{0.49\textwidth}
        \includegraphics[width=1.01\textwidth]
{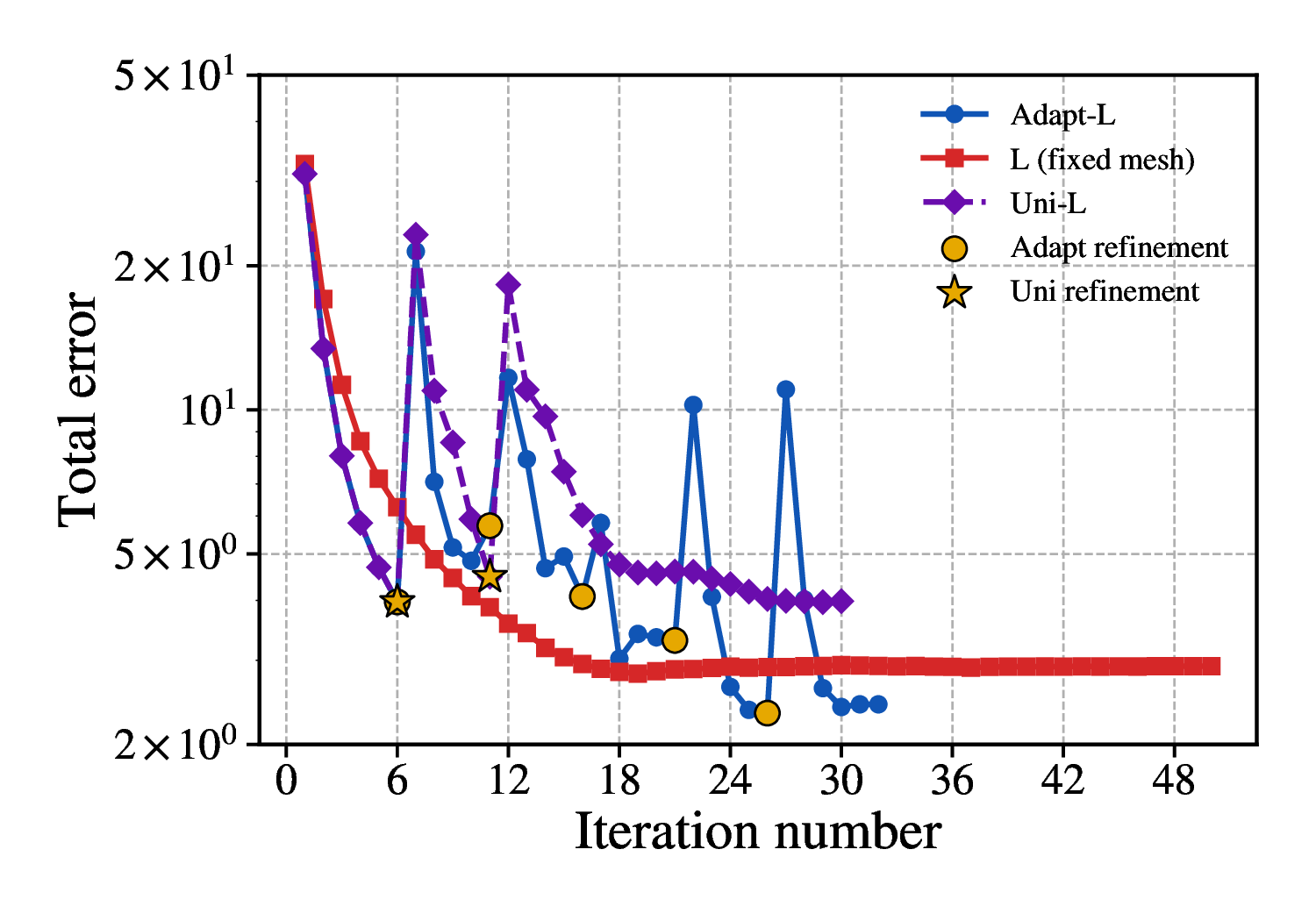}

\end{subfigure}
\centering
\caption{[\Cref{tab:mesh_params_all,sec:Biofilm}, $d=2$] Comparison of total error estimator $\eta^i = \eta^i_{\mathrm{disc}} + L\,\eta^i_{\mathrm{lin}}$ 
for fixed, uniform, and adaptive $L$-schemes. 
\textbf{Left:} $\eta^i$ versus $\mathrm{Cost}(i,\ell)$. 
\textbf{Right:} $\eta^i$ versus $\mathrm{IterNo}(i,\ell)$.}
\label{fig:adaptive_Lschemes_Biofilm_2D}
\end{figure}


\subsection{A double degenerate toy-model}\label{sec:toy}
Finally, we investigate a double degenerate toy-model from \cite{javed2025robust} where $\Phi$ becomes multivalue at $\vs=1$, and there is a degenerate elliptic region (fast-diffusion)  $\{u=\om\}$ of non-zero measure:
\begin{align}\label{eq:toy}
\frac{\partial u}{\partial t} = \Delta w+ \frac{1}{2}\, u , \quad  w \in \Phi( u) \quad \text{ where } \quad \Phi(u) =
\begin{cases}
0 & \text{if } u \leq 0, \\
1-\sqrt{1-u^2} & \text{if } 0 \leq u < 1,\\
 [1, \infty] & \text{if } u = 1.
\end{cases}
\end{align}
The initial condition \eqref{eq:Barenblatt} with $m=2$ and$\g=1.0$ as in the PME case.

For this problem, using \eqref{eq:bBexpression}, the functions  $b,\, B$ can be expressed explicitly. For $u^* = \frac{1}{\sqrt{2}}$,
\[
b(s)=
\begin{cases}
s, & s \le u^*,\\[4pt]
\sqrt{1 - (\sqrt{2} - s)^2}, & u^* < s < \sqrt{2},\\[4pt]
1, &  \text{otherwise},
\end{cases}
\qquad
B(s)=
\begin{cases}
0, & s < 0,\\[4pt]
1 - \sqrt{1 - s^2}, & 0 \le s \le u^*,\\[4pt]
s + 1 - \sqrt{2}, &  \text{otherwise}.
\end{cases}
\]
\Cref{fig:sol_toy_model} shows the computed solution of the toy-model after 8 levels of mesh-adaptation for the one-dimensional space case. The degenerate region $\{u=\om\}$ appears for $t>0$ and grows due to the source term. Mesh adaptation is focused on the free-boundaries and domain boundaries after it intersects the free-boundaries. Unlike previous cases, refinement happens more for later time-points since the gradients near the free-boundaries become steeper. 

\Cref{fig:adaptive_L_toy_1D} shows the decay of the estimator $\eta^i$ for the different refinement strategies of the L-scheme. Both uniform and adaptive strategies work well with a steady decay of error except for spikes when the refinement happens, due to the projection operation. In general, lower error levels are reached with fewer DoFs with the refinement strategies, with Adapt-L strategy performing marginally better, even in one space dimension.

\begin{figure}[h]
\centering

\hfill
  \begin{subfigure}[b]{0.49\textwidth}
        \centering
        \includegraphics[width=\textwidth]{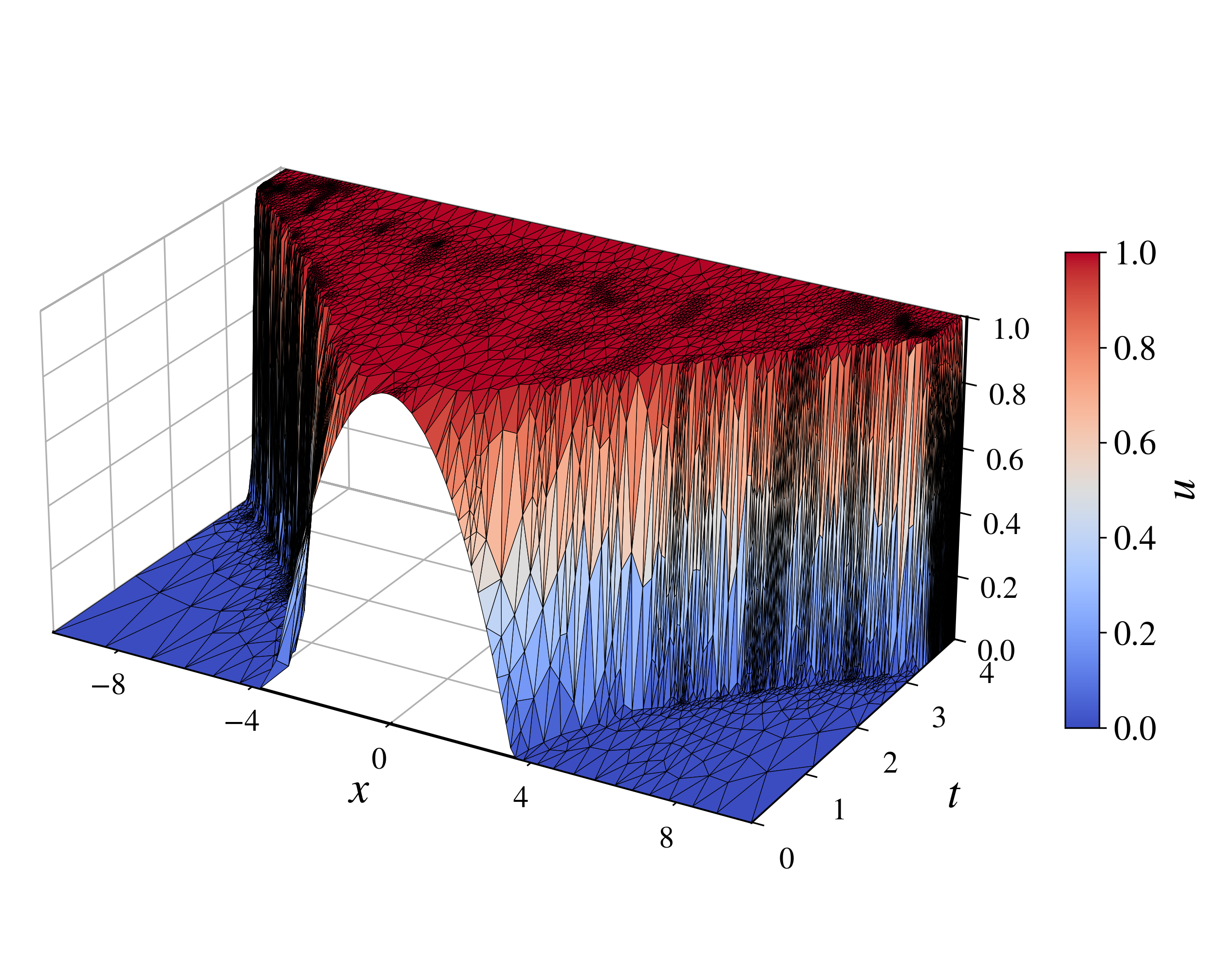}
        \label{fig:toy_u_t_4}
        \end{subfigure}       
        \hfill
        \begin{subfigure}{0.49\textwidth}
  \centering
  \includegraphics[width=\linewidth]{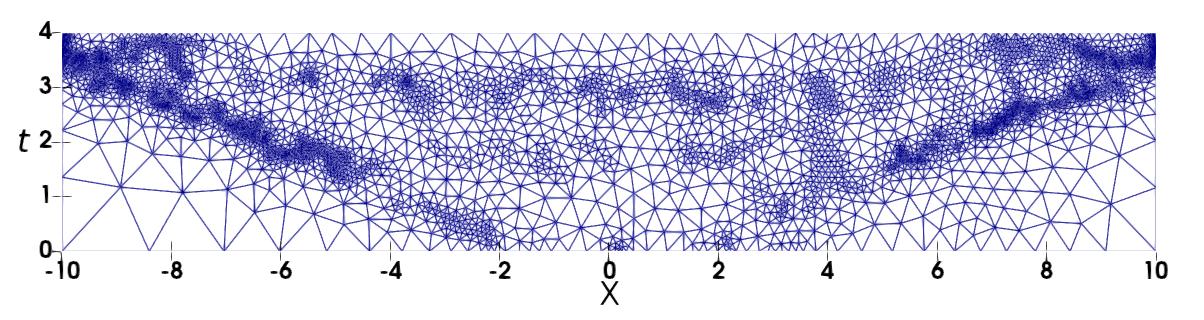}
  \vspace{4em}
\end{subfigure}
\caption{[\Cref{sec:toy}, $d=1$, $m=2$, $\gamma=1.0$, $T=4$] \textbf{Left:} The numerical solution $u = b(s)$ of the toy-model over the final adapted mesh. \textbf{Left:} Final adapted space-time mesh.}\label{fig:sol_toy_model}
\end{figure}

\begin{figure}[h]
\centering
\includegraphics[width=0.49\textwidth]
{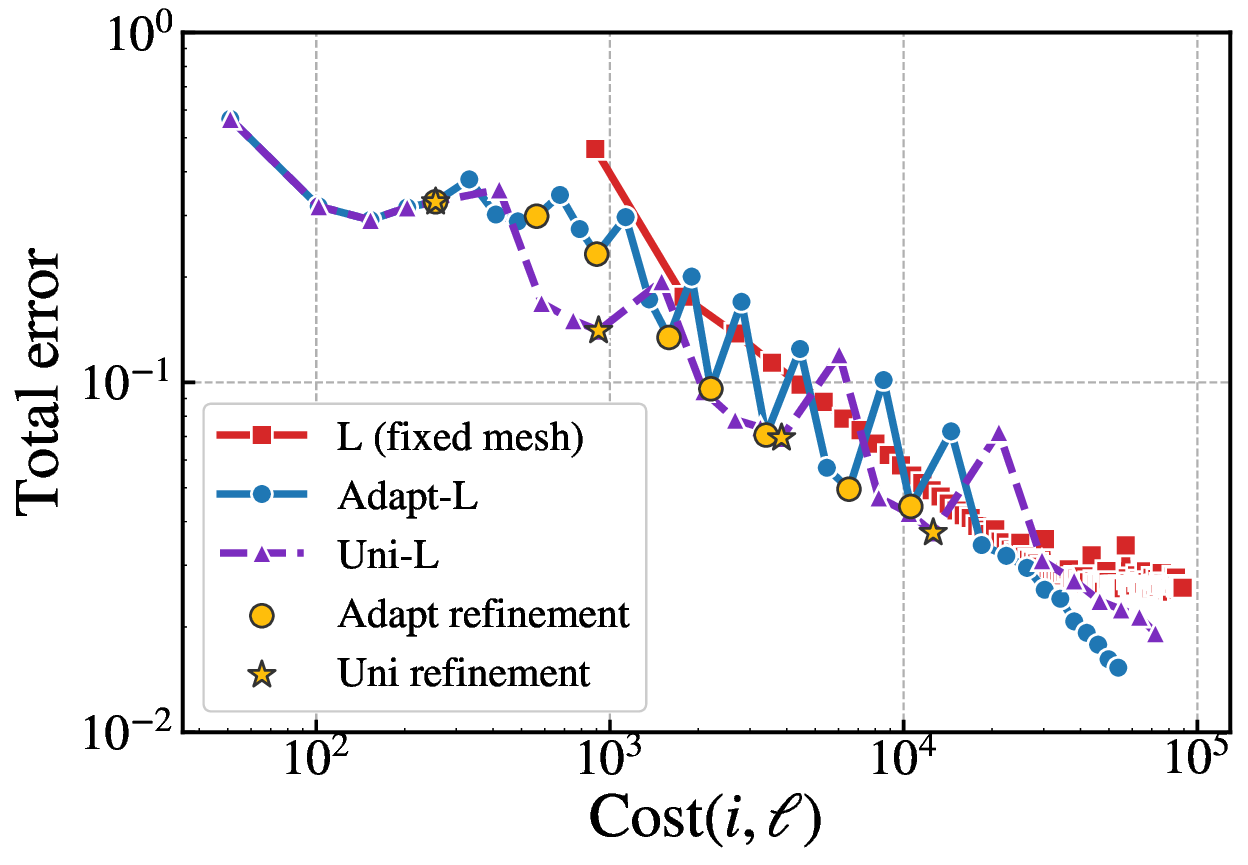}
\hfill
\includegraphics[width=0.49\textwidth]
{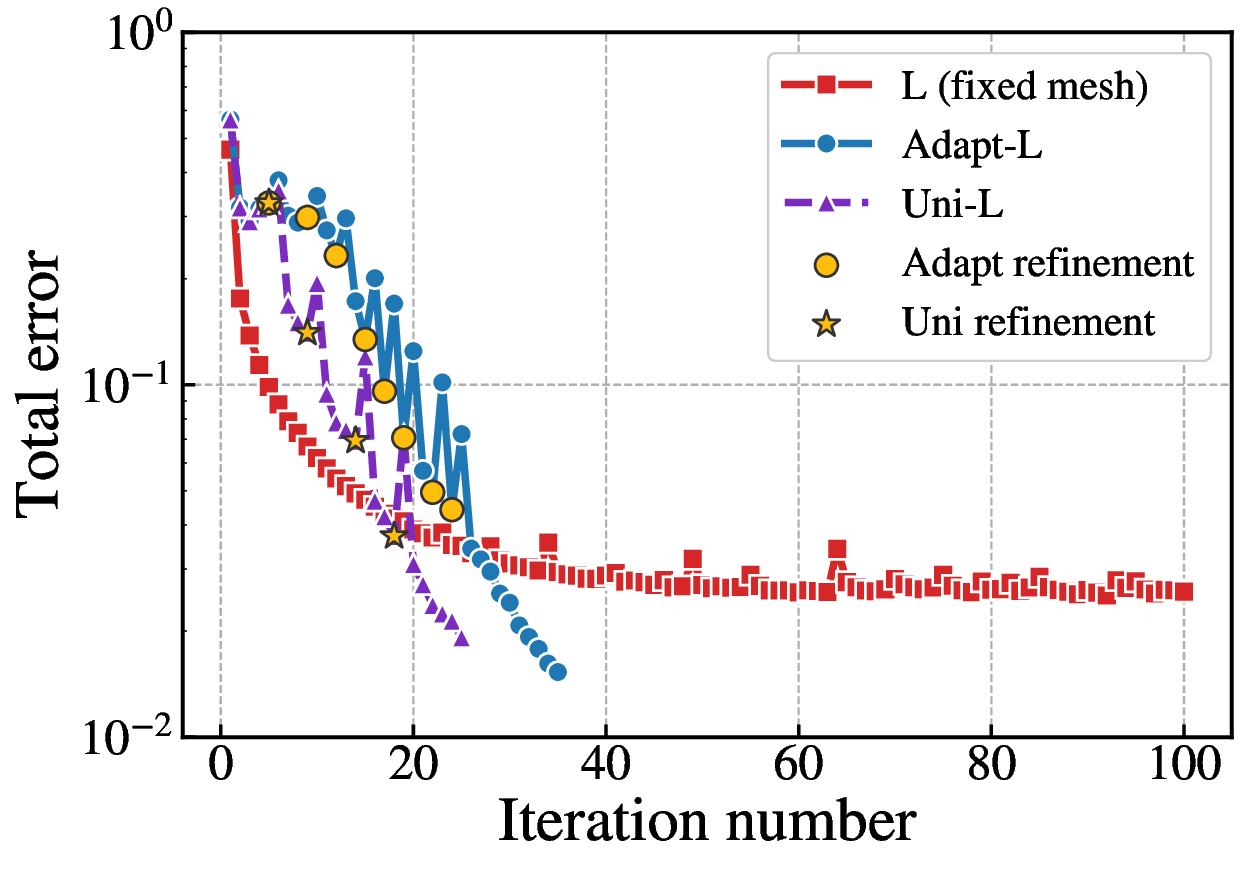}
\caption{[\Cref{tab:mesh_params_all,sec:toy}, $d=1$] Comparison of total error estimator $\eta^i = \eta^i_{\mathrm{disc}} + L\,\eta^i_{\mathrm{lin}}$ 
for fixed, uniform, and adaptive $L$-schemes. 
\textbf{Left:} $\eta^i$ versus $\mathrm{Cost}(i,\ell)$. 
\textbf{Right:} $\eta^i$ versus $\mathrm{IterNo}(i,\ell)$.}
\label{fig:adaptive_L_toy_1D}
\end{figure}

\section{Conclusion}\label{sec:conclusion}
In this work, we developed and analyzed an adaptive space-time framework for a class of nonlinear and degenerate parabolic equations. The approach is motivated by the necessity for local time-step refinement to resolve degeneracies in both the time derivative and the 
diffusion operator.

We employed a reformulation of the problem, based on \cite{javed2025robust}, that splits the complex nonlinearity into two simpler nonlinearities. A iterative linearization scheme (L-scheme) was then proposed to solve the nonlinear problem which requires solving only heat equations. It was then proven that this scheme is well-posed and unconditionally convergent irrespective of the initial guess, and even in double-degenerate regimes. The convergence is linear in $L^2$ if the problem is non-degenerate. The convergence properties were confirmed numerically, with linear decay of the linearization error 
observed in the non-degenerate and single degenerate cases. The linearization further allows the total error corresponding to an approximate solution, measured through the dual norm of the residual, to be decomposed into two components: linearization  and discretization errors.  The linearization error is measured by the difference of consecutive iterates, whereas, the discretization error corresponds to the iterative step which is simply the heat equation. 
Evaluation of these error components are performed only on fixed norms, independent of mesh, and free from any dependence on the nonlinear function. This mirrors the orthogonal decomposition of error  that was proven in \cite{mitra2023guaranteed} for elliptic problems.

Next, we considered a space-time conforming finite element method to  discretize the L-scheme, providing us with approximate numerical solutions.
A central contribution of this work is the development of guaranteed, fully- and parallely-computable a posteriori error estimators based on locally equilibrated flux reconstructions on space-time vertex patches. The quality of the estimators are independent of the nonlinearities. Moreover, they are straightforward to compute for both the L-scheme, and a modified Newton scheme that was considered for comparison. Numerical experiments show that they robustly capture the error distribution even for degenerate problems. 

Based on the a posteriori estimator, a fully adaptive solution strategy was considered where the iterations are continued until the linearization error becomes smaller than a fraction of the discretization error. Upon meeting this criteria, a refinement step is performed, based on the distribution of the discretization estimator. The effectiveness of the adaptive strategy was demonstrated through numerical experiments for the porous medium equation, the biofilm growth model, and a double degenerate toy-model. The refinements automatically concentrate along propagating free boundaries and regions of fast dynamics. The adaptive approach achieves the same overall accuracy with lower total computational costs compared to uniform refinements and fixed meshes. The adaptive scheme particularly shines in higher-dimensional spaces and for harsher nonlinearities.

\textbf{Acknowledgement} AJ acknowledge the financial support of the  HEC grant: 1(2)/HRD/OSS-III/BATCH-3/2022/HEC/384. The work of ISP was supported by the Research Foundation - Flanders (FWO), project G0A9A25N and the German Research Foundation (DFG) through the SFB 1313, project number 327154368.


\end{document}